\documentclass[final]{siamltex}
\usepackage{amssymb, amsmath}
\usepackage{float,epsfig}
\usepackage{algpseudocode}
\usepackage{mathrsfs}
\usepackage{color}
\allowdisplaybreaks[4]

\newtheorem{remark}[theorem]{ Remark}
\newtheorem{example}[theorem]{\bf Example}

\usepackage[ruled,vlined]{algorithm2e}

\def\C{{\mathbb C}}

\def\lam{\lambda}
\def\sig{\sigma}

\def\rar{\rightarrow}

\newcommand{\ba}{\begin{array}}
\newcommand{\ea}{\end{array}}

\newcommand{\vone}{\vskip 2ex}

\newcommand*{\Bigs}[1]{\scalebox{1.4}{\ensuremath#1}}

\newcommand\scalemath[2]{\scalebox{#1}{\mbox{\ensuremath{\displaystyle #2}}}}

\newcommand{\be}{\begin{equation}}
\newcommand{\ee}{\end{equation}}
\newcommand{\beano}{\begin{eqnarray*}}
\newcommand{\eeano}{\end{eqnarray*}}

\def \noin{\noindent}
\def \eig{\mathbf{eig}}
\def \nrk{\mathrm{nrank}}

\usepackage[us,12hr]{datetime}
\title{Unified framework for Fiedler-like strong linearizations of polynomial and rational matrices}
\author{Ranjan Kumar Das \thanks{Corresponding author, Department of Mathematics, Josip Juraj Strossmayer University of Osijek, Osijek-31000, Croatia ({\tt ranjankdiitg@gmail.com, rkdas@mathos.hr}). This author's research is supported by the Croatian Science Foundation, grant no. IP-2019-04-6774, Vibration Reduction in Mechanical Systems - VIMS. } \and Harish K. Pillai \thanks{Department of Electrical Engineering, IIT Bombay, Powai, India ({\tt hp@ee.iitb.ac.in})}.   }

\usepackage{lineno}

\usepackage{cite}

\begin{document}


\maketitle


\vone
\begin{abstract} Linearization is a widely used method for solving  polynomial eigenvalue problems (PEPs) and rational eigenvalue problem (REPs)  in which the PEP/REP is transformed to a generalized eigenproblem  and then solve this generalized eigenproblem with
	algorithms available in the literature. Fiedler-like pencils (Fiedler pencils (FPs), generalized Fiedler pencils (GFPs),  Fiedler pencils with repetition (FPRs) and generalized Fiedler pencils with repetition (GFPRs)) are well known classes of strong linearizations of matrix polynomials with each of them having some special properties. GFPs are an intriguing family of linearizations, and GF pencils are the fundamental building blocks of FPRs and GFPRs. As a result, FPRs and GFPRs have distinctive features and they provide structure-preserving linearizations for structured matrix polynomials. But GFPRs do not use the full potential of GF pencils. Indeed, not all the GFPs are FPRs or GFPRs, and vice versa. 
	
	The main aim of this paper is two-fold. First, to build a unified framework for all the  Fiedler-like pencils FPs, GFPs, FPRs and GFPRs.	To that end, we construct a new family of strong linearizations (named as EGFPs) of a matrix polynomial $P(\lam)$ that subsumes all the Fiedler-like linearizations. A salient feature of the EGFPs family is that it allows the construction of structured preserving banded linearizations with low bandwidth for structured (symmetric, Hermitian, palindromic) matrix polynomial. Low bandwidth structured linearizations may be useful for numerical computations. Second, to utilize EGFPs directly to form a family of Rosenbrock strong linearizations of an $n \times n$ rational matrix $G(\lam)$ associated with a realization. We describe the formulas for the construction of low bandwidth linearizations for $P(\lam)$ and $G(\lam)$. We show that the eigenvectors, minimal bases and minimal indices of $P(\lam)$ and $G(\lam)$ can be easily recovered from those of the strong linearizations of $P(\lam)$ and $G(\lam)$, respectively.
\end{abstract}

\begin{keywords} Matrix polynomial,  Rational matrices, System matrix, Fiedler pencil, Linearization, Eigenvalue, Eigenvector, Minimal basis, Minimal indices.
	
\end{keywords}

\begin{AMS}
65F15, 15A57, 15A18, 65F35
\end{AMS}

\section{Introduction} Matrix polynomials and rational matrices arise in many applications, see~\cite{aa2,aak,rafinami1,volkervoss, flpp1, planchard, bai11, voss1, voss2, Tisseur, rosenbrock70, kailath, mmmm1,vardulakis,glr,admz,amw} and the references therein. Linearization is a widely used method for computing the spectral data (poles, zeros, eigenvalues, eigenvectors, minimal bases and minimal indices) of matrix polynomials and rational matrices.

Let $P(\lam)$ be an $n \times n$ matrix polynomial of degree $m\geq 2$ given by
\begin{equation}\label{matrixpoly}
	P(\lam) = \sum_{i=0}^{m} \lam^i  A_i, ~~ A_0, \ldots, A_m \in \mathbb{C}^{n \times n}, ~ A_m \neq 0.
\end{equation}
 Then  an $mn\times mn$ matrix pencil $L(\lam) := \mathcal{X} + \lam \mathcal{Y}$ is  said to be a linearization~\cite{glr, mmmm1} of $P(\lam)$ if there exist $mn\times mn$  unimodular matrix polynomials $U(\lam)$ and $V(\lam)$ (i.e., $\det(U(\lam))$ and $\det(V(\lam))$ are  nonzero constants independent of $\lam$)  such that $ U(\lam) L(\lam) V(\lam) = \diag(I_{(m-1)n},\; P(\lam))$ for $ \lam \in \C.$ Additionally,  if $ rev L(\lam) := \mathcal{Y} + \lam \mathcal{X} $ is a linearization of $ rev P(\lam)$ then   $L(\lam)$ is said to be a strong linearization~\cite{mmmm1} of $P(\lam)$, where $revP(\lam) := \lam^m P( 1/\lam).$ Inspired by the seminal work of Miroslav Fiedler for scalar polynomial \cite{fiedler}, Fiedler pencils (FPs) and generalized Fiedler pencils (GFPs) of a polynomial matrix $P(\lam)$ are studied in \cite{AV, BDD} which are strong linearizations of $P(\lam)$. Matrix polynomial having structure such as symmetric, skew-symmetric, palindromic, even, odd, etc., arise in many applications, and the structure often induces spectral symmetry in the eigenvalues. The spectral symmetry in the eigenvalues has physical significance and hence it is important to consider a linearization that preserves the spectral symmetry, see~\cite{alamvol,amw,mmmm1,mmmm,mms,vw1,vw2,aa2,MFT,bfs2,Teran_DM11,blkstruc,hh,hmmt,bfs1, bdf, rafiran4, tdmeqv, Huang,ksw,bkms2,bkms1, thpL} and the references therein. To overcome the limitation of FPs and GFPs in order to construct structure preserving linearizations, FPRs and GFPRs are contructed in \cite{VAEN} and \cite{Bueno_DFR15}, respectively.  Note that the basic building blocks of FPRs and GFPRs are a very special type of GFPs and so FPRs and GFPRs do not utilize the potential of GFPs in fullest. Indeed, not all  the GFPs are FPRs or GFPRs, and vice versa. For example, consider the GFP $T(\lam) = \lam M^P_{(-4,-3,-1)} - M^P_{(0,2)}$ and the FPR $L(\lam) = (\lam M^P_{-4,-3} - M^P_{(0,1,2)})M^P_1$ of $P(\lam) := \sum_{i=0}^4 \lam^i A_i$ (see Section~\ref{basicdef} for the definition of $M_j^P$). Then $T(\lam)$ is not an FPR and $L(\lam)$ is not a GFP. Further, consider the pencil $F(\lam) : =  (\lam  M^P_{(-4, -1)} - M^P_{(2,3,0)}) M^P_{2} $. Then $F(\lam)$ is neither  a GFP nor an FPR of $P(\lam)$. This motivates us to present a unified framework that subsumes all the Fiedler-like linearizations, i.e., FPs, GFPs, FPRs, and GFPRs of $P(\lam)$. We named this family as Extended generalized Fiedler pencils (EGFPs) of $P(\lam)$. Indeed, EGFPs is a huge class of linearizations that expand the arena in which to look for the structured preserving linearizations for structured $P(\lam)$. Also, it allows us to construct linearizations with special property like banded linearizations with low bandwidth.
 
  We mention that there does not exist any block penta-diagonal (resp., tridiagonal) symmetric GFPR for symmetric $P(\lam)$ having degree $m \geq 8$ (resp., $m \geq 4$) \cite{RanThesis}. But, the family of EGFPs has no such shortcoming and enables us to construct symmetric block penta-diagonal EGFPs for symmetric $P(\lam)$. For $m=8,$ the pencil $$L(\lam)=\left[\begin{array}{cccccccc} -A_{8} & 0 &\lam  A_{8}& 0 & 0 & 0 & 0 & 0\\ 0 & 0 & -I_n & \lam I_n& 0 & 0 & 0 & 0\\\lam A_{8} & -I_n & \lam A_{7}-A_{6} &\lam A_{6} & 0 & 0 & 0 & 0\\ 0 & \lam I_n & \lam A_{6} & \lam A_{5}+A_{4} & -I_n & 0 & 0 & 0\\ 0 & 0 & 0 & -I_n & 0 & \lam I_n & 0 & 0\\ 0 & 0 & 0 & 0 & \lam I_n & \lam A_{3}+A_{2} & A_{1} & -X\\ 0 & 0 & 0 & 0 & 0 & A_{1} & - \lam A_{1} + A_{0} & \lam X \\ 0 & 0 & 0 & 0 & 0 & -X & \lam  X & 0 \end{array}\right]$$
 	is  a block symmetric and block penta-diagonal EGFP, where $X \in \mathbb{C}^{n\times n}$. Moreover, if $X$ is symmetric then   $ L(\lam) $  is symmetric when  $P(\lam)$ is symmetric. We mention that $L(\lam)$ is not a GFPR of $P(\lam)$ \cite{RanThesis}. For palindromic $P(\lam)$, a framework of structured-preserving linearizations having specific block-bandwidth is studied in \cite{rafiran5} by using the EFPRs of $P(\lam)$. Further, in \cite{rafiran5}, it is discussed how the EGFPs overcome the limitation of the family of GFs and FPRs in order to construct palindromic linearizations with specific bandwidth, see \cite{rafiran5} for details. 
 

Linearization of rational matrices is a relatively new concept and has been studied extensively in past few years by following the state-space model for rational matrices pioneered by Rosenbrock~\cite{rosenbrock70}, see~\cite{rafinami1, rafinami3, admz, rafiran4, rafiran1, bai11,RanThesis, rafiran2, fmqv1,fqp3,fmq} and the reference therein. Consider a realization of an $n \times n$ rational matrix $G(\lam)$ of the form $ G(\lambda)=  P(\lambda) + C(\lambda E -A )^{-1} B$, where $A- \lambda E$ is an $r \times r$ pencil with $E$ being nonsingular, $P(\lam) = \sum_{i=0}^{m} \lam^i A_i$ is an $n\times n$ matrix polynomial of degree $m$, $C\in\mathbb{C}^{n\times r}$ and $B\in\mathbb{C}^{r\times n}.$ Fiedler pencils (FPs) and generalized Fiedler pencils (GFPs) and generalized Fiedler pencils with repetitions (GFPRs) have been constructed in \cite{rafinami1, rafinami3, rafiran4} which are shown to be Rosenbrock strong linearizations (see Definition~\ref{stln1}) of $G(\lam)$ in \cite{rafiran4}. Note that, neither GFPs of $G(\lam)$ is a subset of GFPRs, and vice versa. Further, the construction of FPs, GFPs, and GFPRs of $G(\lam)$ is described by defining Fiedler matrices for $G(\lam)$ separately, that is, the Fiedler matrices of  $G(\lam)$ are defined by invoking the Fiedler matrices of $P(\lam)$ in one block, followed by positioning the constant matrices $B,C,A$ and $E$ in the appropriate positions of the other blocks. This force to state a lot of notations and definitions for $G(\lam)$ parallel to the existing notations and definitions for $P(\lam)$.  We address these issues in this paper. More precisely, the main contributions of the present paper are the following.

We build an unified framework (named as EGFPs) that subsumes all the Fiedler-like linearizations FPs, GFPs, FPRs and GFRs for matrix polynomials. We show by some examples that EGFPs is a rich source to look into the structured preserving linearizations for structured (symmetric, skew-symmetric, even, odd, palindromic, etc.,) matrix polynomials. Then we characterize block tridiagonal and block penta-diagonal EGFPs of $P(\lam)$. This characterization also works for FPs, GFPs, FPRs, and GFPRs of $P(\lam)$. By utilizing the EGFPs of $P(\lam)$, we construct a family of Rosenbrock strong linearizations of $G(\lam)$ directly. This construction is simple and it does not need to define Fiedler matrices for $G(\lam)$ separately (as it is done in case of FPs, GFPs and GFPRs in \cite{rafinami1, rafinami3,rafiran4}). Further, we describe how to construct block tridiagonal and block penta-diagonal Rosenbrock strong linearization of $G(\lam)$. Finally, we describe easy recovery of eigenvectors, minimal bases and minimal indices of $P(\lam)$  and $G(\lam)$ from those of the linearizations.


The rest of the paper is organized as follows.  We collect some basic results in Section~\ref{sec2}. In Section~\ref{EGFPsOfP}, we introduce EGFPs of $P(\lam)$, and give a characterization for block tridiagonal and block penta-diagonal EGFPs of $P(\lam)$.  Construction of Rosenbrock strong linearizations for rational matrices from those of EGFPs of matrix polynomial is described in Section~\ref{EGFPofG}. In Section~\ref{recoAll}, we describe the recovery of eigenvectors, minimal bases and minimal indices of $P(\lam)$ and $G(\lam)$ from those of the linearizations. 

\textbf{Notation.} We denote  by $\C[\lam]$  the ring (over $\C$) of scalar polynomials and  by  $\C(\lam)$ the field of rational functions of the form  $p(\lambda)/ q(\lambda),$ where $p(\lambda)$ and $q(\lambda)$ are  polynomials in $\mathbb{C}[\lambda].$  We denote by $\mathbb{C}[\lambda]^{m\times n}$ (resp., $\mathbb{C}(\lambda)^{m\times n}$) the vector space of $m\times n$ matrix polynomials (resp., rational matrices) over $\C$ (resp., over $\C(\lam)$). The spaces $\C[\lam]^m$ and $\C(\lam)^{m},$ respectively, denote $\mathbb{C}[\lambda]^{m\times n}$ and  $\mathbb{C}(\lambda)^{m\times n}$ when $n=1.$ We denote the \textit{j}-th column of the $n\times n $ identity matrix $I_n$ by $e_j$ and the transpose (resp., conjugate transpose) of an $m\times n$  matrix $A$ by $A^T$ (resp., $A^*$). The right and left null spaces of $A$ are given by $ \mathcal{N}_r(A) := \{ x \in \C^n : Ax = 0\}$ and $\mathcal{N}_l(A) := \{ x \in \C^m : x^T A = 0 \},$ respectively. We denote by $A\otimes B$ the Kronecker product of the matrices $A$ and $B.$  

\section{Preliminaries}\label{sec2} 

Let $ G(\lam) \in \C(\lam)^{m\times n}$. The rank of  $G(\lam)$ over the field $\C(\lam)$ is called the \textit{normal rank} of $G(\lam)$ and is denoted by $\nrk(G).$  If $\nrk(G) = n =m$ then $G(\lam)$  is said to be \textit{regular}, otherwise $G(\lam)$ is said to be \textit{singular}. A complex number $\mu \in \C$ is said to be an \textit{eigenvalue} of $G(\lambda)$ if $\rank(G(\mu)) < \nrk(G).$  We denote the set of eigenvalues of $G$ by $\eig(G).$ Let
\begin{equation} \label{SMform}
	D(\lam) := \diag\left( \frac{\phi_{1}(\lam)}{\psi_{1}(\lam)}, \ldots, \frac{\phi_{k}(\lam)}{\psi_{k}(\lam)}, 0_{m-k, n-k}\right)
\end{equation}
be the Smith-McMillan form~\cite{kailath, rosenbrock70} of  $G(\lam),$ where $ k:= \nrk(G)$ and the scalar polynomials $\phi_{i}(\lam)$ and $ \psi_{i}(\lam)$ are monic and pairwise coprime and  that $\phi_{i}(\lam)$ divides $\phi_{i+1}(\lam)$ and $\psi_{i+1}(\lam)$ divides $\psi_{i}(\lam),$ for $i= 1, 2, \ldots, k-1$. Note that, if $G(\lam)$ is a matrix polynomial, then $\psi_{i}(\lam)=1,$ for $i=1,2,\ldots,k$, and $D(\lam)$ is called the Smith normal form of $G(\lam)$. Set $\phi_{G}(\lam) := \prod _{j=1}^{k} \phi_{j}(\lam) \,\,\, \mbox{ and } \,\,\, \psi_{G}(\lam) := \prod _{j=1}^{k} \psi_{j}(\lam).$ Then $ \mu\in \C$ is a pole of $G(\lam) $ if $\psi_G(\mu) =0.$ A complex number $ \mu$ is said to be a {\em zero} of $G(\lam)$ if $ \phi_G(\mu) =0.$ The {\em spectrum} of $G(\lam)$ is given by $ \mathrm{Sp}(G) := \{ \lam \in \C : \phi_G(\lam) = 0\}$ and consists of the finite zeros of $G(\lam).$
Note that $ \eig(G) \subset \mathrm{Sp}(G).$ See~\cite{rafinami1, kailath, admz} for more on eigenvalues and zeros of $G(\lam).$

 Let $ G(\lam) \in \C(\lam)^{n\times n}$. We consider a realization of $G(\lam)$ of the form
 \begin{equation}\label{minrel2_RCh41}
 	G(\lambda)= \sum\nolimits^m_{j=0} A_j \lam^j+C(\lambda E-A)^{-1}B =:   P(\lambda)+C(\lambda E-A)^{-1}B,
 \end{equation}
 where $ \lam E- A$ is an $r\times r$ matrix pencil with $E$ being nonsingular,  $C\in\mathbb{C}^{n\times r}$ and $B\in\mathbb{C}^{r\times n}.$
 The realization (\ref{minrel2_RCh41}) is said to be {\em minimal} if the size of the pencil $\lam E- A$ is the smallest among all the realizations of $G(\lambda),$ see~\cite{kailath}. The matrix polynomial
 \begin{equation} \label{slamsystemmatrix_RCh41}
 	{\mathcal{S}}(\lambda ) : = \left[
 	\begin{array}{c|c}
 		P(\lambda) & C \\
 		\hline
 		B & A-\lambda E
 	\end{array}
 	\right]
 \end{equation} is called the  {\em{system matrix}} (or the Rosenbrock system matrix) of $G(\lam)$ associated with the realization (\ref{minrel2_RCh41}). The system matrix $\mathcal{S}(\lam)$ is said to be {\em irreducible } if the realization (\ref{minrel2_RCh41}) is minimal. The system matrix  $\mathcal{S}(\lam)$ is irreducible  if and only if $\rank\Big(\left[\begin{array}{cc} B & A - \lam E\end{array}\right]\Big) = r = \rank\Big(\left[\begin{array}{cc} C^T & (A - \lam E)^T\end{array}\right]^T\Big), $ see~\cite{kailath, rosenbrock70}. Observe that  $\eig(G) \subset \eig(\mathcal{S})$ and we have  $ \eig(\mathcal{S}) = \mathrm{Sp}(G)$ when $\mathcal{S}(\lam)$ is irreducible, see~\cite{rafinami1, rosenbrock70}.

 An $n\times n$ matrix polynomial  $U(\lam)$ is said to be {\em unimodular} if $ \det(U(\lam))$ is a nonzero constant independent of $\lam.$ A rational matrix $G(\lam)$ is said to be {\em proper } if $ G(\lam) \rightarrow D$ as $ \lam \rar \infty,$ where $D$ is a matrix. An $n\times n$ rational matrix $F(\lam)$ is said to be {\em  biproper} if $F(\lam)$ is proper and $F(\infty)$ is a nonsingular matrix~\cite{vardulakis}.

 \begin{definition}[\cite{rafiran1}] \label{stln1} Let  $\mathbb{L}(\lam)$ be an $(mn+r)\times (mn+r)$ irreducible system matrix of the form
 	\be\label{lbb1} \mathbb{L}(\lam) := \left[
 	\begin{array}{c|c}
 		\mathcal{X}- \lam \mathcal{Y} & \mathcal{C} \\
 		\hline
 		\mathcal{B} &  H - \lam K \\
 	\end{array}
 	\right], \ee where $ H- \lam K$ is an $r\times r$  pencil with $K$ being nonsingular. Then  $\mathbb{L}(\lam)$ is said to be a Rosenbrock strong linearization of $G(\lam)$ if the following                                       conditions hold.
 	
 	\begin{itemize}
 		\item[(a)] There exist $mn\times mn$ unimodular matrix polynomials $U(\lam)$ and $V(\lam)$, and $r\times r$ nonsingular matrices $U_0$ and $V_0$  such that
 		\be
 		\left[
 		\begin{array}{c|c}
 			U(\lam) & 0 \\
 			\hline
 			0 & U_0 \\
 		\end{array}
 		\right] \mathbb{L} (\lam) \left[
 		\begin{array}{c|c}
 			V(\lam) & 0 \\
 			\hline
 			0 & V_0 \\
 		\end{array}
 		\right] = \left[
 		\begin{array}{c|c}
 			I_{(m-1)n} & 0 \\
 			\hline
 			0 &  \mathcal{S}(\lam) \\
 		\end{array}
 		\right]. \nonumber \ee
 		
 		\item[(b)]	There exist $mn\times mn$  biproper rational  matrices $\mathcal{O}_{\ell} (\lam)  $ and $\mathcal{O}_r (\lam) $ such that
 		\begin{equation}
 			\mathcal{O}_{\ell} (\lam)  \, \lam^{-1} \mathbb{G} (\lam)  \, \mathcal{O}_r (\lam) =
 			\left[ \begin{array}{c|c}
 				I_{(m-1)n} & 0 \\
 				\hline
 				0 & \lam^{-m} G(\lam) \\
 			\end{array}
 			\right], \nonumber
 		\end{equation}
 		where $\mathbb{G}(\lam) :=\mathcal{X}- \lam \mathcal{Y} + \mathcal{ C} (\lam K -H)^{-1} \mathcal{ B}  $ is the transfer function of $ \mathbb{L}(\lam)$.
 	\end{itemize}
 	The pencil $\mathbb{L}(\lam)$ is also referred to as a Rosenbrock strong linearization of $\mathcal{S}(\lam).$
 	
 \end{definition}
 
 We refer to \cite{rafiran1}  for more on Rosenbrock strong linearizations of $G(\lam)$ and the relation between the  structural indices of (finite and infinite) zeros and poles of $G(\lam)$ and $\mathbb{L}(\lam).$
 Suffice it to say that the condition (a) ensures~(see, \cite[Theorem~3.4]{rafinami3}) that $ U(\lam) \mathbb{G}(\lam)V(\lam)  = \diag( I_{(m-1)n}, \; G(\lam))$ which in turn ensures that $G(\lam)$ and $\mathbb{G}(\lam)$ have the same finite zeros and poles. The irreducibility of $\mathbb{L}(\lam)$ guarantees that the finite zeros and poles of $\mathbb{G}(\lam)$ are the same as the finite eigenvalues of $\mathbb{L}(\lam)$ and  $H - \lam K,$ respectively; see~\cite{kailath, rafiran1}. On the other hand, the condition (b) ensures that the structural indices of zeros and poles of $G(\lam)$ at infinity can be recovered from the structural indices of eigenvalues and poles of $\mathbb{L}(\lam)$ at infinity (see~\cite{rafiran1}). See~\cite{admz} for an alternative definition of strong linearization of rational matrices and the follow-up papers \cite{fmqv1, fqp3, fmq} for more on linearization. Thus the zeros and poles of $G(\lam)$ including their structural indices can be obtained by solving the eigenvalue problems $\mathbb{L}(\lam) v = 0$ and $(H-\lam K)u=0$; see~\cite{rafinami1,rafinami3,rafiran1,rafiran4}.

\subsection{Index tuples} \label{basicdef}

Index tuples play a very important role in defining and analyzing Fiedler-like pencils. For a ready reference, we now define index tuples and matrix assignments. 

For $k,\ell \in \mathbb{Z},$ we use the notation $(k:\ell):= (k,k+1,\ldots, \ell) ~ \mbox{if}~ k \leq \ell,$ and $(k:\ell):= \emptyset ~\mbox{if}~ k > \ell.$ We define $(\infty:\ell):= \emptyset$ for any  integer $ \ell$.


%


\begin{definition} [Index tuple] A tuple $\sig : =(j_1 , j_2 ,\ldots , j_p) \in \mathbb{Z}^p$ is said to be an \emph{index tuple} containing indices from $\mathbb{Z}.$  We define $- \sigma : = ( -j_1, -j_2, \ldots , -j_p)$, $rev(\sigma):=( j_p , j_{p-1}, \ldots, j_2, j_1)$ and $\sigma + q:= ( j_1 +q ,  j_2 +q, \ldots , j_p +q)$ for  $q \in \mathbb{Z}. $ Further, for any index tuples $\alpha_1 := ( i_1  , \ldots , i_p )$ and $\alpha_2 := ( j_1  , \ldots , j_q)$, we define  $\alpha_1 \cup \alpha_2 := (\alpha_1 , \alpha_2) = ( i_1  , \ldots , i_p , j_1  , \ldots , j_q)$.
\end{definition}

\begin{definition}  Let $\alpha$ and $\beta$ be index tuples. Then $\alpha$ is said to be a \emph{subtuple} of $\beta$ if $\alpha = \beta$ or if $\alpha$ can be obtained from $\beta$ by deleting some indices in $\beta.$
\end{definition}

Consecutions and inversions of an index tuple are easy to derive and it will be used frequently in this paper.

\begin{definition}  [Consecutions and inversions, \cite{rafiran3}] \label{coninvoftuple} Let $\alpha$ be an index tuple containing indices  from  $\{ 0 : m\}$ (resp., $-\{0:m\}$). Suppose that $t \in \alpha$. Then we say that $\alpha$ has $p$ \emph{consecutions} at $t$ if $ ( t , t+1 , \ldots, t+p)$ is a subtuple of $\alpha$ but $ ( t , t+1 , \ldots, t+p, t+p+1)$ is not a subtuple of $\alpha$. We denote the number of consecutions of $\alpha$ at $t$ by $ c_{t}(\alpha)$. If $ t \notin \alpha$ then we define $ c_{t}(\alpha) :=-1.$ We say that $\alpha$ has $q$  \emph{inversions} at $t$ if $ ( t+q , t+q-1 , \ldots, t)$ is a subtuple of $\alpha$ but $ ( t+q+1 , t+q , \ldots, t)$ is not a subtuple of $\alpha$. We denote the number of inversions of $\alpha$ at $t$ by $ i_{t}(\alpha)$. If $ t \notin \alpha$ then we define $ i_{t}(\alpha) :=-1.$
\end{definition}

\begin{example} \label{exampleofnewconse_RCh4} Let $\alpha := ( 1,0,2,1,3,2,4,1,3,2,1)$ be an index tuple containing indices from $\{0:6\}$.  Then $ c_0(\alpha) = 3$ as $(0,1,2,3)$ is a subtuple of $\alpha$ and $(0,1,2,3,4)$ is not a subtuple of $\alpha$. Similarly, $ c_3(\alpha) = 1$ as $(3,4)$ is a subtuple of $\alpha$ and $(3,4,5)$ is not a subtuple of $\alpha$. Further, $ i_0(\alpha) = 1$ as $(1,0)$ is a subtuple of $\alpha$ and $(2,1,0)$ is not a subtuple of $\alpha$. Similarly, observe that $i_1(\alpha) = 3$ and $ i_3(\alpha) = 1$. As $5 \notin \alpha$ we have  $c_5(\alpha) = -1 $ and $i_5(\alpha) =-1.$
\end{example}

For   $X \in \mathbb{C}^{n \times n}$, we consider the elementary matrices  given by~\cite{Bueno_DFR15}
$$ M_{0} (X) := \left[
\begin{array}{@{}cc@{}}
	I_{(m-1)n} &  \\
	& X \\
\end{array}
\right],~M_{i} (X) := \scalemath{.98}{ \left[
	\begin{array}{@{}cccc@{}}
		I_{(m-i-1)n} &  & & \\
		& X & I_{n} & \\
		& I_{n} & 0  & \\
		&   &   & I_{(i-1)n}\\
	\end{array}
	\right]} , \, i=1:m-1, $$
$$ M_{-m}(X) := \left[
\begin{array}{@{}cc@{}}
	X &  \\
	& I_{(m-1)n} \\
\end{array}
\right] \text{ and } M_{-i}(X) := \scalemath{.9}{\left[
	\begin{array}{@{}cccc@{}}
		I_{(m-i-1)n} &  & & \\
		& 0 & I_{n} & \\
		& I_{n} & X  &  \\
		&   &   & I_{(i-1)n}\\
	\end{array}
	\right]}  \, i=1:m-1. $$
For $i=1:m-1$, the matrices $M_i(X)$ and $M_{-i}(X)$ are invertible for any $X$ and  $(M_i(X))^{-1} = M_{-i}(-X).$  On the other hand, the matrices $M_{0} (X)$ and $M_{-m} (X)$ are invertible if and only if $X$ is invertible. Note that $M_{i} (X) M_{j} (Y) = M_{j} (Y) M_{i} (X)$ holds for any matrices $X, Y \in \mathbb{C}^{n \times n}$ if $\scalemath{1.2}{|}|i|-|j|\scalemath{1.2}{|} > 1$.

For the rest of the paper, we assume that $P(\lambda) := \sum_{j=0}^m \lambda^j A_j.$ For $i \in \{-m:m-1\}$, we define \cite{Bueno_DFR15}
$$ M_i^P  : =
\left\{ \begin{array}{ll}
	M_i(-A_i), & \mbox{if}~ i \geq 0\\
	M_i(A_{-i}), & \mbox{if}~ i <  0.
\end{array}
\right.$$
Then  $ M_i^P $  are the Fiedler companion matrices  associated with $P(\lam)$ and are given by \cite{AV,tdm, Bueno_DFR15}
$$ M_{0}^P = \left[
\begin{array}{@{}cc@{}}
	I_{(m-1)n} &  \\
	& -A_0 \\
\end{array}
\right], \, M_{i}^P = \left[
\begin{array}{cccc}
	I_{(m-i-1)n} &  & & \\
	& -A_i & I_{n} & \\
	& I_{n} & 0  & \\
	&   &   & I_{(i-1)n}\\
\end{array}
\right] ~\mbox{for}~ i= 1 : m-1,$$
$$~M_{-m}^P = \left[
\begin{array}{@{}cc@{}}
	A_m &  \\
	& I_{(m-1)n} \\
\end{array}
\right], \,  M^P_{-i}= \left[
\begin{array}{cccc}
	I_{(m-i-1)n} &  & & \\
	& 0 & I_{n} & \\
	& I_{n} & A_i &  \\
	&   &   & I_{(i-1)n}\\
\end{array}
\right] \text{ for }  i= 1 : m-1.$$ Note that $M^P_{0}$ (resp., $M^P_{-m} $ ) is invertible if and only if $A_0$ (resp., $A_m$) is invertible. We define $M^P_{-0} := (M^P_{0})^{-1}$ and $M_{m}^P := (M^P_{-m})^{-1}. $

\begin{definition} [Matrix assignments, \cite{Bueno_DFR15}] Let $\textbf{t} := (t_1 ,t_2, \ldots , t_r)$ be a nonempty index tuple containing indices from $\{ \pm 0, \pm 1, \ldots, \pm m \}$ and $X:=( X_1 ,X_2, \ldots , X_r)$ be a tuple of $n \times n$ matrices. We define
	$ M_{\textbf{t}} (X) : = M_{t_1} (X_1) M_{t_2} (X_2) \cdots M_{t_r} (X_r)$ and say that $X$ is a \emph{matrix assignment} for $\textbf{t}$. Further, we say that the matrix $X_j$ is assigned to the index $t_j$ in $\textbf{t}$. The matrix assignment $X$ for $\textbf{t}$ is said to be nonsingular if the matrices assigned by $X$ to the positions in $\textbf{t}$ occupied by the $ \pm 0$ and $\pm m$ indices are nonsingular. Further, we define $rev(X):=(X_r, \ldots, X_1).$
\end{definition}

For an index tuple $\textbf{t}$,  $M_{\textbf{t}} (X)$ is invertible if $X$ is an invertible matrix assignment. We define $M_{\textbf{t}} (X) : =  I_{mn}$ if $\textbf{t}$ is an empty set. Further, if $X^1 , \ldots , X^s$ are matrix assignments for the index tuples $\textbf{t}_1, \ldots, \textbf{t}_s$, respectively, then we define \cite{Bueno_DFR15}
$$
M_{(\textbf{t}_1 , \ldots, \textbf{t}_s)} (X^1 , \ldots, X^s) := M_{\textbf{t}_1} (X^1) \cdots M_{\textbf{t}_s} (X^s).
$$
We define $M_{\textbf{t}}^P : = M^P_{t_1} M^P_{t_2} \cdots M^P_{t_r}$ if $\textbf{t} := (t_1 , t_2,\ldots, t_r)$ is a nonempty index tuple.



%

\begin{definition}  \cite{Bueno_DFR15} Let $\textbf{t} := (t_1 ,t_2, \ldots , t_r)$ be an index tuple containing indices from $\{ \pm 0, \pm 1, \ldots, \pm m \}$. Then $X:=( X_1 ,X_2, \ldots , X_r)$ is said to be the \emph{trivial matrix assignment} associated with $P(\lam)$ if  $M_{t_j} (X_j) = M_{t_j}^P$ for $ j=1:r$.
\end{definition}

Thus, if $X$ is the trivial matrix assignment for $\textbf{t}$ associated with $P(\lambda)$, then $M_{\textbf{t}} (X) = M_{\textbf{t}}^P$.

With a view to providing a unique representation of equivalent index tuples, we now define SIP, rsf, and csf of an index tuple, which we will use extensively.

\begin{definition} \cite{VAEN, MFT, Bueno_DFR15} Let $\tau:= (j_1, j_2, \ldots , j_q)$ be an index tuple containing indices from $\mathbb{Z}$ and $\sigma := (i_1, i_2, \ldots , i_t)$ be an index tuple containing indices from $\{0: h\}$ for some non-negative integer $h$. Then:
	\begin{enumerate}
		\item[(a)] $j_p$ is said to be a \emph{simple index} of $\tau$ if $j_p \neq j_k$ for $k= 1: q$ and $k \neq p.$ We say that $\tau$ is \emph{simple} if each index $j_p $ is a simple index for $p =1: q$.
		
		\item[(b)] $\sigma$ is said to satisfy the \emph{Successor Infix Property (SIP)} if for every pair of indices $i_a, i_b \in \sigma$ with $1 \leq a < b \leq t$ satisfying $i_a = i_b,$ there exists at least one index $i_c = i_a +1$ such that $a < c <b.$ Set $\alpha := -\sigma$. Then $\alpha$ is said to satisfy the SIP if $\alpha + h $ satisfies the SIP.
		
		\item[(c)]$\sigma$ is said to be in  \emph{column standard form} if
		$$\sigma= ( a_s: b_s, a_{s-1}: b_{s-1}, \ldots , a_2:b_2, a_1 : b_1),$$
		for some $0 \leq b_1 < \cdots < b_{s-1} < b_s \leq h $ and $0 \leq a_j \leq b_j ,$ for all $j =1 , \ldots ,s. $ Set $\beta := -\sigma$. Then $\beta$ is said to be in column standard form if $\beta + h$ is in column standard form. 
		
		\item[(d)]$\sigma$ is said to be in  \emph{row standard form} if
		$$\sigma= ( rev(a_1: b_1), rev(a_{2}: b_{2}), \ldots , rev(a_{s-1}:b_{s-1}), rev(a_s : b_s)),$$
		for some $0 \leq b_1 < \cdots < b_{s-1} < b_s \leq h $ and $0 \leq a_j \leq b_j ,$ for all $j =1 , \ldots ,s. $ Set $\beta := -\sigma$. Then $\beta$ is  said to be in row standard form if $\beta + h$ is in row standard form.
	\end{enumerate}
\end{definition}

Let $\alpha$ be an index containing indices from $\{0:m\}$ (resp., $-\{0:m\}$) such that $\alpha$ satisfies the SIP. Then the positions of the block entries of  $M_{\alpha} (\mathcal{X})$ do not depend upon the particular matrix assignment $\mathcal{X}$, that is, the positions of the block entries of  $M_{\alpha} (\mathcal{X})$ depend only on $\alpha$, see \cite{Bueno_DFR15}.

\begin{definition}\cite{VAEN, Bueno_DFR15} Let $\alpha$ and $\beta$ be two index tuples containing indices from $\{0:m\}$ (resp., $-\{0:m\}$). Then $\alpha$ is said to be \emph{equivalent} to $\beta$ (written as $\alpha \sim \beta$) if $M^P_\alpha = M^P_\beta$.
\end{definition}

It is shown in~\cite{VAEN, MFT, Bueno_DFR15} that an index tuple satisfying the SIP is equivalent to a unique index  tuple in column/row standard form, see~\cite{VAEN, MFT, Bueno_DFR15} for more details.

	\begin{definition}[csf, rsf, \cite{VAEN, MFT}] Let $\sig$ be an index tuple containing indices from $\{0:m\}$ (resp., $-\{0:m\}$) such that $\sig$ satisfies the SIP. Then the unique tuple in column standard form which is equivalent to $\sig$ is called the column standard form of $\sig$ and is denoted by $csf(\sigma).$ Similarly, the unique tuple in row standard form which is equivalent to $\sig$ is called the row standard form of $\sig$ and is denoted by $rsf(\sigma).$
\end{definition}

The following results will be used frequently in this paper. 
 \begin{proposition}\cite{rafiran3, rafiran4, RanThesis} \label{lemeqnconrsfcsf16a18n730_PCh3} Let $\alpha$ be an index tuple containing indices from $ \{ 0:m \}$ such that $ \alpha$ satisfies the SIP. Let $s \in \alpha$ and $ c_s(\alpha)$ be the number of consecutions of $\alpha$ at $s$. Then $\alpha \sim \left( \alpha^L, s, s+1, \ldots , s+c_s(\alpha) , \alpha^R \right)$ for some tuples $\alpha^L$ and $\alpha^R$  such that $s \notin \alpha^L$ and $ s+c_s(\alpha) , s+c_s(\alpha)+1 \notin \alpha^R$. Similarly, let $ t \in \alpha$ and $ i_t(\alpha)$ be  the number of inversions of $ \alpha$ at $ t$. Then 
	$\alpha \sim \left( \alpha^L,  t+i_t(\alpha), \ldots ,t+1, t , \alpha^R \right)$for some tuples $\alpha^L$ and $\alpha^R$  such that $ t\notin \alpha^R$ and $ t+i_t(\alpha) , t+i_t(\alpha)+1 \notin \alpha^L$.
\end{proposition}

\begin{proposition} \label{lemeqnconrsfcsf17a18d10} Let $\alpha$ be an index tuple containing indices from $ -\{ 0:m\}$ such that $ \alpha$ satisfies the SIP. Then we have the following. 
	\begin{itemize}
		\item[(a)] If $-s \in \alpha$ and $ c_{-s} (\alpha) = p$, then $\alpha \sim \big( \alpha^L, -s, -(s-1), \ldots , -(s-p), \alpha^R \big)$ for some index tuples $ \alpha^L$ and $ \alpha^R$ such that $ -s \notin \alpha^L$ and $ -(s-p ) , -(s-p-1) \notin \alpha^R$.
		
		\item[(b)] if $ -t \in \alpha$ and $ i_{-t} (\alpha) =  q $, then
		$\alpha \sim \big(\alpha^L,  -(t-q), \ldots ,-(t-1),-t, \alpha^R \big) $ for some index tuples $ \alpha^L$ and $ \alpha^R$ such that $ -t \notin \alpha^R$ and $ -(t-q ) , -(t-q-1) \notin \alpha^L$.
	\end{itemize}
\end{proposition}

\section{Extended generalized Fiedler pencil  of $P(\lam)$} \label{EGFPsOfP}  Fiedler pencils (FPs), generalized Fiedler pencils (GFPs), Fiedler pencils with repetition (FPRs) and generalized Fiedler pencils with repetition (GFPRs) are well known classes of strong linearizations of matrix polynomials \cite{tdm,BDD,VAEN,Bueno_DFR15}. In this section, we construct a new family linearizations (named as EGFPs) of $P(\lam)$ that subsumes FPs, GFPs, FPRs and GFPRs. We show by examples that EGFPs are rich source for construction of structure preserving linearizations for structured $P(\lam)$. Then in Section~\ref{LowbandP}, we characterize block tridiagonal and block penta-diagonal EGFPs of $P(\lam)$.


\begin{definition}[EGFP pf $P(\lam)$] \label{ggfprdef_PCh4} Let $(\sigma,\omega)$ be a permutation of $\{ 0 :m \}$. Set  $\tau:= - \omega.$ Let $\sigma_1$ and $\sigma_2$ be index tuples containing indices from $\sigma \setminus \{ m-1,m\}$ such that $(\sigma_1,\sigma , \sigma_2)$ satisfies the SIP. Similarly, let $\tau_1$ and $\tau_2$ be index tuples containing indices from $\tau \setminus \{ -1,-0\}$ such that $( \tau_1 , \tau ,\tau_2)$ satisfies the SIP. Let $X_1$, $X_2$, $Y_1$ and $Y_2$ be any arbitrary nonsingular matrix assignments for $ \sigma_1$, $\sigma_2$, $\tau_1$, and $\tau_2$, respectively. Then the pencil
	\begin{equation}\label{EGFPDef}
		L(\lambda):= M_{\tau_1} (Y_1) \, M_{\sigma_1} (X_1) \, ( \lambda M^P_{\tau} -  M^P_{\sigma} ) \, M_{\sigma_2} (X_2) \, M_{\tau_2} (Y_2)
	\end{equation}
	is said to be an \emph{extended generalized Fiedler pencil (EGFP)} of $P(\lambda)$.
\end{definition}


\begin{example}  \label{eg_egfpr} Let $P(\lambda) := \sum_{i=0}^5 \lambda^i A_i$. Then $L(\lam) :=  \big ( \lambda M^P_{(-5,-1)} - M^P_{(3,4,2,0)} \big ) M_{3} (X)  $ 
	    $$=\left[\begin{array}{ccccc} \lam A_{5}+A_{4} & -X & -I_n & 0 & 0\\ A_{3} &  \lam X + A_{2} & \lam I_n & -I_n & 0\\ -I_n & \lam I_n & 0 & 0 & 0\\ 0 & -I_n & 0 & 0 & \lam I_n\\ 0 & 0 & 0 & \lam I_n & \lam A_{1}+A_{0} \end{array}\right]$$
	is an EGFP of $P(\lam)$. Note that $L(\lam)$ does not belong to FPs, GFPs, FRs, and GFPRs. Note that if $P(\lam)$ is symmetric (i.e., $A_i=A_i^T$, for $i=0:m$) then $L(\lam)$ is symmetric for $X=-A_3$. Further, we mention that this symmetric pencil can not be generated using the framework given in  \cite{Bueno_DFR15} for symmetric linearizations. $\blacksquare$
\end{example}

It is easy to observe that the family of EGFPs of $P(\lam)$ contains all the Fiedler-like pencils FPs, GFPs, FPRs and GFPRs of $P(\lam)$. Indeed we have the following remark. 
\begin{remark} Let $L(\lam)$ be as given in Definition~\ref{ggfprdef_PCh4}. Then by considering special types of $\sig,\tau, \sig_j$ and $\tau_j$, $j=1,2,$ we have the following:
	\begin{itemize}
		\item[(1)] if  $\sig$ is a permutation of $\{0:m-1\}$, $\tau=(-m)$, $\sig_1=\emptyset=\sig_{2}$ and $\tau_1=\emptyset=\tau_{2}$ then $L(\lam)$ is an FP of $P(\lam)$. 
	\item[(2)] if $\sig_1=\emptyset=\sig_{2}$ and $\tau_1=\emptyset=\tau_{2}$ the  $L(\lam)$ is a GFP of $P(\lam)$. 
	\item[(3)] if $\sig$ is a permutation of $\{0:h\}, ~0\leq h \leq m-1$, $\tau$ is a permutation of $\{-m:-(h+1)\}$ then $L(\lam)$
	is a GFPR of $P(\lam)$.
	\item[(4)] if $\sig$ and $\tau$ are as given in (3), and  $X_i$ and $ Y_i, ~i=1,2,$ are  trivial matrix assignments, then  $L(\lam)$ is an FPR of $P(\lam)$.
	\end{itemize} 
\end{remark}

%

The following result shows that, with some generic nonsingularity conditions, the EGFPs are strong linearization of $P(\lam)$.  
\begin{theorem}\label{singularcond11_PCh4} If  all the matrix assignments $X_j$ and $Y_j$, $j=1,2,$ are nonsingular, then the EGFP $L(\lambda)$ as given in Definition~\ref{ggfprdef_PCh4}
	is a strong linearization of $P(\lambda)$. 
\end{theorem}

\begin{proof} We have   $L(\lambda)=  M_{(\tau_1,\sigma_1)} (Y_1,X_1)\, T(\lambda)\,  M_{(\sigma_2,\tau_2)} (X_2,Y_2)$, where  $ T(\lambda):= \lambda M^P_{\tau} - M^P_{\sigma} $ is a GF pencil of $P(\lambda)$. As $X_j$ and $Y_j$, $j=1,2,$ are nonsingular matrix assignments, $ M_{(\tau_1,\sigma_1)} (Y_1,X_1)$ and $M_{(\sigma_2,\tau_2)} (X_2,Y_2)$ are nonsingular matrices. Since $T(\lambda)$ is a strong linearization of $P(\lambda)$, it follows that $L(\lambda)$ is a strong linearization of $P(\lambda)$.
\end{proof}

\noin {\bf Assumption:} Henceforth, for simplicity and without loss of generality, throughout this paper we assume that the matrix assignments are always~nonsingular.

\begin{remark} It follows from Definition~\ref{ggfprdef_PCh4} that  $ (\tau_1, \sigma_1) \not\sim (\sigma_1, \tau_1)$ and $ (\sigma_2, \tau_2) \not\sim ( \tau_2, \sigma_2)$ in general. So $ M_{\tau_1} (Y_1) M_{\sigma_1}(X_1) \neq M_{\sigma_1} (X_1) M_{\tau_1} (Y_1)$ and $M_{\sigma_2} (X_2) M_{\tau_2}  (Y_2)\neq M_{\tau_2} (Y_2)M_{\sigma_2}(X_2)$. Thus, for 
	$
	L(\lam) :=M_{\tau_1}(Y_1) M_{\sigma_1}(X_1) (\lambda  M^P_{\tau}  -  M^P_{\sigma}) M_{\sigma_2} (X_2)M_{\tau_2} (Y_2)
	$
	and
	$T(\lam):=  M_{\sigma_1} (X_1) $ $M_{\tau_1} (Y_1) (\lambda  M^P_{\tau}  -  M^P_{\sigma}) M_{\tau_2} (Y_2) M_{\sigma_2} (X_2),$
	we have $L(\lam) \neq T(\lam).$
	By interchanging the positions of $M_{\sig_j}(X_j)$ with $M_{\tau_j}(Y_j)$ in $L(\lam)$ we obtain a new family of pencils for $P(\lam)$. 
	
\end{remark}

An EGFP $L(\lam):=\lam L_1-L_0$ is said to be operation if each block of $L_1$ and $L_0$ is either $0,\, \pm I_n, \, A_j, A_j^{-1}, \,j=0:m,$ or any one of the matrices in the matrix assignments $X_j, Y_j$, $j=1,2.$ The following result shows that EGFPs are operation free. We prove Proposition~\ref{OF_EGFPR} in Appendix~\ref{OpeFreeAppen}.

\begin{proposition} \label{OF_EGFPR} Let $L(\lambda):=  M_{\tau_1} (Y_1) M_{\sigma_1} (X_1) (\lambda  M^P_{\tau}  -  M^P_{\sigma}) M_{\sigma_2}  (X_2)M_{\tau_2}  (Y_2) =: \lam L_1 - L_0$ be an EGFP of $P(\lam)$ such that $m-1$ and $m$ simultaneously do not belong to $\sigma$,  and $-1$ and $-0$ simultaneously do not belong to $\tau$. Then $L(\lam)$ is operation free.
\end{proposition}

We end this section by constructing few examples of structure preserving linearizations for structured (symmetric, even, odd, palindromic, skew-symmetric) $P(\lam)$.

\begin{example} Let $P(\lambda) := \sum_{i=0}^6 \lambda^i A_i$ be symmetric. Then the EGFP $L(\lam)$ given by $L(\lam) :=  \big ( \lambda M^P_{(-6,-3,-2,-4,-0)} - M^P_{(1,5)} \big ) M^P_{-3} =$
	$$ \left[\begin{array}{cccccc} \lam A_{6} + A_{5} & -I_n & 0 & 0 & 0 & 0\\ -I_n& 0 & 0 & \lam I_n & 0 & 0\\
		 0 & 0 & 0 & -I_n & \lam I_n & 0\\ 
		 0 & \lam I_n & - I_n & \lam A_{4}-A_{3} &\lam  A_{3} & 0\\
		  0 & 0 & \lam I_n&\lam A_{3} & \lam A_{2} +A_{1} & -I_n\\ 0 & 0 & 0 & 0 & -I_n & -A^{-1}_{0} \end{array}\right]
	$$ 	is a symmetric strong linearization of $P(\lam)$ if $A_0$ is nonsingular. 
\end{example}

 A matrix polynomial $ P(\lam) := \sum_{j=0}^{m} A_j \lam^j$ is said to be  $T$-even if $P(-\lam)^T =P(\lam)$,  i.e., $A_i^T=(-1)^iA_i$  for $i=0:m$. Similarly, $ P(\lam)$ is said to be  $T$-odd if $P(-\lam)^T =-P(\lam)$,  i.e., $A_i^T=(-1)^{i+1}A_i$  for $i=0:m$.

\begin{example} Let $P(\lambda) := \sum_{i=0}^5 \lambda^i A_i$ be $T$-even. Let $L(\lam) $ be the EGFP of $P(\lam)$ as given in Example~\ref{eg_egfpr}, where $X=-A_3$.  Set $Q= \diag (I_n, -I_n, I_n, -I_n, I_n)$. Then 
	$$QL(\lam)=\left[\begin{array}{ccccc} \lam A_{5}+A_{4} & A_3 & -I_n & 0 & 0\\ 
		-A_{3} &  \lam A_3 - A_{2} & -\lam I_n & I_n & 0\\ 
		-I_n & \lam I_n & 0 & 0 & 0\\ 
		0 & I_n & 0 & 0 & -\lam I_n\\ 0 & 0 & 0 & \lam I_n & \lam A_{1}+A_{0} \end{array}\right]$$
	is a $T$-even strong linearization of $P(\lam)$ but it cannot be generated using the framework of FPs, GFPs, FPRs and GFPRs. 	
\end{example}

\begin{example} Let $P(\lambda) := \sum_{i=0}^5 \lambda^i A_i$ be $T$-odd.  Let $L(\lam)$ be the EGFP of $P(\lam)$ as given in Example~\ref{eg_egfpr}, where $X=-A_3$. Set $Q= \diag (I_n, -I_n, -I_n, I_n, -I_n)$. Then 
	$$QL(\lam)=\left[\begin{array}{ccccc}
		 \lam A_{5}+A_{4} & A_3 & -I_n & 0 & 0\\ 
		 -A_{3} &  \lam A_3 - A_{2} & -\lam I_n & I_n & 0\\ 
		 I_n & -\lam I_n & 0 & 0 & 0\\
		  0 & -I_n & 0 & 0 & \lam I_n\\
		   0 & 0 & 0 & -\lam I_n & -\lam A_{1}-A_{0} \end{array}\right]$$
	is a $T$-odd strong linearization of $P(\lam)$ but it cannot be generated using the framework of FPs, GFPs, FPRs and GFPRs.  $\blacksquare$
\end{example}

The matrix polynomial $ P(\lam) := \sum_{j=0}^{m} A_j \lam^j$ is said to be  skew-symmetric if $P(\lam)^T=-P(\lam)$, i.e., $A^T_i = -A_{i}$ for $i=0,1,\ldots, m$. 
\begin{example} Let $P(\lambda) := \sum_{i=0}^5 \lambda^i A_i$ be skew-symmetric.  Let $L(\lam) :=  \big ( \lambda M^P_{(-5,-1)} - M^P_{(3,4,2,0)} \big ) M^P_{3} $ be the EGFP of $P(\lam)$ as given in Example~\ref{eg_egfpr}. Set $Q= \diag (I_n, I_n, -I_n, -I_n, I_n)$. Then 
	  $$QL(\lam)= \left[\begin{array}{ccccc} 
	  	\lam A_{5}+A_{4} & A_3 & -I_n & 0 & 0\\
	  	 A_{3} &  -\lam A_3 + A_{2} & \lam I_n & -I_n & 0\\
	  	  I_n & -\lam I_n & 0 & 0 & 0\\
	  	   0 & I_n & 0 & 0 & -\lam I_n\\ 0 & 0 & 0 & \lam I_n & \lam A_{1}+A_{0} \end{array}\right]$$
	is a skew-symmetric strong linearization of $P(\lam)$  but it cannot be generated using the framework of FPs, GFPs, FPRs and GFPRs.
\end{example}

The matrix polynomial $ P(\lam) := \sum_{j=0}^{m} A_j \lam^j$ is said to be  $T$-palindromic (resp., $T$-anti-palindromic) if $A^T_i = A_{m-i}$ (resp., $A^T_i = -A_{m-i}$) for $i=0:m$. A subfamily of EGFPs that preserve the $T$-Palindromic and $T$-anti-palindromic structure of $P(\lam)$ is studied in \cite{rafiran5}. The following example is a palindromic linearization of a palindromic $P(\lam)$, see \cite{rafiran5} for banded palindromic linearizations with specific block-bandwidth.  Let 
$P(\lam) :=\sum_{i=0}^7 \lam^i A_i$ be $T$-palindromic. Then the anti-block-penta-diagonal pencil 
{\small  \begin{equation*}  \label{ExPalin} 
		\lam \left[\begin{array}{@{}ccccccc@{}}
			0 & 0 & 0 & 0 & 0 & - I_n & 0\\ 0 & 0 & 0 & 0 & 0 & - A_1 & - A_0\\ 0 & 0 & 0 & I_n & A_3 & 0 & 0\\ 0 & 0 & 0 & 0 & -I_n & 0 & 0\\ 
			I_n & A_6 & A_5 & 0 & 0 & 0 & 0\\ 0 & A_7 & A_6 & 0 & 0 & 0 & 0\\ 0 & 0 & A_7 & 0 & 0 & 0 & 0 \end{array}\right]
		+
		\left[\begin{array}{@{}ccccccc@{}}
			0 & 0 & 0 & 0 & I_n & 0 & 0\\ 0 & 0 & 0 & 0 & A_1 & A_0 & 0\\ 0 & 0 & 0 & 0 & A_2 & A_1 & A_0\\ 
			0 & 0 & I_n & 0 & 0 & 0 & 0\\ 0 & 0 & A_4 & -I_n & 0 & 0 & 0\\ - I_n & - A_6 & 0 & 0 & 0 & 0 & 0\\ 0 & - A_7 & 0 & 0 & 0 & 0 & 0 \end{array}\right] 
\end{equation*}}is a $T$-palindromic strong linearization of $P(\lam)$ but it cannot be generated using the framework of FPs, GFPs, FPRs and GFPRs.


\subsection{Low bandwidth banded EGFPs} \label{LowbandP} Low bandwidth banded linearizations of $P(\lam)$ are important from the point of numerical computations.  As EGFPs of $P(\lam)$ substantially enlarge the arena in which to look for strong linearizations of $P(\lam)$, it is natural to investigate the possibility to construct low bandwidth banded EGFPs. In particular, we are interested to construct block tridiagonal and block penta-diagonal EGFPs of $P(\lam)$.  The following simple observations from the elementary matrices of $P(\lam)$ will be used throughout of the paper.

We make the convention that $e_j  = 0$ for $j \leq 0$ and $j >m$. Let $Z$ be an $ n \times n$ arbitrary matrix. Then for $j=0:m-2$, we have  	
$$(e^T_{m-i} \otimes I_n) M_{i+1} (Z) = (e^T_{m-i} \otimes I_n) \left[
\begin{array}{@{}cccc@{}}
	I_{(m-i-2)n} &  & & \\
	& Z & I_n &\\
	& I_n & 0 &\\
	
	& & & I_{in}                                                               \end{array}
\right] =  e^T_{m-(i+1)} \otimes I_n,$$
and for $i=0:m-1$ and $j \notin \{ i,i+1\}$, we have
$(e^T_{m-i} \otimes I_n) M_{j} (Z) = e^T_{m-i} \otimes I_n.$
This shows that
\begin{equation}\label{eqnreco2_PCh4} 
	(e^T_{m-i}  \otimes I_n) \, M_{j} (Z)=
	\left\{ 
	\begin{array}{ll}
		e^T_{m-(i+1)} \otimes I_n & \mbox{for } j=i+1,~i=0:m-2, \\
		e^T_{m-i} \otimes I_n & \mbox{for } j \notin \{ i,  i+1 \},~i=0:m-1.  \end{array}
	\right.
\end{equation} 
Similarly, we have 
\begin{equation}\label{eqnreco1_PCh4}
	M_{j} (Z) (e_{m-i} \otimes I_n)=
	\left\{ 
	\begin{array}{ll}
		e_{m-(i+1)} \otimes I_n & \mbox{for } j=i+1, ~ i=0:m-2 ,\\
		e_{m-i} \otimes I_n & \mbox{for } j\notin \{ i,  i+1 \},~i=0:m-1.
	\end{array}
	\right. 
\end{equation} 

Further, for $i=1:m-1$, we have
$$(e^T_{m-i} \otimes I_n) M_{-i} (Z) = (e^T_{m-i} \otimes I_n) \left[
\begin{array}{@{}cccc@{}}
	I_{(m-i-1)n} &  & & \\
	& 0 & I_n &\\
	& I_n & Z &\\
	& & & I_{(i-1)n}                                                               \end{array}
\right] =    e^T_{m-(i-1)} \otimes I_n,$$
and for $i=0:m-1, ~j \notin \{i,i+1\}, $ we have  $(e^T_{m-i} \otimes I_n) M_{-j} (Z)= e^T_{m-i} \otimes I_n.$ Thus

\begin{equation}\label{eqnrecogfprp0329sept172046n_PCh2}
	(e^T_{m-i} \otimes I_n) M_{-j} (Z) =
	\left\{
	\begin{array}{ll}
		e^T_{m-(i-1)} \otimes I_n & \mbox{for}~j = i~ \mbox{and}~ i=1:m-1, \\
		e^T_{m-i} \otimes I_n &    \mbox{for} ~j \notin \{i,i+1\},~ j=0:m-1,
	\end{array}
	\right.
\end{equation}
Similarly, we have
\begin{equation}\label{eqnrecogfprgsettingp0329sept172045n_PCh2}
	M_{-j} (Z)(e_{m-i} \otimes I_n) =
	\left\{
	\begin{array}{ll}
		e_{m-(i-1)} \otimes I_n & \mbox{for}~j = i~ \mbox{and}~ i=1:m-1, \\
		e_{m-i} \otimes I_n & \mbox{for}~j \notin \{i,i+1\},~ i=0:m-1.

	\end{array}
	\right.
\end{equation}

We need the following propositions to characterize the bock penta-diagonal EGFPs of $P(\lam)$. Since the proofs are lengthy and involved, we provide the proofs of Propositions~\ref{prop1:Bpenta}, \ref{Prop:2:Blkpenta}, \ref{not_blkpentaPve} and \ref{not_blkpentaNve} in Appendix~\ref{appendixBB} to make this paper simpler to read.

\begin{proposition} \label{prop1:Bpenta} Let $\alpha$ be an index tuple containing indices from $\{0:m-1\}$ such that $\alpha$ satisfies the SIP. Suppose that $c_k(\alpha) \leq 1$ and  $i_k(\alpha) \leq 1$ for any index $1 \leq k \leq m-1$. Let $\mathcal{X}$ be a matrix assignment for $\alpha$. Then   $M_\alpha (\mathcal{X})$ is a block penta-diagonal matrix. Equivalently, we have
	\begin{equation} \label{aim:to:prove}
		(e^T_{m-k} \otimes I_n) M_\alpha (\mathcal{X}) = \sum_{j= -2}^{2} (e^T_{m-(k-j)} \otimes X_j) \text{ for } k = 0:m-1,
	\end{equation}
	where $X_j = 0$ or $X_j$ belongs to the matrix assignment $\mathcal{X}$.    
\end{proposition}
We prove Proposition~\ref{prop1:Bpenta} in Appendix~\ref{appendix1}.

Analogous to Proposition~\ref{prop1:Bpenta} we have the following result for index tuple containing negative indices.

\begin{proposition} \label{Prop:2:Blkpenta} Let $\alpha$ be an index tuple containing indices from $\{-m:-1\}$ such that $\alpha$ satisfies the SIP. Suppose that $c_{-k}(\alpha) \leq 1$ and  $i_{-k}(\alpha) \leq 1$ for any index $ -(m-1) \leq -k \leq -1$ (this implies that $c_{-m}(\alpha) \leq 2$ and  $i_{-m}(\alpha) \leq 2$). Let $\mathcal{X}$ be a matrix assignment for $\alpha$. Then $M_\alpha (\mathcal{X})$ is a block penta-diagonal matrix. Equivalently, we have
	\begin{equation}  \label{aim:to:prove2}
		(e^T_{m-k} \otimes I_n) M_\alpha (\mathcal{X}) = \sum_{j= -2}^{2} (e^T_{m-(k-j)} \otimes X_j) \text{ for } k = 0:m-1,
	\end{equation}
	where $X_j = 0$ or $X_j$ belongs to the matrix assignment $\mathcal{X}$.   
\end{proposition}


\begin{proposition} \label{not_blkpentaPve} Let $\alpha$ be an index tuple containing indices from $\{0:m\}$ such that $\alpha$ satisfies the SIP. Suppose that $c_j(\alpha) \geq 2$ or $i_j(\alpha) \geq 2$ for some $1 \leq j \leq m-1$. Let $\mathcal{X}$ be any arbitrary nonsingular matrix assignment for $\alpha$. Then $M_\alpha (\mathcal{X})$ is not a block penta-diagonal matrix.   
\end{proposition}


For index tuple containing negative indices we have the following result.
\begin{proposition} \label{not_blkpentaNve} Let $\beta$ be an index tuple containing indices from $\{-m:-0\}$ such that $\beta$ satisfies the SIP. Suppose that $c_{-j}(\beta) \geq 2$ or $i_{-j} (\beta) \geq 2$ for some $1 \leq j \leq m-1$. Let $\mathcal{X}$ be a matrix assignment for $\beta$. Then $M_\beta (\mathcal{X})$ would never be a block penta-diagonal matrix.   
\end{proposition}

To characterize block penta-diagonal EGFPs we need the following definition.  

\begin{definition} \label{endpt} Let $\alpha$ and $\beta $ be  sub-permutations of $\{0:m-1\}$ and $\{-m:-1\}$, respectively. Then $k \in \alpha \setminus \{0\}$ is said to be an end index of  $\alpha$ if $ k-1 \notin \alpha$ or $k+1 \notin \alpha$. Similarly, $- t \in \beta \setminus \{ -m\}$ is said to be an end index of  $\beta$ if $ -(t -1) \notin \beta$ or $-(t+1) \notin \beta$. 
\end{definition}

The following theorem characterizes all block penta-diagonal EGFPs of $P(\lam)$. 

\begin{theorem} \label{thm:blk_penta} Let $L(\lambda):=  M_{\tau_1} (Y_1) M_{\sigma_1} (X_1) (\lambda  M^P_{\tau}  -  M^P_{\sigma}) M_{\sigma_2}  (X_2)M_{\tau_2}  (Y_2)$ be an EGFP of $P(\lam)$. Suppose that $\sig_j$ (resp., $\tau_j$), for $j=1,2$, does not contain the end indices of $\sig$ (resp., $\tau$).
	Then $L(\lam)$ is block penta-diagonal if and only if $ c_t(\sig_1, \sig, \sig_2) \leq 1$, $ i_t(\sig_1, \sig, \sig_2) \leq 1$, $ c_{-t}(\tau_1, \tau, \tau_2) \leq 1$ and $ i_{-t}(\tau_1, \tau, \tau_2) \leq 1$ for any index $1 \leq t \leq m-1$. 
\end{theorem}

\begin{proof} ($\Longrightarrow$) The forward implication follows from Proposition~\ref{not_blkpentaPve} and Proposition~\ref{not_blkpentaNve}.

	($\Longleftarrow$) 	 We have  $L(\lam) =: \lam L_1 - L_0$, where $ L_1 : =  M_{\tau_1} (Y_1) M_{\sigma_1} (X_1)  M^P_{\tau}  M_{\sigma_2}  (X_2)M_{\tau_2}  (Y_2) $ and  $ L_0 : =  M_{\tau_1} (Y_1) M_{\sigma_1} (X_1) M^P_{\sigma} M_{\sigma_2}  (X_2)M_{\tau_2}  (Y_2)$. Note that $L(\lam)$ is a block penta-diagonal pencil if and only if  both $L_0$ and $L_1$ are block penta-diagonal matrices.  We only prove that $L_0$ is a block penta-diagonal matrix. Similarly, $L_1$ is block penta-diagonal.

	Since $\tau_1$ and $\tau_2$ do not contain the end indices of $\tau$, it follows that $||j| - |k|| \geq 2$ for $ j \in (\tau_1,\tau_2) $ and $k \in (\sig_1, \sig, \sig_2) $. Hence  $\tau_j $ commutes with $ (\sig_1, \sig, \sig_2)$, $j=1,2$, and hence $L_0=  M_{\tau_1} (Y_1) M_{\tau_2}  (Y_2) M_{\sigma_1} (X_1) M^P_{\sigma} $ $ M_{\sigma_2}  (X_2) =   M_{\sigma_1} (X_1) M^P_{\sigma}$ $ M_{\sigma_2}  (X_2)   M_{\tau_1} (Y_1) M_{\tau_2}  (Y_2)$.

	Suppose that $ m \in \sig$. Recall that, as $\sig_1$ and $\sig_2$ contain indices from $\{0:m-2\}$, there are no repetition of $m-1$ and $ m$ in the indices of $(\sig_1, \sig, \sig_2)$. Hence $ (\sig_1, \sig, \sig_2) \sim (\beta ,m)$ (resp., $(\sig_1, \sig, \sig_2) \sim (m,\beta)$) if $m-1 \in \sig $ and $ \sig$ has a consecution (resp., inversion) at $m-1$, where $\beta : = (\sig_1, \sig, \sig_2) \setminus \{m\} $. Further, since $ M_m (W) = \diag (W^{-1}, I_{(m-1)n})$ (for any nonsingular matrix $W \in \mathbb{C}^{n \times n}$) is a block diagonal matrix, without loss of generality we assume that $m \notin \sig$, i.e., $ m \notin (\sig_1,\sig , \sig_2)$.

	Now we prove that $L_0$ is a block penta-diagonal matrix.  Note that to prove $L_0$ is a block penta-diagonal matrix it is enough to show that   
	\begin{equation} \label{thm:aim:to:prove}
		(e^T_{m-k} \otimes I_n) L_0 = \sum_{j= -2}^{2} (e^T_{m-(k-j)} \otimes Z_j) \text{ for } k = 0:m-1,
	\end{equation}
	where $Z_j $ is any one of the matrices $ 0, \pm I_n, \pm A_0, \ldots , \pm A_m $ or the matrices in  $X_1,X_2,Y_1$ and $Y_2$.  Since $(\sig, -\tau)$ is a permutation of $\{0:m\}$, we have either $k \in \sig$ or $k \in -\tau$.  
	
	Case-I: Suppose that $k \in \sig$. Then $ -k \notin \tau$. This implies that $-k,-(k+1) \notin (\tau_1,\tau_2)$ since $(\tau_1,\tau,\tau_2)$ satisfies the SIP. Hence by (\ref{eqnrecogfprp0329sept172046n_PCh2}), we have $(e^T_{m-k} \otimes I_n) M_{\tau_1} (Y_1) M_{\tau_2}  (Y_2) = e^T_{m-k} \otimes I_n. $ This implies that $(e^T_{m-k} \otimes I_n) L_0 = (e^T_{m-k} \otimes I_n)  M_{\sigma_1} (X_1) M^P_{\sigma} M_{\sigma_2}  (X_2) . $ Now (\ref{thm:aim:to:prove}) follows from Proposition~\ref{prop1:Bpenta} by considering $\alpha : = (\sig_1, \sig, \sig_2)$ and $\mathcal{X} : = (X_1, P, X_2)$, where $P$ denotes the trivial matrix assignment for $\sig.$

	Case-II: Suppose that $k \in -\tau$. Then there are two cases.
	
	(a) Suppose that $k+1 \in \sig$ (i.e., $-k$ is an end index of $\tau$).  Since $\tau_j$, $j=1,2,$ does not contain the end indices of $\tau$, we have $k,k+1 \notin -(\tau_1,\tau_2)$.  Now by following the similar arguments as given in Case-I we have (\ref{thm:aim:to:prove}).
	
	(b) Suppose that $-(k+1) \in \tau$. Then, since $(\sig,-\tau)$ is a permutation of $\{0:m\}$, we have $k,k+1 \notin (\sig_1,\sig,\sig_2)$. Hence by (\ref{eqnreco2_PCh4}), we have $(e^T_{m-k} \otimes I_n) M_{\sigma_1} (X_1) M^P_{\sigma} M_{\sigma_2}  (X_2) = e^T_{m-k} \otimes I_n$.  Consequently, we have
	$(e^T_{m-k} \otimes I_n)  L_0  = (e^T_{m-k} \otimes I_n)  M_{\tau_1} (Y_1) M_{\tau_2}  (Y_2) . $ Now (\ref{thm:aim:to:prove}) follows from Proposition~\ref{Prop:2:Blkpenta} by considering $\beta: = (\tau_1,\tau_2)$ and $\mathcal{X} : = (Y_1,Y_2)$. Hence $L_0$ is a block penta-diagonal matrix. This completes the proof. 
\end{proof}

For constructing block tridiagonal EGFPs of $P(\lam)$ we need the following results. 
\begin{proposition} \label{not_blktri} Let $\alpha$ be an index tuple containing indices from $\{0:m\}$ such that $\alpha$ satisfies the SIP. Suppose that $c_j(\alpha) \geq 1$ or $i_j(\alpha) \geq 1$ for some $1 \leq j \leq m-1$.  Let $\mathcal{X}$ be any arbitrary nonsingular matrix assignment for $\alpha$. Then $M_\alpha (\mathcal{X})$ is not a block-tridiagonal matrix. 
	
	Similarly, let $\beta$ be an index tuple containing indices from $\{-m:-0\}$ such that $\beta$ satisfies the SIP. Suppose that $c_{-j}(\beta) \geq 1$ or $i_{-j} (\beta) \geq 1$ for some $1 \leq j \leq m-1$.  Let $\mathcal{Y}$ be any arbitrary nonsingular matrix assignment for $\beta$. Then $M_\beta (\mathcal{Y})$ is not a block-tridiagonal matrix.   
\end{proposition}

\begin{proof} The proof is similar to those of Proposition~\ref{not_blkpentaPve} and Proposition~\ref{not_blkpentaNve}.
\end{proof}

The following theorem characterizes all block tridiagonal EGFPs of $P(\lam)$.  Note that, for an index tuple $\alpha$ containing nonnegative indices and any index $ t $, we have $ c_t (\alpha) = 0$ and $ i_t (\alpha) = 0$ imply that $ t \in \alpha$ but $ t-1 , t+1 \notin \alpha$. Similar thing holds for index tuple containing negative indices.

\begin{theorem} \label{thm:blk_tri} Let $L(\lambda):=  M_{\tau_1} (Y_1) M_{\sigma_1} (X_1) (\lambda  M^P_{\tau}  -  M^P_{\sigma}) M_{\sigma_2}  (X_2)M_{\tau_2}  (Y_2)$ be an EGFP of $P(\lam)$. Then $L(\lam)$  is block tridiagonal if and only if $ c_t(\sig_1, \sig, \sig_2) = 0 $, $ i_t(\sig_1, \sig, \sig_2) =0$,  $ c_{-t}(\tau_1, \tau, \tau_2) = 0 $ and $ i_{-t}(\tau_1, \tau, \tau_2) =0$ for any index $1 \leq t \leq m-1$. 
\end{theorem}

\begin{proof} ($\Longrightarrow$) The forward implication follows from Proposition~\ref{not_blktri}.

	($\Longleftarrow$) Define $\mathcal{A} := M_{\tau_1} (Y_1) M_{\sig_1} (X_1)  $ and $\mathcal{B} : =  M_{\sigma_2}  (X_2)M_{\tau_2}  (Y_2)$. Then we have $L(\lam) =: \lam L_1 - L_0$, where $ L_1 : = \mathcal{A} M^P_{\tau} \mathcal{B}  $ and  $ L_0 : =  \mathcal{A} M^P_{\sigma} \mathcal{B}$. Note that $L(\lam)$ is a block-tridiagonal pencil if and only if both $L_0$ and $L_1$ are block-tridiagonal matrices.  We only proof that $L_0$ is a block tridiagonal matrix. Similarly, $L_1$ is block tridiagonal. 
	
	It is given that, for any index $1 \leq t \leq m-1$, if $t \in \sig$ then we have $ c_t(\sig_1, \sig, \sig_2) = 0 $ and $ i_t(\sig_1, \sig, \sig_2) =0$. This implies that, for $ 1 \leq j \leq m-2$, if $ j \in \sig$ then $j+1 \notin \sig$. Moreover, we have $\sig_j =(0)$ or $\sig_j = \emptyset$, $j=1,2.$ Since $M_{0}(Z) = \diag (I_{(m-1)n}, Z) $, we have $M_{\sig_j}  (X_j) = \diag (I_{(m-1)n}, X_j)$ or $M_{\sig_j}  (X_j) = I_{mn}$ for $j =1,2.$ Further, given that  $ c_{-p}(\tau_1, \tau, \tau_2) = 0 $ and $ i_{-p}(\tau_1, \tau, \tau_2) =0$ for any index $-(m-1) \leq -p  \leq -1 $ and $-p \in \tau$. This implies that $\tau_j = (-m)$ or $\tau_j = \emptyset$, $j=1,2.$ Since  $M_{-m}(Z) = \diag (Z, I_{(m-1)n})$, we have $M_{\tau_j}  (Y_j) =  \diag (Y_j, I_{(m-1)n})$ or $M_{\tau_j}  (Y_j) =  I_{mn}$, $j=1,2.$ Consequently, $\mathcal{A}$ and $\mathcal{B}$ are block diagonal matrices. Hence $L_0$ is block tridiagonal if and only if $ M^P_\sig $ is block tridiagonal. Since, $|j-k| \geq 2$ for distinct $ j,k \in \sig$ with $ 1 \leq j, k \leq m-1$, it is clear from the Fiedler matrices $M^P_{\ell}$, $\ell = 0:m$, that $M^P_{\sig
	}$ is tridiagonal. This completes the proof. 
\end{proof}

\begin{remark} Since EGFPs subsumes FPs, GFPs, FPRs and GFPRs, Theorem~\ref{thm:blk_penta} and Theorem~\ref{thm:blk_tri}, respectfully, provides a characterization for block penta-diagonal and block tridiagonal FPs, GFPs, FPRs and GFPRs. 
\end{remark}

\section{Rosenbrock strong linearizations of rational matrices}\label{EGFPofG} 
In this section, we construct a family of Rosenbrock strong linearizations of a rational matrix $G(\lam)$ directly from the EGFPs of $P(\lam)$ as given in Section~\ref{EGFPsOfP}. We show that this family subsumes all the Fiedler-like linearizations constructed in \cite{rafinami1, rafinami3, rafiran4}. Throughout of this section, we consider the rational matrix $G(\lambda)= P(\lambda)+C(\lambda E-A)^{-1}B$ and its associated system matrix $\mathcal{S}(\lam)$ as given in (\ref{minrel2_RCh41}) and (\ref{slamsystemmatrix_RCh41}), respectfully.  

\begin{definition}[EGFP of $G(\lam)$] \label{EGFofG} Let $L(\lam)$ be an EGFP of $P(\lam)$ as given in (\ref{EGFPDef}) such that $0 \in \sig$ and $-m\in \tau$. Then the pencil
\begin{equation}\label{DefEGFP}
	\mathbb{L}(\lam) := 
	\left[
	\begin{array}{@{}c|c@{}}
		L(\lam) &  e_{m-i_0 ( \sigma_1, \sigma)} \otimes C \\[.1em] \hline \\[-1em]
		e^T_{m-c_0(\sigma, \sigma_2)} \otimes B  & A-\lam E
	\end{array}
	\right]
\end{equation} is said to be an EGFP of $G(\lam)$. We also refer to $\mathbb{L}(\lam)$ as an EGFP of $\mathcal{S}(\lambda)$.
\end{definition}

\begin{example} \label{exaEGFPG}Let  $G(\lambda) := \sum_{i=0}^5 \lambda^i A_i + C(\lambda E- A)^{-1}B.$ Let $L(\lam) :=  \big ( \lambda M^P_{(-5,-3)} - M^P_{(4,1,2,0)} \big ) M^P_{1} $ be the EGFP of $P(\lam)$ associated with $\sig=(4,1,2,0)$, $\tau = (-5,-3)$, $\sig_{1}=(1)$, $\sig_{2}=\emptyset$, and $\tau_1=\emptyset=\tau_2$. Note that $c_0(\sig,\sig_2)=1$ and $i_0(\sig_1,\sig)=1$. Then $$\mathbb{L}(\lam) = \left[
	\begin{array}{@{}ccccc|c@{}}
		\lam A_{5}+A_{4} & -I_n & 0 & 0 & 0&0\\ 
		-I_n&  0 & \lam I_n & 0& 0&0\\
		 0 & \lam I_n & \lam A_3+A_2 & A_1 & -I_n&0\\
		  0 & 0 & A_1 & -\lam A_1+A_0 & \lam I_n& C\\
		   0 & 0 & -I_n & \lam I_n & 0 &0\\ \hline
		   0&0&0&B&0&A-\lam E
	\end{array}\right]$$ is an EGFP of $G(\lam)$. 
\end{example}

Note that if $G(\lam)$ is symmetric (i.e., $G(\lam)^T=G(\lam)$) then there exists a minimal symmetric realization of $G(\lam)$ of the form $G(\lam) = P(\lam) + B^T(\lam E -A)^{-1}B$, where $P(\lam)$ is symmetric (i.e., $A_i^T=A_i, ~i=0:m$), and  $A$ and $E$ are symmetric matrices with $E$ being nonsingular \cite{fmq}. The system matrix 
	$\mathcal{S} (\lam) : =
	\left[ \begin{array}{@{}c|c@{}} P(\lam) &  B^T\\ \hline  B &   A -  \lam E \end{array} \right]
$ is obviously symmetric and irreducible. Then $\mathbb{L}(\lam)$ as given in Example~\ref{exaEGFPG} is a symmetric EGFP of $G(\lam)$ and it follows from Theorem~\ref{EGFPofGLin} that $\mathbb{L}(\lam)$ is a Rosenbrock strong linearization of $G(\lam)$.

From Definition~\ref{EGFofG}, it is clear that the family of EGFPs of $G(\lam)$ contains all the Fiedler-like pencils (FPs, GFPs, FPRs, GFRs) of $G(\lam)$ constructed in \cite{rafinami1, rafinami3, rafiran4}. Indeed, if the pencil $L(\lam)$ in Definition~\ref{EGFofG} is an FP (respectively, GFP, FPR and GFPR) of $P(\lam)$, then $\mathbb{L}(\lam)$ is an FP (respectively, GFP, FPR and GFPR) of $G(\lam)$.

Next we show that EGFPs of $G(\lam)$ are Rosenbrock strong linearizations of $G(\lam)$. To that end, we need the following result which is a corollary of \cite[Lemma~3.10]{rafiran3}.


\begin{lemma} \label{gfprptogArbitraryCoeff_PCh3} Let $0 \leq h \leq m-1,$ and let $\sigma$ be a permutation of $\{ 0 :h \}$. Let $\sigma_1$ and $\sigma_2$ be index tuples containing indices from $\{0 :h-1 \}$ such that $(\sigma_1,\sigma , \sigma_2)$ satisfies the SIP. Then $(e_{m-c_0(\sigma)}^{T}\otimes I_n) M_{\sigma_{2}} (Y)= e_{m-c_0(\sigma, \sigma_2)}^{T}\otimes I_n $
	and $M_{\sigma_{1}} (X) $ $ (e_{m-i_0(\sigma)}\otimes I_n) = e_{m-i_0(\sigma_1, \sigma)}\otimes I_n$ for any arbitrary matrix assignments $X$ and $Y$ for  $\sigma_1$ and $\sigma_2 $, respectively. 
\end{lemma}

Analogs to Lemma~\ref{gfprptogArbitraryCoeff_PCh3}, we have following result for EGFPs of $P(\lam)$ which will be used in proving that EGFPs of $G(\lam)$ are Rosenbrock strong linearizations of $G(\lam)$. 
\begin{lemma} \label{blk_rows_n_col}
	Let $ L(\lambda):= M_{\tau_1} (Y_1) \, M_{\sigma_1} (X_1) \, (\lambda M^P_{\tau} -  M^P_{\sigma})  \, M_{\sigma_2} (X_2) \,M_{\tau_2} (Y_2)$ be an EGFP of $P(\lambda)$ as given in (\ref{EGFPDef}).  Suppose that $0 \in \sigma$ and $-m \in \tau$. Then
	\begin{align}
	(e_{m-c_0(\sigma)}^{T}\otimes I_n) \, M_{\sigma_{2}} (X_2) \,  M_{\tau_{2}}  (Y_2)= e_{m-c_0(\sigma,\sigma_2)}^{T}\otimes I_n, \label{lem1eq1}\\
	\text{~and~}  M_{\tau_{1}} (Y_1)\, M_{\sigma_{1}} (X_1) \, (e_{m-i_0(\sigma)}\otimes I_n)= e_{m-i_0(\sigma_1, \sigma)}\otimes I_n.\label{lem1eq2}
	\end{align}

\end{lemma}

\begin{proof} Given that $0 \in \sig$ and $-m \in \tau$ (i.e., $m \notin \sig$). Let $h$ be the integer such that $0,1,\ldots,h \in \sig$ and $h+1 \notin \sig$. Then $c_0(\sig) \leq h $ and $ i_0(\sig)  \leq h$. Further, we have $h , h+1 \notin \sig_1 \cup \sig_2$ as  $h+1 \notin \sig$ and $(\sig_1, \sig, \sig_2)$ satisfies the SIP.

	Let $\widehat{\sig} $ and $\widehat{\sig}_j$, $j=1,2$, respectively, be the subtuples of $\sig$ and $\sig_j$ with indices $\{0:h\}$. Similarly, let $\widehat{\widehat{\sig}} $ and $\widehat{\widehat{\sig}}_j$, $j=1,2$, respectively,  be the subtuples of $\sig$ and $\sig_j$ with indices $\{h+2:m\}$. Then $\widehat{\sig} $ and $\widehat{\widehat{\sig}} $ commutes since for any indices $k \in\widehat{\sig} $ and $ \ell \in \widehat{\widehat{\sig}} $ we have $|k - \ell| >1$. Thus $\sig \sim (\widehat{\sig}, \widehat{\widehat{\sig}}) \sim ( \widehat{\widehat{\sig}},\widehat{\sig})$. Similarly, $\sig_j \sim  (\widehat{\sig}_j, \widehat{\widehat{\sig}}_j) \sim ( \widehat{\widehat{\sig}}_j, \widehat{\sig}_j)$ for $j =1,2.$ Further, $(\widehat{\sig} ,\widehat{\widehat{\sig}}_j) \sim (\widehat{\widehat{\sig}}_j,\widehat{\sig})$ for $j =1,2.$ Since $h+1 \notin \sig$, we have  $i_0(\sig) = i_0(\widehat{\sig})$ and $c_0(\sig) = c_0(\widehat{\sig})$. Further, $c_0(\sig, \sig_2)  =  c_0(\widehat{\widehat{\sig}},\widehat{\sig}, \widehat{\widehat{\sig}}_2,\widehat{\sig}_2)=c_0(\widehat{\widehat{\sig}}, \widehat{\widehat{\sig}}_2, \widehat{\sig}, \widehat{\sig}_2)= c_0(\widehat{\sig}, \widehat{\sig}_2)$, where  the last equality holds as $0 \notin \widehat{\widehat{\sig}} \cup \widehat{\widehat{\sig}}_2 $. Similarly, $i_0(\sig_1, \sig) =  i_0(\widehat{\sig}_1, \widehat{\sig}).$ Note that $X_2 $ and $Y_2$ are arbitrary matrix assignments. We by denote $(*)$ any arbitrary matrix assignment. Then we have
	\begin{align*}
		&    (e^T_{m-c_0(\widehat{\sigma})}\otimes I_n)   \, M_{\widehat{\widehat{\sigma}}_{2}} (*) M_{\widehat{\sigma}_{2}} (*) M_{\tau_{2}} (*) \\
		& = (e^T_{m-c_0(\widehat{\sigma})}\otimes I_n) \, M_{\widehat{\sigma}_{2}} (*) M_{\tau_{2}} (*)  \text{ by } (\ref{eqnreco2_PCh4}) \text{ since }  \left\{\begin{array}{l}
			c_0(\widehat{\sig}) = c_0(\sig) \leq h \text{ and }  \widehat{\widehat{\sigma}}_{2}\\  \text{contains indices from } \{h+2:m\} \end{array} \right. \\ 
		& = (e^T_{m-c_0(\widehat{\sigma}, \widehat{\sig}_2)} \otimes I_n ) M_{\tau_{2}} (*) ~\text{by Lemma~}\ref{gfprptogArbitraryCoeff_PCh3}\\
		&=  e^T_{m-c_0(\widehat{\sigma}, \widehat{\sig}_2)} \otimes I_n \text{ by (\ref{eqnrecogfprp0329sept172046n_PCh2})~since}  \left\{\begin{array}{l}
			c_0(\widehat{\sigma}, \widehat{\sig}_2) \leq h \text{ and }  \tau_{2}  \text{ contains indices } \\\text{from } \{-m:-(h+2)\} \end{array} \right..
	\end{align*}
	Thus  $(e_{m-c_0(\sigma)}^{T}\otimes I_n) \,  M_{\sigma_{2}} (X_2) M_{\tau_{2}} (Y_2)= e^T_{m-c_0(\sigma, \sig_2)} \otimes I_n$ which prove (\ref{lem1eq1}). Similar proof holds for  (\ref{lem1eq2}). Indeed, we have 
	\begin{align*}
		&   M_{\tau_{1}} (*)M_{\widehat{\sigma}_{1}} (*) M_{\widehat{\widehat{\sigma}}_{1}} (*) \, (e_{m-i_0(\widehat{\sigma})}\otimes I_n)\\
		& = M_{\tau_{1}} (*) M_{\widehat{\sigma}_{1}} (*) \, (e_{m-i_0(\widehat{\sigma})}\otimes I_n) \text{ by } (\ref{eqnreco1_PCh4}) \text{ since }  \left\{\begin{array}{l}
			i_0(\widehat{\sig}) = i_0(\sig) \leq h \text{ and } \widehat{\widehat{\sigma}}_{1} \\
			\text{contains indices from } \{h+2:m\}
		\end{array} \right. \\ 
		& =M_{\tau_{1}} (*) (e_{m-i_0(\widehat{\sigma}_1, \widehat{\sig})}\otimes I_n)   \text{ by Lemma } \ref{gfprptogArbitraryCoeff_PCh3}\\
			&  = e_{m-i_0(\widehat{\sigma}_1, \widehat{\sig})}\otimes I_n  \text{ by } (\ref{eqnrecogfprgsettingp0329sept172045n_PCh2}) \text{ since }  \left\{\begin{array}{l}
				i_0(\widehat{\sigma}_1, \widehat{\sig}) \leq h \text{ and } \tau_{1} \text{ contains indices} \\ \text{ from } \{-m:-(h+2)\}
			\end{array} \right..
	\end{align*}
	Thus  $M_{\tau_{1}} (Y_1) M_{\sigma_{1}}(X_1) \, (e_{m-i_0(\sigma)}\otimes I_n)= e_{m-i_0(\sigma_1, \sig)}\otimes I_n$. This completes the proof.
\end{proof}

The following proposition is a restatement of \cite[Theorem~3.6 and Theorem~3.9]{rafiran4} which shows that the Fiedler pencils are Rosenbrock strong linearization of $G(\lam)$. 

\begin{proposition} \label{prop:gfprstlin} Let $T(\lam) := \lam M^P_{-m} -M^P_{\alpha} $ be the Fiedler pencil of $P(\lam)$ associated with a permutation $\alpha$ of $\{0:m-1\}$. Then $\mathbb{T}(\lam) :=\left[
	\begin{array}{@{}c|c@{}}
		T(\lam) &  e_{m-i_0 (\alpha)} \otimes C \\[.1em] \hline \\[-1em]
		e^T_{m-c_0(\alpha)} \otimes B  & A-\lam E
	\end{array}
	\right]$ 
is the Fiedler pencil of $G(\lam)$ associated with  $\alpha$. Moreover, $\mathbb{T}(\lam)$ is a Rosenbrock strong linearization of $G(\lam)$. Further, a pencil $\mathbb{L} (\lam) $ given by $\mathbb{L} (\lam) : = \diag (\mathcal{X}, X_0)   \mathbb{T}(\lam)  \diag (\mathcal{Y}, Y_0)  $, where $\mathcal{X}, \mathcal{Y} \in \mathbb{C}^{mn \times mn}$ and  $X_0, Y_0 \in \mathbb{C}^{r \times r}$ are nonsingular matrices, is a  Rosenbrock strong linearization of $G(\lam)$.
\end{proposition}

%

We are now ready to prove that EGFPs of $G(\lam)$ are Rosenbrock strong linearizations of $G(\lam)$.

\begin{theorem} \label{EGFPofGLin}Let $\mathbb{L}(\lam)$ be an EGFP of $G(\lam)$ given as in (\ref{DefEGFP}). Then $\mathbb{L}(\lam)$ is a Rosenbrock strong linearization of $G(\lam)$.
\end{theorem}

\begin{proof} Let $(\sig,\omega)$ be the permutation of $\{0:m\}$ where $0\in \sig$ and $m\in \omega$. We have $\tau = (-\beta, -m, -\gamma )$ for some permutations $\beta$ and $\gamma$. Define $\alpha := (rev(\beta), \sig,rev(\gamma))$. Then $\alpha$ is a permutation of $\{0:m-1\}$ and $T(\lam) := \lam M^P_{-m} -M^P_{\alpha} $ is a Fiedler linearization of $P(\lam)$. This implies that 
$\mathbb{T}(\lam) :=\left[
	\begin{array}{@{}c|c@{}}
		T(\lam) &  e_{m-i_0 (\alpha)} \otimes C \\[.1em] \hline \\[-1em]
		e^T_{m-c_0(\alpha)} \otimes B  & A-\lam E
	\end{array}
\right]$ is a Rosenbrock strong linearization of  $G(\lam) $ follows from Proposition~\ref{prop:gfprstlin}. Note that  $\mathcal{A} := M_{(\tau_1,\sig_1)}(Y_1, X_1)  M^P_{-\beta}$ and  $\mathcal{B} :=
  M^P_{-\gamma} M_{(\sig_2,\tau_2)}(X_2, Y_2)$ are nonsingular matrices.  Hence by Proposition~\ref{prop:gfprstlin}, we only need to show that 
$\diag(\mathcal{A},I_r) \mathbb{T}(\lam) \diag(\mathcal{B},I_r) =\mathbb{L}(\lam).$  Since $\diag(\mathcal{A},I_r) \mathbb{T}(\lam) \diag(\mathcal{B},I_r)=$
\begin{align*}
	&\left[
	\begin{array}{@{}c|c@{}}
		M_{(\tau_1,\sig_1)}(Y_1, X_1)  M^P_{-\beta} & \\[.1em] \hline \\[-1em]  & I_r
	\end{array}
	\right] \mathbb{T}(\lam)
	\left[
	\begin{array}{@{}c|c@{}}
	M^P_{-\gamma} M_{(\sig_2,\tau_2)}(X_2, Y_2)	 & \\[.1em] \hline \\[-1em]  & I_r
	\end{array}
	\right]\\
	&= \left[
	\begin{array}{@{}c|c@{}}
		T(\lam)  & 	M_{(\tau_1,\sig_1)}(Y_1, X_1)  M^P_{-\beta} (e_{m-i_0 (\alpha)} \otimes C)\\[.1em] \hline \\[-1em] (e^T_{m-c_0 (\alpha)} \otimes B)	M^P_{-\gamma} M_{(\sig_2,\tau_2)}(X_2, Y_2) & A-\lam E
	\end{array}
	\right],
\end{align*}
the remaining we need to prove is that $(e^T_{m-c_0 (\alpha)} \otimes I_n) M^P_{-\gamma} M_{(\sig_2,\tau_2)}(X_2, Y_2)=e^T_{m-c_0 (\sigma, \sigma_{2})} \otimes I_n$ and $	M_{(\tau_1,\sig_1)}(Y_1, X_1)  M^P_{-\beta}  (e_{m-i_0 (\alpha)} \otimes I_n) = e_{m-i_0 (\sig_{1},\sigma)} \otimes I_n$. From Lemma~\ref{blk_rows_n_col} we have 
 $(e^T_{m-c_0 (\sigma)} \otimes I_n)  M_{(\sig_2,\tau_2)}(X_2, Y_2)=e^T_{m-c_0 (\sigma, \sigma_{2})} \otimes I_n $ and $	M_{(\tau_1,\sig_1)}(Y_1, X_1)  (e_{m-i_0 (\sigma)} \otimes I_n) = e_{m-i_0 (\sigma_1,\sig_2)} \otimes I_n$. Hence the only thing left to show is that $(e^T_{m-c_0 (\alpha)} \otimes I_n) M^P_{-\gamma} = e^T_{m-c_0 (\sigma)} \otimes I_n $ and $M^P_{-\beta}  (e_{m-i_0 (\alpha)} \otimes I_n) = e_{m-i_0 (\sigma)} \otimes I_n$. We proceed as follows.

We have $0 \neq \beta$ (since $0\in \sig$). This implies $c_0(\alpha) = c_0(rev(\beta), \sig, rev(\gamma)) = c_0(\sig, rev(\gamma)$. Let $c_0(\sig)=k$. Now we have two cases. Case-I: suppose that $k+1 \in rev(\gamma)$, and Case-II: suppose that $k+1 \notin rev(\gamma)$. 

Case-I: Let $k+1 \in rev(\gamma)$ and $c_{k+1}(rev(\gamma)) =\ell$. Then by the definition of consecutions, we have $k+1, k+2, \ldots, k+1+\ell$ is a subtuple of $rev(\gamma)$ and $k+1, k+2, \ldots, k+1+\ell,k+\ell+2 $ is not a subtuple of $rev(\gamma)$. This implies that $rev(\gamma) \sim (\delta, k+1, k+2, \ldots, k+1+\ell)$ and hence $M^P_{rev(\gamma)}=M_{(\delta, k+1, k+2, \ldots, k+1+\ell)}^P $, where $\delta$ is the subpermutation of $rev(\gamma)$ such that $k+1,k+2, \ldots, k+1+\ell \notin \delta.$ Consequently, we have $c_0(\alpha)= k+1+\ell$. Now $(e^T_{m-c_0 (\alpha)} \otimes I_n) M^P_{-\gamma}= (e^T_{m- (k+1+\ell)} \otimes I_n) M^P_{-\gamma}=$
\begin{align*}
	&~~~~ (e^T_{m- (k+1+\ell)} \otimes I_n)  M^P_{(-(k+1+\ell), -(k+\ell), \ldots,  -(k+2),-(k+1))} M^P_{-rev(\delta)}\\
	&=(e^T_{m- k} \otimes I_n) M^P_{-rev(\delta)} ~\text{by applying (\ref{eqnrecogfprp0329sept172046n_PCh2}) repetatively} \\
	&= e^T_{m- k} \otimes I_n ~ \text{follows from (\ref{eqnrecogfprp0329sept172046n_PCh2}) since~} -k, -(k+1) \notin -rev(\delta)\\
	& =  e^T_{m- c_0(\sig)} \otimes I_n.
\end{align*}

%
	
Case-II: Suppose that $k+1 \notin rev(\gamma)$. Then it follows from the definition of consecutions that $c_0(\alpha) = c_0(\sig)=k$. Consequently,  $-k, -(k+1) \notin -\gamma$ since $k\in \sig$. So by (\ref{eqnrecogfprp0329sept172046n_PCh2}) we have $(e^T_{m-c_0 (\alpha)} \otimes I_n) M^P_{-\gamma} = (e^T_{m-k} \otimes I_n) M^P_{-\gamma} = e^T_{m-k} \otimes I_n =e^T_{m- c_0(\sig)} \otimes I_n.$ 

A similar proof holds for $M^P_{-\beta}  (e_{m-i_0 (\alpha)} \otimes I_n) = e_{m-i_0 (\sigma)} \otimes I_n$ by using inversions instead of consecutions and (\ref{eqnrecogfprgsettingp0329sept172045n_PCh2}). This completes the proof. 
\end{proof}



We end this section by characterizing EGFPs of $G(\lam)$  having block tridiagonal and block penta-diagonal form. We define the block-tridiagonal and block penta-diagonal system matrices as follows.

\begin{definition}
	Let $\mathcal{A}(\lam) :=\left[
	\begin{array}{c|c}
		\mathcal{X}- \lam \mathcal{Y} & \mathcal{C} \\
		\hline
		\mathcal{B} &  H - \lam K \\
	\end{array}
	\right] $ be an $(mn+r)\times (mn+r)$ matrix pencil, where $\mathcal{X}=(X_{ij})$ and $\mathcal{Y} = (Y_{ij})$ are $m \times m$ block matrices with $X_{ij}, Y_{ij} \in \mathbb{C}^{n\times n}$, $\mathcal{C} \in \mathbb{C}^{mn \times r}, \mathcal{B} \in \mathbb{C}^{r \times mn}$ and $ H, K \in \mathbb{C}^{r \times r}$. Then $\mathcal{A} (\lam)$ is said to be  block-tridiagonal if   $\mathcal{X}$ and $ \mathcal{Y}$ are block tridiagonal matrices and $(e^T_k \otimes I_n)\mathcal{C} =0$ and $\mathcal{B}(e_k \otimes I_n) =0$ for $k=1:m-1.$ Similarly, $\mathcal{A} (\lam)$ is said to be  block penta-diagonal if  $\mathcal{X}$ and $ \mathcal{Y}$ are block penta-diagonal matrices and $(e^T_k \otimes I_n)\mathcal{C} =0$ and $\mathcal{B}(e_k \otimes I_n) =0$ for $k=1:m-2.$ 
\end{definition}

%
%

Observe that the EGFP $\mathbb{L}(\lam)$ given in Example~\ref{exaEGFPG} is a block penta-diagonal linearization of $G(\lam)$. The following results which follow immediately from Theorem~\ref{thm:blk_tri} and Theorem~\ref{thm:blk_penta}, respectively, characterize bock tridiagonal and block penta-diagonal EGFPs of $G(\lam)$  

\begin{corollary} Let $\mathbb{L}(\lambda)$ be an EGFP of $G(\lam)$ as given in Definition~\ref{EGFofG}. Then $\mathbb{L}(\lambda)$  is block tridiagonal  if and only if $ c_t(\sig_1, \sig, \sig_2) = 0 $, $ i_t(\sig_1, \sig, \sig_2) =0$,  $ c_{-t}(\tau_1, \tau, \tau_2) = 0 $ and $ i_{-t}(\tau_1, \tau, \tau_2) =0$ for any index $1 \leq t \leq m-1$.
\end{corollary}

\begin{corollary}  Let $\mathbb{L}(\lambda)$ be an EGFP of $G(\lam)$ as given in Definition~\ref{EGFofG}. Suppose that $\sig_j$ (resp., $\tau_j$), for $j=1,2$, does not contain the end indices of $\sig$ (resp., $\tau$).
	Then $\mathbb{L}(\lambda)$ is block penta-diagonal if and only if $ c_t(\sig_1, \sig, \sig_2) \leq 1$, $ i_t(\sig_1, \sig, \sig_2) \leq 1$, $ c_{-t}(\tau_1, \tau, \tau_2) \leq 1$ and $ i_{-t}(\tau_1, \tau, \tau_2) \leq 1$ for any index $1 \leq t \leq m-1$. 
\end{corollary}



\section{Recovery of minimal bases and minimal indices}\label{recoAll}

This section is devoted in describing the recovery of eigenvectors, minimal bases and minimal indices of $P(\lambda)$ and $G(\lam)$ from those of the linearizations constructed in Section~\ref{EGFPsOfP} and Section~\ref{EGFPofG}, respectfully.

When $G(\lam)$ is singular, the {\em{right null space}} $\mathcal{N}_r(G) $ and the {\em{left null space}} $\mathcal{N}_l(G)$ of $G(\lam)$ are given by
\beano \mathcal{N}_r(G) &:=& \lbrace x(\lambda)\in\mathbb{C}(\lambda)^n :G(\lambda)x(\lambda) = 0 \rbrace   \subset \C(\lam)^n, \\
\mathcal{N}_l(G) &:=&\lbrace y(\lambda)\in\mathbb{C}(\lambda)^m:y(\lambda)^T G(\lambda) = 0\rbrace \subset \C(\lam)^m.\eeano
Let $ \mathcal{B} := \big( x_1(\lam), \ldots, x_p(\lam)\big)$ be a polynomial basis~\cite{kailath, forney}  of $\mathcal{N}_r(G)$ ordered so that $\deg(x_1) \leq \cdots \leq \deg(x_p),$ where $x_1(\lam), \ldots, x_p(\lam)$ are vector polynomials, that is, are elements of $\C[\lam]^n.$ Then $\mathrm{Ord}(\mathcal{B}) := \deg(x_1) + \cdots + \deg(x_p)$ is called the \textit{order} of the basis $\mathcal{B}.$  A basis $\mathcal{B}$ is said to be a minimal polynomial basis~\cite{kailath} of $\mathcal{N}_r(G)$ if $\mathcal{E}$ is any polynomial basis of $\mathcal{N}_r(G)$ then $ \mathrm{Ord}(\mathcal{E}) \geq \mathrm{Ord}(\mathcal{B}).$ A minimal polynomial basis $ \mathcal{B} := \big( x_1(\lam), \ldots, x_p(\lam)\big)$ of $\mathcal{N}_r(G)$ with $\deg(x_1) \leq \cdots \leq \deg(x_p)$ is called a {\em right minimal basis} of $G(\lam)$ and $ \deg(x_1) \leq \cdots \leq \deg(x_p)$ are called the {\em{right minimal indices}} of $G(\lam).$ A {\em{left minimal basis}} and the {\em{left minimal indices}} of $G(\lam)$ are defined similarly. See~\cite{kailath, forney} for further details.

We say that a $k\times p$ matrix polynomial $ Z(\lam)$ is a minimal basis if the columns of $Z(\lam)$ form a minimal basis of the subspace of $\C(\lam)^k$ spanned (over the field $\C(\lam)$) by the columns of $Z(\lam).$ 

We need the following result which is a restatement of \cite[Theorem~4.1 and Theorem~4.2]{BDD} describes the recovery of minimal bases and minimal indices of $P(\lambda)$ from those of the GF pencils.

\begin{theorem} \cite{BDD}\label{recoGF} Let $(\sig, \omega)$ be a permutation of $\{0:m\}$ such that $0\in \sig$ and $m\in \omega$. Set $\tau :=-\omega$. Let $T(\lambda) := \lambda M^P_{\tau} - M^P_{\sigma}$ be the GF pencil of $P(\lam)$ associated with $ (\sig, \tau)$.  Let $ \tau$ be given by $ \tau :=( \tau_l, -m, \tau_r)$. Set $ \alpha:= (- rev(\tau_l), \sigma, -rev(\tau_r))$. Let $c(\alpha) $ and $ i(\alpha)$, respectively,  be the total number of consecutions and inversions of the permutation $\alpha$ of $\{0:m-1\}$. Consider $F^{\scalebox{.4}{GF}
	}(P)  = e^T_{m-c_0(\sigma)} \otimes I_n$ and $K^{\scalebox{.4}{GF}}(P)  = e^T_{m-i_0(\sigma)} \otimes I_n$. Let $Z(\lam)$ be an $mn\times p$ matrix polynomial. Then we have the following.

  If $Z(\lam)$ is a right (resp., left) minimal basis of $T(\lambda)$ then $[F^{\scalebox{.4}{GF}
 }(P) Z(\lam)]$ (resp., $[K^{\scalebox{.4}{GF}
}(P) Z(\lam)]$) is a right (resp., left) minimal basis of $P(\lambda).$ Further, if  $\varepsilon_1 \leq \cdots \leq\varepsilon_p$ are the right (resp., left) minimal indices of $T(\lambda)$ then  $\varepsilon_1 -i(\alpha) \leq \cdots \leq \varepsilon_p - i(\alpha) $ (resp., $\varepsilon_1 -c(\alpha) \leq \cdots \leq \varepsilon_p - c(\alpha) $) are the right (resp., left)  minimal indices of  $P(\lambda).$  
\end{theorem}

The following result describe the recovery of minimal bases and minimal indices of a singular $P(\lam)$ from those of the EGFPs of $P(\lam).$
			
\begin{theorem} \label{recoverymbdgwmiofOFGFPRPGFpencil26j18_PCh4} Let $ L(\lambda):=   M_{\tau_1} (Y_1)  \, M_{\sigma_1} (X_1) \, (\lambda M^P_{\tau} -  M^P_{\sigma}) \, M_{\sigma_2} (X_2)  \, M_{\tau_2} (Y_2)  $ be an EGFP of $P(\lambda)$ as given in (\ref{EGFPDef}). Suppose that $P(\lambda)$ is singular.  Let $ \tau$ be given by $ \tau :=( \tau_l, -m, \tau_r)$. Let $c_{L}:=c(- rev(\tau_l), \sigma, -rev(\tau_r))$ and $i_{L}:= i(- rev(\tau_l), \sigma, -rev(\tau_r))$, respectively, be the  total number of consecutions and inversions of the permutation $(- rev(\tau_l), \sigma, -rev(\tau_r))$ of $\{ 0:m-1\}$.  Let $Z(\lam)$ be an $mn\times p$ matrix polynomial. Then we have the following.
	
	\begin{enumerate}
		\item[(a)]  {\bf Right and left minimal bases.} If $Z(\lam)$ is a right (resp., left) minimal basis of $L(\lambda)$ then $\big[ (e^T_{m-c_0(\sigma,\sigma_2)} \otimes I_n ) Z(\lam) \big]$ (resp., $\big[(e^T_{m-i_0(\sigma_1,\sigma)} \otimes I_n ) Z(\lam) \big]$) is a right (resp., left) minimal basis of $ P(\lambda)$.
		
		\item[(b)] If $\varepsilon_1\leq \cdots \leq\varepsilon_p$ are the right (resp., left) minimal indices of $L(\lambda)$ then $\varepsilon_1 -i_{L} \leq  \cdots \leq \varepsilon_p - i_{L}$ (resp., $\varepsilon_1 -c_{L} \leq \cdots \leq \varepsilon_p -c_{L}$) are the right (resp., left) minimal indices of  $P(\lambda).$ 
	\end{enumerate}
\end{theorem}

			\begin{proof} We have $L(\lam) = \mathcal{A} \, T(\lambda) \mathcal{B}$, where $\mathcal{A} : = M_{\tau_1} (Y_1) \, M_{\sigma_1} (X_1)$, $\mathcal{B}:= \,M_{\sigma_2} (X_2) \, M_{\tau_2} (Y_2) $, and $ T(\lambda) := \lambda M^P_{\tau} -  M^P_{\sigma}$ is a GF pencil of $P(\lambda)$ associated with the permutation $\omega:=(\sigma,-\tau)$ of $\{0:m\}$ such that $0 \in \sig$ and $-m \in \tau$. Since $X_j$ and $Y_j$, $j=1,2$, are nonsingular matrix assignments, $ \mathcal{A} $ and $ \mathcal{B}$ are nonsingular matrices. Thus, the map $ \mathcal{B } : \mathcal{N}_r(L) \rightarrow \mathcal{N}_r(T),$ $ x(\lambda) \mapsto \mathcal{B } x(\lambda) ,$ is an isomorphism and maps a minimal basis of $\mathcal{N}_r(L) $ to a minimal basis of $\mathcal{N}_r(T)$. On the other hand, by Theorem~\ref{recoGF},
				$ F^{\scalebox{.4}{GF}}(P) : \mathcal{N}_r(T) \rightarrow  \mathcal{N}_r(P), ~y(\lambda) \mapsto  (e^T_{m-c_0(\sigma)} \otimes I_n) y(\lambda)$, is an isomorphism and maps a minimal basis of $\mathcal{N}_r(T)$ to a minimal basis of $\mathcal{N}_r(P)$. Consequently, we have $F^{\scalebox{.4}{GF}}(P) \mathcal{B }:\mathcal{N}_r(L) \rightarrow \mathcal{N}_r(P),  ~z(\lambda) \mapsto  (e^T_{m-c_0(\sigma)} \otimes I_n) \mathcal{B} z(\lambda)$, is an isomorphism and maps a minimal basis of $\mathcal{N}_r(L)$ to a minimal basis of $\mathcal{N}_r(P)$. By Lemma~\ref{blk_rows_n_col}, we have
				$$(e^T_{m-c_0(\sigma)} \otimes I_n) \mathcal{B} = (e^T_{m-c_0(\sigma)} \otimes I_n)  \, M_{\sigma_2} (X_2) M_{\tau_2} (Y_2) = e^T_{m-c_0(\sigma, \sigma_2)} \otimes I_n $$ which proves the recovery of right minimal bases. 
				
				Next, we prove the recovery of left minimal bases. Note that $ \mathcal{A}^T:\mathcal{N}_l(L) \rightarrow \mathcal{N}_l(T),$ $x(\lambda) \mapsto  \mathcal{A}^T x(\lambda), $ is a linear isomorphism and maps a minimal basis of $\mathcal{N}_l(L)$ to a minimal basis of $\mathcal{N}_l(T)$. By Theorem~\ref{recoGF}, $ K^{\scalebox{.4}{GF}}(P) :\mathcal{N}_l(T) \rightarrow \mathcal{N}_l(P), ~ y (\lam) \mapsto(e^T_{m-i_0(\sigma)} \otimes I_n) y(\lam),$ is a linear isomorphism and maps a minimal basis of $\mathcal{N}_l(T)$ to a minimal basis of $\mathcal{N}_l(P)$. Consequently, we have $K^{\scalebox{.4}{GF}} (P)  \,  \mathcal{A}^T  : \mathcal{N}_l(L) \rightarrow \mathcal{N}_l(P), ~ z (\lam) \mapsto(e^T_{m-i_0(\sigma)} \otimes I_n) \mathcal{A}^T z(\lam)$, is an isomorphism and maps a minimal basis $\mathcal{N}_l(L)$ to a minimal basis $\mathcal{N}_l(P)$. By Lemma~\ref{blk_rows_n_col}, we have $(e^T_{m-i_0(\sigma)} \otimes I_n) \mathcal{A}^T = (e^T_{m-i_0(\sigma)} \otimes I_n) \Bigs ( M_{\tau_1} (Y_1) M_{\sigma_1} (X_1) \Bigs)^T =  \Bigs (  M_{\tau_1} (Y_1) M_{\sigma_1} (X_1) (e_{m-i_0(\sigma)} \otimes I_n) \Bigs  )^T = e^T_{m-i_0(\sigma_1, \sigma)} \otimes I_n$ which proves the recovery of left minimal bases.

				As $L(\lambda)$ is strictly equivalent to $T(\lambda)$, the left (resp., right) minimal indices of $L(\lambda)$ and $T(\lambda)$ are the same. Hence the desired results for minimal indices follow from Theorem~\ref{recoGF}.
			\end{proof}

			\subsection{Recovery of eigenvectors} In this section we  describe the recovery of  eigenvectors of $P(\lambda)$ corresponding to an eigenvalue $\mu \in \mathbb{C}$  from those of the EGFPs of $P(\lam).$ The next result is a restatement of \cite[Theorems~3.2]{BDD} which describes the recovery of eigenvectors of $P(\lam)$ from those of the GF pencils.  We present these results for ready reference.
	
	\begin{theorem} \cite{BDD} \label{reco_GFP_PCh1} Suppose that $P(\lambda)$ is regular and $\mu \in \mathbb{C}$ is an eigenvalue of $P(\lambda)$. Let $T(\lambda) := \lambda M^P_{\tau} - M^P_{\sigma}$ be a GF pencil of $P(\lam)$ associated with a permutation $ ({\sigma}, {\omega})$ of $\{0: m\}$, where $\tau:= - \omega.$ Define
		$$ F^{\scalebox{.4}{GF}
		}(P) : =\left\{ \begin{array}{ll}
			e^T_{m-c_0(\sigma)} \otimes I_n  & \mbox{if }~  0 \in \sig~ \text{and}~ c_0(\sigma) < m \\
			
			A^{-1}_m(e^T_{1} \otimes I_n )  & \mbox{if }~ 0 \in \sig ~\text{and}~ c_0(\sigma) = m  \\
			
			e^T_{m-s} \otimes I_n & \mbox{if } \left\{ \begin{array}{l}
				0 \in \omega, \,  i_0(\omega)+1 \in \sig \text{~and~}  \\
				s:= i_0(\omega) + c_{i_0(\omega)+1}(\sigma) +1 < m \end{array} \right. \\ 
			
			A^{-1}_m ( e^T_1 \otimes I_n ) & \mbox{if } \left\{ \begin{array}{l}
				0 \in \omega, \,  i_0(\omega)+1 \in \sig \text{~and~}  \\
				i_0(\omega) + c_{i_0(\omega)+1}(\sigma) +1 = m \end{array} \right. \\   
			
			e^T_{m-i_0(\omega)} \otimes I_n  &  \mbox{if } \begin{array}{l}
				0 \in \omega , \, i_0(\omega) <m ~\mbox{and}~ i_0(\omega)+1 \notin \sigma \end{array}  \\ 
			
			A^{-1}_0 ( e^T_{m} \otimes I_n )  & 
			\mbox{if }~ 0 \in \omega ~ \text{and}~i_0(\omega) =m  \\ 
		\end{array}
		\right.$$ 
		and
		$$ K^{\scalebox{.4}{GF}
		}(P) : =\left\{ \begin{array}{ll}
			e^T_{m-i_0(\sigma)} \otimes I_n  & \mbox{if }~0 \in \sig ~\text{and}~ i_0(\sigma) < m   \\
			
			A^{-T}_m(e^T_{1} \otimes I_n )  & \mbox{if }~0 \in \sig ~\text{and}~i_0(\sigma) = m   \\
			
			e^T_{m-s} \otimes I_n &  \mbox{if } \left\{ \begin{array}{l}
				0 \in \omega, ~ c_0(\omega)+1 \in \sig~\text{and}\\
				s:= c_0(\omega)+i_{c_0(\omega)+1}(\sigma) +1 <m \end{array} \right. \\ 
			
			A^{-T}_m ( e^T_1 \otimes I_n ) & \mbox{if } \left\{ \begin{array}{l}
				0 \in \omega, ~ c_0(\omega)+1 \in \sig~\text{and}\\ c_0(\omega)+i_{c_0(\omega)+1}(\sigma) +1 =m \end{array} \right. \\

			e^T_{m-c_0( \omega)} \otimes I_n  & \mbox{if }  \begin{array}{l}
				0 \in \omega, ~ c_0( \omega)<m~\text{and}~ c_0( \omega)+1 \notin \sigma\end{array} \\ 
			
			A^{-T}_0 ( e^T_{m} \otimes I_n )  & 
			\mbox{if }~ 0 \in \omega ~\text{and}~c_0( \omega) =m.\\ 
		\end{array}
		\right.$$ 
		Define the maps $ F^{\scalebox{.4}{GF}
		}(P) : \mathcal{N}_r(T) \rightarrow \mathcal{N}_r(P),~ x \mapsto F^{\scalebox{.4}{GF}
		}(P) x ,$ and $ K^{\scalebox{.4}{GF}
		}(P) : \mathcal{N}_l(T) \rightarrow \mathcal{N}_l(P),~ y \mapsto K^{\scalebox{.4}{GF}
		}(P) y $. Let Z be an $mn\times p$ matrix such that $\rank(Z)=p$. Then we have the following. 
	
	 If $Z$ is a basis of $\mathcal{N}_r( T(\mu))$ (resp., $\mathcal{N}_l( T(\mu))$), then $\big[F^{\scalebox{.4}{GF}
		}(P) Z \big]$ (resp., $\big[K^{\scalebox{.4}{GF}
	}(P) Z \big]$) is a basis of $\mathcal{N}_r(P(\mu))$ (resp., $\mathcal{N}_l(P(\mu))$).



		
	\end{theorem}

Let $ L(\lambda) = M_{(\tau_1,\sigma_1)} (Y_1,X_1) (\lambda M^P_{\tau} -  M^P_{\sigma}) M_{(\sigma_2,\tau_2)} (X_2,Y_2)$ be an EGFP of $P(\lambda)$. 
Observe that the basic building blocks $\lambda M^P_{\tau} -  M^P_{\sigma}$ of $L(\lam)$ is a GF pencil of $P(\lam)$. In order to derive eigenvector recovery of $P(\lam)$ from those of $L(\lam)$ we need to compute  $F^{\scalebox{.4}{GF}}(P) \,  M_{(\sigma_2,\tau_2)} (X_2,Y_2)$ and $ M_{(\tau_1,\sigma_1)} (Y_1,X_1) \, (K^{\scalebox{.4}{GF}})^T$ for the various cases of $F^{\scalebox{.4}{GF}}(P)$ and $ K^{\scalebox{.4}{GF}}(P)$ given as in Theorem~\ref{reco_GFP_PCh1}. We only state the outcomes of these computations in the following results for ease reading of the paper as the proofs are not important to the development that follows. We refer to Appendix~\ref{appendix_reco} for the proofs.


%

\begin{lemma} \label{egfpr_right_reco1}
	Let $ L(\lambda):= M_{\tau_1} (Y_1) M_{\sigma_1} (X_1) (\lambda M^P_{\tau} -  M^P_{\sigma}) M_{\sigma_2} (X_2)M_{\tau_2} (Y_2)$ be an EGFP of $P(\lambda)$ as given in (\ref{EGFPDef}). Suppose that $0 \in \sigma$. Then we have the following.
	\begin{enumerate}
		\item [(a)] If  $c_0(\sig) < m$ then  $	(e_{m-c_0(\sigma)}^{T}\otimes I_n) \, M_{\sigma_{2}} (X_2) \,  M_{\tau_{2}}  (Y_2) = e_{m-c_0(\sigma,\sigma_2)}^{T}\otimes I_n $. Similarly, if $i_0(\sig) < m$ then $M_{\tau_{1}} (Y_1)\, M_{\sigma_{1}} (X_1) \, (e_{m-i_0(\sigma)}\otimes I_n) = e_{m-i_0(\sigma_1, \sigma)}\otimes I_n.$
		
		\item [(b)] If $c_0(\sig) =m$ then $ (e_{1}^{T}\otimes I_n) M_{\sigma_{2}} (X_2) M_{\tau_{2}}  (Y_2)=  e_{1}^{T}\otimes I_n,$ and if $i_0(\sig) =m$ then $M_{\tau_{1}} (Y_1) M_{\sigma_{1}} (X_1) (e_{1}\otimes I_n) =e_{1}\otimes I_n.$
	\end{enumerate}
	
\end{lemma}

\begin{lemma} \label{egfpr_right_reco}
	Let $ L(\lambda):= M_{\tau_1} (Y_1) M_{\sigma_1} (X_1) (\lambda M^P_{\tau} -  M^P_{\sigma}) M_{\sigma_2} (X_2)M_{\tau_2} (Y_2)$ be an EGFP of $P(\lambda)$ as given in (\ref{EGFPDef}). Suppose that $0 \in \omega$ (recall that $\tau = - \omega$). Then we have the following.
	\begin{enumerate}
		\item [(a)] If $i_0(\omega) +1 \in \sig$ and $s:= i_0(\omega) + c_{i_0(\omega)+1}(\sig) +1 < m$ then 
		$$(e^T_{m-s} \otimes I_n) M_{\sigma_{2}} (X_2) M_{\tau_{2}}  (Y_2) = e^T_{m-p} \otimes I_n,  $$
		where $ p:= i_0(\omega) + c_{i_0(\omega)+1}(\sig, \sig_2) +1.$
		
		\item [(b)] If $s =m$ in part (a), then $(e_1^T \otimes I_n)  M_{\sigma_{2}} (X_2) M_{\tau_{2}}  (Y_2) = e_p^T \otimes I_n,$
		where $p := c_{-(m-1)} (\tau_2) +2.$	
		\item [(c)] If $i_0(\omega) <m$ and $i_0(\omega) +1 \notin \sig$ then 
		\begin{eqnarray} \label{eqn:recoev2CH3}
			(e^T_{m-i_0(\omega)} \otimes I_n)  M_{\sigma_{2}} (X_2) M_{\tau_{2}}  (Y_2) = e^T_{m-p} \otimes I_n,
		\end{eqnarray}
		where $p:= i_0(\omega) -c_{-i_0(\omega)}(\tau_2)-1.$
		
		\item [(d)] If $i_0(\omega) =m$ in part (c), then $(e^T_m \otimes I_n) M_{\sigma_{2}} (X_2) M_{\tau_{2}}  (Y_2) =e^T_m \otimes I_n. $
	\end{enumerate}
	
\end{lemma}

The following result is analogous to Lemma~\ref{egfpr_right_reco}. 

\begin{lemma} \label{egfpr_left_reco}
	Let $ L(\lambda):= M_{\tau_1} (Y_1) M_{\sigma_1} (X_1) (\lambda M^P_{\tau} -  M^P_{\sigma}) M_{\sigma_2} (X_2)M_{\tau_2} (Y_2)$ be an EGFP of $P(\lambda)$ as given in (\ref{EGFPDef}). Suppose that $0 \in \omega$ (recall that $\tau = - \omega$). Then we have the following.

	\begin{enumerate}
		\item [(a)] If $c_0(\omega) +1 \in \sig$ and $s:= c_0(\omega) + i_{c_0(\omega)+1}(\sig) +1 < m$ then 
		$$  M_{\tau_{1}}  (Y_1) M_{\sigma_{1}} (X_1) (e_{m-s} \otimes I_n)= e_{m-p} \otimes I_n,  $$
		where $ p:= c_0(\omega) + i_{c_0(\omega)+1}(\sig_1, \sig) +1.$
		
		\item [(b)] If $s =m$ in part (a), then $   M_{\tau_{1}}  (Y_1)  M_{\sigma_{1}} (X_1) (e_1 \otimes I_n)= e_p \otimes I_n,$
		where $p := i_{-(m-1)} (\tau_1) +2.$
		
		\item [(c)] If $c_0(\omega) <m$ and $c_0(\omega) +1 \notin \sig$ then 
		\begin{equation} \label{eqn:recoevCH3}
			M_{\tau_{1}}  (Y_1) M_{\sigma_{1}} (X_1) (e_{m-c_0(\omega)} \otimes I_n) = e_{m-p} \otimes I_n,
		\end{equation}
		where $p:= c_0(\omega) -i_{-c_0(\omega)}(\tau_1) -1. $
		
		\item [(d)] If $c_0(\omega) =m$ in part (c), then $M_{\tau_{1}}  (Y_1)  M_{\sigma_{1}} (X_1)  (e_m \otimes I_n) =e_m \otimes I_n. $
	\end{enumerate}
	
\end{lemma}
			
We are now ready to describe the  recovery of eigenvectors of a regular $P(\lam)$ from those of the EGFPs of $P(\lam)$.

\begin{theorem} \label{non_OF_EGFPR} Suppose that $P(\lambda)$ is regular and $\mu \in \mathbb{C}$ is an eigenvalue of $P(\lambda)$. Let $ L(\lambda):= M_{\tau_1} (Y_1) M_{\sigma_1} (X_1) (\lambda M^P_{\tau} -  M^P_{\sigma}) M_{\sigma_2} (X_2)M_{\tau_2} (Y_2)$ be an EGFP of $P(\lambda)$ as given in (\ref{EGFPDef}). (Recall that $\omega = - \tau$).  Let Z be an $mn\times p$ matrix such that $\rank(Z)=p$. Define
	$$ F^{\scalebox{.4}{EGFP}
	}(P) : =\left\{ \begin{array}{ll}
		e^T_{m-c_0(\sigma, \sig_2)} \otimes I_n  & \mbox{if }~  0 \in \sig~ \text{and}~ c_0(\sigma) < m \vspace{.1cm} \\ 
		
		A^{-1}_m(e^T_{1} \otimes I_n )  & \mbox{if }~ 0 \in \sig ~\text{and}~ c_0(\sigma) = m  \vspace{.1cm} \\
		
		e^T_{m-p} \otimes I_n & \mbox{if } \left\{ \begin{array}{l}
			0 \in \omega, \,  i_0(\omega)+1 \in \sig \text{~and~}  \\
			s:= i_0(\omega) + c_{i_0(\omega)+1}(\sigma) +1 < m, \\
			\text{where } p:= i_0(\omega) + c_{i_0(\omega)+1}(\sigma, \sig_2) +1\end{array} \right. \vspace{.1cm}  \\ 
		
		A^{-1}_m ( e^T_p \otimes I_n ) & \mbox{if } \left\{ \begin{array}{l}
			0 \in \omega, \,  i_0(\omega)+1 \in \sig \text{~and~}  \\
			i_0(\omega) + c_{i_0(\omega)+1}(\sigma) +1 = m ,\\
			\text{ where } p := c_{-(m-1)} (\tau_2) +2\end{array} \right. \vspace{.1cm}  \\   
		
		e^T_{m-p} \otimes I_n  &  \mbox{if } \left\{ \begin{array}{l}
			0 \in \omega , \, i_0(\omega) <m ~\mbox{and}~ i_0(\omega)+1 \notin \sigma,\\
			\text{where } p:=  i_0(\omega) - c_{-i_0(\omega)}(\tau_2) -1\end{array} \right. \vspace{.1cm} \\ 
		
		A^{-1}_0 ( e^T_{m} \otimes I_n )  & 
		\mbox{if }~ 0 \in \omega ~ \text{and}~i_0(\omega) =m 
	\end{array}
	\right.$$ 
	and
	$$ K^{\scalebox{.4}{EGFP}
	}(P) : =\left\{ \begin{array}{ll}
		e^T_{m-i_0(\sigma_1,\sig)} \otimes I_n  & \mbox{if }~0 \in \sig ~\text{and}~ i_0(\sigma) < m  \vspace{.1cm}  \\
		
		A^{-T}_m(e^T_{1} \otimes I_n )  & \mbox{if }~0 \in \sig ~\text{and}~i_0(\sigma) = m  \vspace{.1cm}  \\
		
		e^T_{m-p} \otimes I_n &  \mbox{if } \left\{ \begin{array}{l}
			0 \in \omega, ~ c_0(\omega)+1 \in \sig~\text{and}\\
			s:= c_0(\omega)+i_{c_0(\omega)+1}(\sigma) +1 <m, \\
			\text{where } p:= c_0(\omega) + i_{c_0(\omega)+1}(\sigma_1, \sig) +1
		\end{array} \right. \vspace{.1cm} \\ 
		
		A^{-T}_m ( e^T_p \otimes I_n ) & \mbox{if } \left\{ \begin{array}{l}
			0 \in \omega, ~ c_0(\omega)+1 \in \sig~\text{and}\\ c_0(\omega)+i_{c_0(\omega)+1}(\sigma) +1 =m ,\\
			\text{ where } p := i_{-(m-1)} (\tau_1) +2\end{array} \right. \vspace{.1cm}  \\

		e^T_{m-p} \otimes I_n  & \mbox{if }  \left\{ \begin{array}{l}
			0 \in \omega, ~ c_0( \omega)<m~\text{and}~ c_0( \omega)+1 \notin \sigma,\\
			\text{where } p:=  c_0(\omega) - i_{-c_0(\omega)}(\tau_1) -1 \end{array} \right. \vspace{.1cm} \\ 
		
		A^{-T}_0 ( e^T_{m} \otimes I_n )  & 
		\mbox{if }~ 0 \in \omega ~\text{and}~c_0( \omega) =m.
	\end{array}
	\right.$$ 
	Then the maps $ F^{\scalebox{.4}{EGFP}
	}(P) : \mathcal{N}_r(L(\mu)) \rightarrow \mathcal{N}_r(P(\mu)),~ x \mapsto F^{\scalebox{.4}{EGFP}
	}(P) x, $ and $ K^{\scalebox{.4}{EGFP}
	}(P) : \mathcal{N}_l(L(\mu)) \rightarrow \mathcal{N}_l(P(\mu)),~ y \mapsto K^{\scalebox{.4}{EGFP}
	}(P) y ,$ are linear isomorphisms. Thus, if $Z$ is a basis of $\mathcal{N}_r( L(\mu))$ (resp., $\mathcal{N}_l( L(\mu))$), then $\big[F^{\scalebox{.4}{EGFP}
}(P) Z \big]$ (resp., $\big[K^{\scalebox{.4}{EGFP}
}(P) Z \big]$) is a basis of $\mathcal{N}_r(P(\mu))$ (resp., $\mathcal{N}_l(P(\mu))$).

%
%
	
\end{theorem}

\begin{proof}  We have $L(\lam) = M_{\tau_1} (Y_1) M_{\sigma_1} (X_1) T(\lambda) M_{\sigma_2} (X_2) M_{\tau_2} (Y_2)$, where $ T(\lambda) := \lambda M^P_{\tau} -  M^P_{\sigma}$ is a GF pencil of $P(\lambda)$ associated with the permutation $(\sigma,-\tau)$ of $\{0:m\}$. Since $X_j$ and $Y_j$, $j=1,2,$ are nonsingular matrix assignments, $M_{(\tau_1,\sigma_1) } (Y_1,X_1)$ and $M_{(\sigma_2,\tau_2) } (X_2,Y_2)$ are nonsingular. Hence the map $ M_{(\sigma_2,\tau_2)} (X_2,Y_2): \mathcal{N}_r(L(\mu)) \rightarrow \mathcal{N}_r(T(\mu)),$ $ x \mapsto \big(M_{(\sigma_2,\tau_2)} (X_2,Y_2)\big) x, $ is a linear isomorphism. On the other hand, by Theorem \ref{reco_GFP_PCh1},
	$ F^{\scalebox{.4}{GF}}(P) : \mathcal{N}_r(T(\mu)) \rightarrow  \mathcal{N}_r(P(\mu))$ is a linear isomorphism. Thus the map
	$$\mathcal{N}_r(L(\mu)) \rightarrow \mathcal{N}_r(P(\mu)) ,~ x \mapsto \big( F^{\scalebox{.4}{GF}}(P) M_{(\sigma_2,\tau_2) } (X_2,Y_2)\big) x,
	$$ 
	is a linear isomorphism. Now the desired result for recovery of right eigenvectors follows from Theorem~\ref{reco_GFP_PCh1} and Lemmas~\ref{egfpr_right_reco1} and \ref{egfpr_right_reco}.

	Next we prove the recovery of left eigenvectors. As $(M_{(\sigma_1,\tau_1)} (Y_1,X_1))^T$ is nonsingular, $(M_{(\sigma_1,\tau_1)} (Y_1,X_1) )^T:\mathcal{N}_l(L(\mu)) \rightarrow \mathcal{N}_l(T(\mu)),$ $y \mapsto (M_{(\sigma_1,\tau_1)} (Y_1,X_1))^T y, $ is a linear isomorphism. Also, by Theorem \ref{reco_GFP_PCh1}, $ K^{\scalebox{.4}{GF}}(P) :\mathcal{N}_l(T(\mu)) \longrightarrow \mathcal{N}_l(P(\mu))$ is a linear isomorphism. Thus the map
	$$ \mathcal{N}_l(L(\mu)) \rightarrow \mathcal{N}_l(P(\mu)), \, y  \mapsto \Big ( K^{\scalebox{.4}{GF}} (P) (M_{(\sigma_1,\tau_1)} (Y_1,X_1))^T y\Big ),$$ is a linear isomorphism. Now the desired result for recovery of left eigenvectors follows from Theorem~\ref{reco_GFP_PCh1} and Lemmas~\ref{egfpr_right_reco1} and \ref{egfpr_left_reco}.
\end{proof}

\vspace{.5cm}
We now illustrate eigenvector recovery rule for $P(\lam)$ from those of the EGFPs of $P(\lam)$ by considering a few examples.

\begin{example}  \label{eg_egfpr2} Let $P(\lambda) := \sum_{i=0}^6 \lambda^i A_i$.  Suppose that $P(\lam)$ is regular and $\mu \in \mathbb{C}$ is an eigenvalue of $P(\lam)$. Let $\sigma:=(1,2,5)$, $\sigma_1:=\emptyset, \sigma_2:=(1),\tau:=(-6,-3,-4,-0), \tau_1 :=(-4)$ and $\tau_2:= \emptyset$. Let $X$ and $Y$ be any arbitrary matrix assignments for $\tau_1$ and $\sig_2$, respectively. Then the pencil
	$L(\lam) =  M_{-4} (X) \big ( \lambda M^P_{(-3,-4,-6,-0)} - M^P_{(1,2,5)} \big ) M_{1} (Y) =$ 
	
	$$ \left[
	\begin{array}{cccccc}
		\lam  A_6  + A_5&        -I_n & 0     &  0           &0 &      0\\
		0 & 0 &  -I_n    &    \lam I_n &           0&      0\\
		-I_n&          0 &   \lam I_n - X &       \lam X   &        0    &   0\\
		0 &           \lam I_n &   \lam A_4 & \lam A_3 + A_2  &   - Y  &  -I_n\\
		0 &       0 &    0 &    A_1  & \lam Y - I_n  &  \lam I_n \\
		0 &          0 &   0 &        -I_n  &      -\lam A_0^{-1}  & 0
	\end{array} \right]
	$$
	is an  EGFP of $P(\lam)$.

	Let $u \in \mathcal{N}_r(L) $ and $v \in \mathcal{N}_l(L)  $. Define $ u_i := (e_i^T \otimes I_n) u$ and $ v_i := (e_i^T \otimes I_n) v$,  $i=1:6.$ Note that $0 \in \omega$ and $ i_0(\omega) =0$. Further, $i_0(\omega)+1 = 1 \in \sig$ and $c_{i_0(\omega)+1}(\sig) =1  $, which implies that $s:=i_0(\omega) + c_{i_0(\omega)+1}(\sigma) +1= 2 < 6.$ As $c_{1}(\sigma, \sig_2) =1$, we have  $ p:= i_0(\omega) + c_{i_0(\omega)+1}(\sigma, \sig_2) +1 = 2 . $ Thus by Theorem~\ref{non_OF_EGFPR}, $(e^T_{6-2} \otimes I_n )u = u_4 \in \mathcal{N}_r(P(\mu)).$ To verify the recovery rule, consider  $L(\lambda) u =0.$ This gives
	\begin{eqnarray}
		(\lam  A_6  + A_5)u_1 - u_2 =0 \label{eg2_1}\\
		-u_3 + \lam u_4  =0 \label{eg2_2}\\
		-u_1 +    (\lam I_n - X)u_3 + \lam X u_4  =0\label{eg2_3}\\
		\lam u_2 + \lam A_4 u_3 + ( \lam A_3 +A_2 ) u_4  - Y u_5 -u_6=0\label{eg2_4}\\
		A_1 u_4 + ( \lam Y -I_n ) u_5 + \lam u_6=0 \label{eg2_5}\\
		- u_4 -\lam A_0^{-1} u_5 = 0\label{eg2_6}
	\end{eqnarray}
	From (\ref{eg2_2}) we have $u_3 =\lambda u_4$. Substituting $u_3 =\lambda u_4$ in (\ref{eg2_3}) we have $u_1 = \lam^2 u_4.$ Adding $\lambda$ times (\ref{eg2_4}) with (\ref{eg2_5}) and then substituting $u_3 =\lambda u_4$  we have $\lam^2 u_2 + (\lam^3 A_4 + \lam^2 A_3 + \lam A_2 +A_1  )u_4 -u_5  = 0$ which together with (\ref{eg2_6}) gives
	\begin{eqnarray} 
		\lam^3 u_2 + (\lam^4 A_4 + \lam^3 A_3 + \lam^2 A_2 + \lam A_1  +A_0)u_4  &= &0 .  \label{eg2_8}
	\end{eqnarray} 
	Now substituting  the value of  $\lam^3 u_2 $ from (\ref{eg2_8}) and $u_1 = \lam^2 u_4 $ in (\ref{eg2_1}), we have $P(\lam) u_4 =0. $

	Next, consider  $v \in \mathcal{N}_l(L)  $.  Define $ v_i := (e_i^T \otimes I_n) v$,  $i=1:6.$ 
	Note that  $0 \in \omega$ and $ c_0(\omega) =0$. Further, $c_0(\omega)+1 = 1 \in \sig$ and $i_{1} (\sig) =0  $ which implies that $s:=c_0(\omega) + i_{c_0(\omega)+1}(\sigma) +1= 1 < 6.$ As $i_{1}(\sigma_1, \sig) =0$, we have  $ p:= c_0(\omega) + i_{c_0(\omega)+1}(\sigma_1, \sig) +1 = 1 . $ Hence by Theorem~\ref{non_OF_EGFPR}, we have $(e^T_{6-1} \otimes I_n )v = v_5 \in \mathcal{N}_l(P(\mu))$ which can be easily verified. $\blacksquare$
\end{example}

Next, we consider an EGFP  which is not operation free (as $-1$ and $-0$ simultaneously belong to $\tau$) but the recovery of eigenvector is operation free.

\begin{example}  \label{eg_egfpr2e} Let $P(\lambda) := \sum_{i=0}^5 \lambda^i A_i$. Suppose that $P(\lam)$ is regular and $\mu \in \mathbb{C}$ is an eigenvalue of $P(\lam)$. Let $\sigma:=(4,2,3)$, $\sigma_1:=\emptyset, \sigma_2:=(2),\tau:=(-5,-1,-0)$ and $ \tau_1= \emptyset= \tau_2$. Let $X$ be any arbitrary matrix assignment for $\sig_2$. Then the  EGFP
	$L(\lam) =   \big ( \lambda M^P_{(-5,-1,-0)} - M^P_{(4,2,3)} \big ) M_{2} (X)$ of $P(\lam)$ is given by
	$$L(\lam) =  \left[
	\begin{array}{ccccc}
		\lam A_5  + A_4 &  A_3 &           -X &  -I_n &   0\\
		-I_n &  \lam I_n &             0 &   0 &       0 \\
		0 &  A_2 &  \lam  X - I_n & \lam I_n &      0\\
		0 &  -I_n        &     0 &   0&           -\lam A^{-1}_0\\
		0 &   0 &          \lam I_n   & 0 & - \lam A_1 A^{-1}_0 - I_n
	\end{array} \right].
	$$
	
	%
	%
	
	
	Let $u \in \mathcal{N}_r(L) $ and $v \in \mathcal{N}_l(L)  $. Define $ u_i := (e_i^T \otimes I_n) u$ and $ v_i := (e_i^T \otimes I_n) v$,  $i=1:5.$ Note that $0 \in \omega$ and $ i_0(\omega) =1$. Further, $i_0(\omega)+1 = 2 \in \sig$ and $c_{i_0(\omega)+1}(\sig) =1  $, which imply that $s:=i_0(\omega) + c_{i_0(\omega)+1}(\sigma) +1= 3 < 5.$ As $c_{2}(\sigma, \sig_2) =1$, we have  $ p:= i_0(\omega) + c_{i_0(\omega)+1}(\sigma, \sig_2) +1 = 3. $ Thus by Theorem~\ref{non_OF_EGFPR}, $(e^T_{5-3} \otimes I_n )u = u_2 \in \mathcal{N}_r(P(\mu))$ which can be easily verified.

	Note that $0 \in \omega$ and $ c_0(\omega) =0 < 5$. Further, $c_0(\omega)+1 = 1 \notin \sig$. As $\tau_1 = \emptyset$, we have $i_{-c_0(\omega)} (\tau_1) = -1$. So $p : = c_0(\omega) - i_{-c_0(\omega)} (\tau_1) -1 = 0$. Hence  by Theorem~\ref{non_OF_EGFPR},  $(e^T_{5-0} \otimes I_n )v = v_5 \in \mathcal{N}_l(P(\mu))$ which can be easily verified. $\blacksquare$
\end{example}

The EGFP in the following example is not operation free (as $-1$ and $-0$ simultaneously belong to $\tau$), but we can easily recover the eigenvectors of $P(\lam)$ from those of the EGFP.

\begin{example}  \label{eg_NOF_egfpr} Let $P(\lambda) := \sum_{i=0}^3 \lambda^i A_i$. Suppose that $P(\lam)$ is regular and $\mu \in \mathbb{C}$ is an eigenvalue of $P(\lam)$. Let $\sigma:=(3)$, $\sigma_1:=\emptyset= \sigma_2,\tau:=(-2,-1,-0), \tau_1=\emptyset$ and $\tau_2:=(-2) $. Then the EGFP
	$L(\lam) =  \big ( \lambda M^P_{(-2,-1,-0)} - M^P_{3} \big ) M^P_{-2} $ is given by
	$$L(\lam) = \lam  \left[
	\begin{array}{ccr}
		0&         0 & - A_0^{-1} \\
		0 &   I_n  & -  A_2 A_0^{-1}\\
		I_n &   A_2   &  -  A_1 A_0^{-1} 
	\end{array} \right] - \left[
	\begin{array}{ccc}
		0&         A_3^{-1} & 0 \\
		I_n & A_2  & 0\\
		0 &  0   &   I_n
	\end{array} \right]. 
	$$
	%
	%

	Let $x \in \mathcal{N}_r(L(\mu)) $ and $y \in \mathcal{N}_l(L(\mu))   $. Define $ x_i := (e_i^T \otimes I_n) x$ and $ y_i := (e_i^T \otimes I_n) y$ for $i=1:3.$ Note that $0 \in \omega$ and $ i_0(\omega) =2$. Further, $i_0(\omega)+1 = 3 \in \sig$ and $c_{i_0(\omega)+1}(\sig) =0  $, which implies $s:=i_0(\omega) + c_{i_0(\omega)+1}(\sigma) +1= 3$, i.e., $s=m$. Now, $ c_{-(m-1)}(\tau_2) =c_{-2}(\tau_2) =0  $. Thus by Theorem~\ref{non_OF_EGFPR}, $A_m^{-1} (e_2^T \otimes I_n) x =A_3^{-1} x_2 \in \mathcal{N}_r(P(\mu))$ which can be easily verified.  
	
	For left eigenvector, we have  $c_0(\omega) = 0 <m$, $c_0(\omega) +1 =1 \notin \sig$ and $i_{-c_0(\omega)}(\tau_1) = -1$ (as $\tau_1 = \emptyset$). Thus $p = c_0(\omega) - i_{-c_0(\omega)}(\tau_1) -1 =0$. Thus by Theorem~\ref{non_OF_EGFPR}, $ (e_{m-p}^T \otimes I_n) y = y_3 \in \mathcal{N}_l(P(\mu)) $ which can be easily verified.  $\blacksquare$
\end{example}


			

			%

			\subsection{Eigenvalue at infinity and recovery of eigenvectors} This section is devoted in describing the recovery of eigenvectors of $P(\lam)$ corresponding to an eigenvalue at $\infty$ from the eigenvectors of EGFPs. Recall that $\infty$ is an eigenvalue of $P(\lam) = \sum_{i=0}^m \lam^i A_i$ if and only if $0 $ is an eigenvalue of $rev(P(\lam)) =  \sum_{i=0}^m \lam^i A_{m-i}$, i.e., $0$ is an eigenvalue of $A_m$. Thus $A_m$ is singular, which implies that $M^P_{-m}$ is singular, i.e., $M^P_{m}$ does not exist as $M^P_{m} = (M^P_{-m})^{-1}$.  So $-m$ always belongs to $\tau$, that is, $m \notin \sig$. To derive the recovery formulas, we need the following result which is a restatement of \cite[Theorem 3.4]{BDD} that describes the recovery of eigenvectors of $P(\lam)$ corresponding to an eigenvalue at $\infty$ from those of the GF pencils of $P(\lam).$  
		
		\begin{theorem}\cite{BDD}\label{recoGFinfinityevn_PCh2}   Let $P(\lambda)$ be a regular matrix polynomial. Let $T(\lambda) := \lambda M^P_{-\tau} - M^P_{\sigma}$ be a GF pencil of $P(\lam)$ associated with a permutation $ (\sigma, \tau)$ of $\{0:m\}$, where  $m \in \tau.$ Suppose that $\infty$ is an eigenvalue of $P(\lambda)$. Let $Z$ be an $mn\times k$ matrix such that $\rank(Z)=k$.

			\noin
			~~{\bf Right eigenvectors at $\infty$.}	If $Z$ is a basis of the right eigenspace of $T(\lambda)$ at $\infty$, then we have the following.
			\begin{enumerate} 
				\item [(a)] If $c_{-m} (\tau) < m$ then $\big[ (e^T_{c_{-m} (\tau)+1} \otimes I_n) Z \big]$  is a basis of the right eigenspace of $P(\lambda)$ at $\infty$.
				
				\item [(b)] If $c_{-m} (\tau)  = m$ then $\big[ A_0^{-1} (e^T_{m} \otimes I_n) Z \big]$ is a basis of the right eigenspace of $P(\lambda)$ at $\infty$.
			\end{enumerate}	
			
			\noin
			~~{\bf Left eigenvectors at $\infty$.}	If $Z$ is a basis of the left eigenspace of $T(\lambda)$ at $\infty$, then we have the following.
			\begin{enumerate} 	
				
				\item [(c)] If $i_{-m} (\tau) < m$ then $\big[ (e^T_{i_{-m} (\tau)+1} \otimes I_n)Z \big]$ is a basis of the left eigenspace of $P(\lambda)$ at $\infty$.

				\item [(d)]	If $i_{-m} (\tau)  = m$ then  $\big[ A_0^{-T}(e^T_{m} \otimes I_n)Z\big]$is a basis of the left eigenspace of $P(\lambda)$ at $\infty$.
			\end{enumerate}	
		\end{theorem}

The following result follows from the proof of \cite[Theorem~3.21 (see (3.9) and (3.14))]{rafiran3}.
	

\begin{lemma}\label{inf_reco_brbc} Let $0 \leq h \leq m-1,$ and let $\tau$ be a permutation of $\{ -m :-(h+1) \}$. Let $\tau_1$ and $\tau_2$ be index tuples containing indices from $\tau$ such that $(\tau_1,\tau , \tau_2)$ satisfies the SIP. Then  $(e^T_{c_{-m}(\tau) +1} \otimes I_n) M_{\tau_2} (Y) = e^T_{c_{-m}(\tau, \tau_2)+1} \otimes I_n $ and  
	$M_{\tau_1}(X) (e_{i_{-m} (\tau)+1} \otimes I_n) =  e_{i_{-m} (\tau_1,\tau)+1} \otimes I_n $ for any arbitrary matrix assignments $X$ and $Y$ for  $\tau_1$ and $\tau_2 $, respectively.
\end{lemma}

			
Similar to Lemma~\ref{inf_reco_brbc}, we have the following result for EGFPs which will play a crucial  role in the recovery of eigenvectors of $P(\lam)$ corresponding to  an  eigenvalue $\infty$ from those of the  EGFPs.

\begin{lemma} \label{inf_brbc_efpr} Let $ L(\lambda):=  M_{\tau_1} (Y_1)  M_{\sigma_1} (X_1)  (\lambda M^P_{\tau} -  M^P_{\sigma}) M_{\sigma_2} (X_2) M_{\tau_2} (Y_2)  $ be an EGFP  of $P(\lambda)$ such that $-m \in \tau$.
				\begin{enumerate}
					\item [(a)] Suppose that $c_{-m} (\tau) <m$ and $i_{-m} (\tau) < m.$ Then 
					$$ (e^T_{c_{-m}(\tau) +1} \otimes I_n) M_{\sigma_2} (X_2)M_{\tau_2} (Y_2)= e^T_{c_{-m}(\tau, \tau_2)+1} \otimes I_n $$ and  $$ M_{\tau_1} (Y_1)  M_{\sigma_1} (X_1) (e_{i_{-m}(\tau) +1} \otimes I_n)= e_{i_{-m}(\tau_1, \tau)+1} \otimes I_n. $$
					
					\item [(b)] If  $c_{-m} (\tau) =m$ then $(e^T_{m} \otimes I_n) \, M_{\sigma_2} (X_2)M_{\tau_2} (Y_2)=  e^T_{m} \otimes I_n$, and if $i_{-m} (\tau) =m$ then $M_{\tau_1} (Y_1)  M_{\sigma_1} (X_1) \,  (e_{m} \otimes I_n)= e_{m} \otimes I_n. $
				\end{enumerate}
\end{lemma}	
			
			\begin{proof} (a) For an arbitrary $Z \in \mathbb{C}^{ n\times n}$, from (\ref{eqnrecogfprp0329sept172046n_PCh2}) and (\ref{eqnrecogfprgsettingp0329sept172045n_PCh2}) we have
				\begin{eqnarray}
					(e^T_{j} \otimes I_n) M_{-k} (Z)=  e^T_{j} \otimes I_n \mbox{ for } ~k \notin \{m-j,m-j+1\},~ j=1:m, \label{eqnrecogfprp0329sept172046n}\\
					M_{-k} (Z)	(e_{j} \otimes I_n)=  e_{j} \otimes I_n \mbox{ for } ~k \notin \{m-j,m-j+1\},~ j=1:m.\label{eqnrecogfprgsettingp0329sept172045n}
				\end{eqnarray}

				Case-I: Suppose that $\sig \neq \emptyset$.  Let $0\leq  h \leq m-1$ be the integer such that $-m, -(m-1),\ldots, -(m-h) \in \tau$  and $-(m-h-1) \notin \tau$.  This implies that $-(m-h), -(m-h-1) \notin \tau_1 \cup \tau_2$ since  $-(m-h-1) \notin \tau$ and $(\tau_1, \tau, \tau_2)$ satisfies the SIP. Further, we have $c_{-m}(\tau) \leq h$ and $ i_{-m}(\tau)  \leq h$.

				Let $\widehat{\tau} $ and $\widehat{\tau}_j$, $j=1,2$, respectively, be the subtuples of $\tau$ and $\tau_j$ with indices $\{-m:-(m-h)\}$. Similarly, let $\widehat{\widehat{\tau}} $ and $\widehat{\widehat{\tau}}_j$, $j=1,2$, respectively,  be the subtuples of $\tau$ and $\tau_j$ with indices $\{-(m-h-2):-0\}$, where $\{-a:-0\} := \{-a, -(a-1), \ldots, -1,-0\} $ for any integer $a \geq 0$.	Then $\widehat{\tau} $ and $\widehat{\widehat{\tau}} $ commutes since for any indices  $k \in\widehat{\tau} $ and $ \ell \in \widehat{\widehat{\tau}} $, we have $||k| - |\ell|| >1$. Thus $\tau \sim (\widehat{\tau}, \widehat{\widehat{\tau}}) \sim ( \widehat{\widehat{\tau}},\widehat{\tau})$. Similarly, $\tau_j \sim  (\widehat{\tau}_j, \widehat{\widehat{\tau}}_j) \sim ( \widehat{\widehat{\tau}}_j,\widehat{\tau}_j)$ for $j =1,2,$ as $\widehat{\tau_j} $ and $\widehat{\widehat{\tau_j}} $ commutes. Also we have $(\widehat{\tau}, \widehat{\widehat{\tau}}_j) \sim  (\widehat{\widehat{\tau}}_j, \widehat{\tau}) $ for $j =1,2.$

				Since $-(m-h-1) \notin \tau,$ we have $i_{-m}(\tau) = i_{-m}(\widehat{\tau})$ and $c_{-m}(\tau) = c_{-m}(\widehat{\tau})$. As $(\widehat{\tau}, \widehat{\tau}_2)$ contains indices from $\{-m:-(m-h)\}$, we have $c_{-m}(\widehat{\tau}, \widehat{\tau}_2) \leq h$. This shows that $c_{-m}(\widehat{\tau}, \widehat{\tau}_2, \widehat{\widehat{\tau}}_2) = c_{-m}(\widehat{\tau}, \widehat{\tau}_2)$ as $\widehat{\widehat{\tau}}_2 $ contains indices from $\{-(m-h-2): -0\} $. Consequently, we have $c_{-m}(\tau, \tau_2) = c_{-m}(\widehat{\widehat{\tau}},\widehat{\tau}, \widehat{\tau}_2, \widehat{\widehat{\tau}}_2) = c_{-m}(\widehat{\tau}, \widehat{\tau}_2, \widehat{\widehat{\tau}}_2) = c_{-m}(\widehat{\tau}, \widehat{\tau}_2)$. Similarly, we have $i_{-m}(\tau_1, \tau) =  i_{-m}(\widehat{\tau}_1, \widehat{\tau}).$


				Note that $X_2$ and $Y_2$ are arbitrary matrix assignments. We denote by $(*)$ for any arbitrary matrix assignment. Now we have $(e_{c_{-m}(\tau)+1}^{T}\otimes I_n) M_{\sigma_{2}} (*)M_{\tau_{2}} (*)= (e^T_{c_{-m}(\widehat{\tau})+1} \otimes I_n) \, M_{\sig_{2}} (*) \, M_{\widehat{\widehat{\tau}}_{2}} (*) \, M_{\widehat{\tau}_{2}}  (*) =$
				\begin{align*}
					& (e^T_{c_{-m}(\widehat{\tau})+1} \otimes I_n) \, M_{\widehat{\widehat{\tau}}_{2}} (*) \, M_{\widehat{\tau}_{2}}  (*)  \text{ by } (\ref{eqnreco2_PCh4}) \text{ since } \left\{\begin{array}{l}
						c_{-m}(\widehat{\tau})  \leq h \text{ and } \sigma_2 \text{ contains}\\ 
						\text{indices from } \{0:m-h-2 \}  \end{array} \right. \\ 
					& = (e^T_{c_{-m}(\widehat{\tau})+1}\otimes I_n)   \, M_{\widehat{\tau}_{2}}  (*) \text{ by } (\ref{eqnrecogfprp0329sept172046n}) \text{ since } \left\{\begin{array}{l}
						c_{-m}(\widehat{\tau})  \leq h \text{ and } \widehat{\widehat{\tau}}_2 \text{ contains}\\ 
						\text{indices from } \{-(m-h-2): -0\} \end{array} \right. \\ 
					& =  e^T_{c_{-m}(\widehat{\tau}, \widehat{\tau}_2)+1} \otimes I_n  =  e^T_{c_{-m}(\tau, \tau_2)+1} \otimes I_n \text{ by Lemma } \ref{inf_reco_brbc}.
				\end{align*}
				Hence  $(e_{c_{-m}(\tau)+1}^{T}\otimes I_n) M_{\sigma_{2}} (X_2)M_{\tau_{2}} (Y_2) = e^T_{c_{-m}(\tau, \tau_2)+1} \otimes I_n.$
				
				Similarly, we have $M_{\tau_{1}} (*) \,  M_{\sigma_{1}} (*) \, (e_{i_{-m}(\tau)+1} \otimes I_n) =$
				\begin{align*}
					&     M_{\widehat{\tau}_{1}}  (*) \, M_{\widehat{\widehat{\tau}}_1} (*) \,  M_{\sig_{1}} (*) \, (e_{i_{-m}(\widehat{\tau})+1} \otimes I_n)  \\
					& = M_{\widehat{\tau}_{1}} (*) \, M_{\widehat{\widehat{\tau}}_{1}}  (*) \, (e_{i_{-m}(\widehat{\tau})+1} \otimes I_n)    \text{ by } (\ref{eqnreco1_PCh4}) \text{ since } \left\{\begin{array}{l}
						i_{-m}(\widehat{\tau})  \leq h \text{ and } \sigma_1 \text{ contains}\\ 
						\text{indices from } \{0: m-h-2\}  \end{array} \right. \\ 
					& = M_{\widehat{\tau}_{1}} (*) \, (e_{i_{-m}(\widehat{\tau}) +1}\otimes I_n)    \text{ by } (\ref{eqnrecogfprgsettingp0329sept172045n}) \text{ since } \left\{\begin{array}{l}
						i_{-m}(\widehat{\tau})  \leq h \text{ and } \widehat{\widehat{\tau}}_1 \text{ contains}\\ 
						\text{indices from } \{-(m-h-2): -0\} \end{array} \right. \\ 
					& =  e_{i_{-m}(\widehat{\tau}_1, \widehat{\tau}) + 1} \otimes I_n =  e_{i_{-m}(\tau_1, \tau) + 1} \otimes I_n \text{ by Lemma } \ref{inf_reco_brbc}.
				\end{align*}
				Thus, we have  $M_{\tau_{1}} (Y_1) \, M_{\sigma_{1}} (X_1) \, (e_{i_{-m}(\tau)+1} \otimes I_n) = e_{i_{-m}(\tau_1, \tau) + 1} \otimes I_n.$

				Case-II: Suppose that $\sig =  \emptyset$, i.e., $\tau$ is a permutation of $\{-m:-0\}$. This implies  $\sigma_1 = \emptyset = \sig_2$ and hence $M_{\sigma_1} (X_1) = I_{mn} = M_{\sigma_2} (X_2)$. Given that $c_{-m}(\tau) \leq m-1$  and $i_{-m}(\tau) \leq m-1$.  By defining  $\widetilde{\tau} := \tau \setminus \{-0\}$ and $\widetilde{\sig} :=(0)$, and by considering $h=m-1$, a verbatim proof of Case-I  gives the desired result.
				
				%
				%
				%
				
				(b) If  $c_{-m} (\tau) =m$ or  $i_{-m} (\tau) =m$, then $\sig = \emptyset$. So there are no choices for $\sig_1$ and $\sig_2$, i.e., $\sig_1 = \emptyset = \sig_2$. Hence $M_{\sig_1} (X_1) = I_{mn} = M_{\sig_2} (X_2) $. Further, from the definition of EGFPs of $P(\lam)$ we have $\tau_j$, $j=1,2,$ contains indices from  $\{ -m:-2\}$. Consequently, from (\ref{eqnrecogfprp0329sept172046n}) and (\ref{eqnrecogfprgsettingp0329sept172045n}), 
				 we have  $(e^T_{m} \otimes I_n) \, M_{\tau_2} (Y_2)=  e^T_{m} \otimes I_n$ and $M_{\tau_1} (Y_1)  \,  (e_{m} \otimes I_n)= e_{m} \otimes I_n $ which gives the desired result. 
			\end{proof}

			We are now ready to describe the recovery of eigenvectors of $P(\lam)$ corresponding to the eigenvalue at $\infty$ from those of the EGFPs of $P(\lam)$.

			\begin{theorem} \label{inf_reco_egfpr} Let   $ L(\lambda):=  M_{\tau_1} (Y_1)  M_{\sigma_1} (X_1)  (\lambda M^P_{\tau} -  M^P_{\sigma}) M_{\sigma_2} (X_2)  M_{\tau_2} (Y_2)  $ be an EGFP  of $P(\lambda)$.  Suppose that $P(\lambda)$ is regular and $\infty$ is an eigenvalue of $P(\lambda)$.  Let $Z$ be an $mn\times k$ matrix such that $\rank(Z)=k$. 
				
				{\bf Right eigenvectors at $\infty$.} If $Z$ is a basis of the right eigenspace of $L(\lambda)$ at $\infty$, then we have the following.
				\begin{enumerate}
					\item [(a)] If $c_{-m}(\tau) < m$, then $\Bigs[ (e^T_{c_{-m} (\tau, \tau_2)+1} \otimes I_n) Z \Bigs]$ is a basis of the right eigenspace of $P(\lambda)$ at $\infty.$
					\item [(b)]  If $c_{-m} (\tau)  = m$, then $\big[ A_0^{-1} (e^T_{m} \otimes I_n) Z \big]$  is a basis of the right eigenspace of $P(\lambda)$ at $\infty.$
				\end{enumerate}
				
				{\bf Left eigenvectors at $\infty$.} If $Z$ is a basis of the left eigenspace of $L(\lambda)$ at $\infty$, then we have the following.		
				\begin{enumerate}	
					\item [(c)]  If $i_{-m}(\tau) < m$, then   $\Bigs[ (e^T_{i_{-m} (\tau_1, \tau)+1} \otimes I_n) Z \Bigs]$ is a basis of the left eigenspace of $P(\lambda)$ at $\infty$.
					
					\item [(d)] If $i_{-m} (\tau)  = m$, then  $\Bigs [ A_0^{-T}(e^T_{m} \otimes I_n) Z \Bigs]$is a basis of the left eigenspace of $P(\lambda)$ at $\infty$.
				\end{enumerate}

			\end{theorem}

			\begin{proof} Set $ L_1 :=  M_{(\tau_1,\sigma_1)} (Y_1,X_1) ~ M^P_{\tau} ~ M_{(\sigma_2, \tau_2)} (X_2,Y_2)$ and $L_0 :=   M_{(\tau_1,\sigma_1)} (Y_1,X_1) ~ M^P_{\sig}  ~ $ $M_{(\sigma_2, \tau_2)} (X_2,Y_2) $. Then $ L(\lambda)= \lam L_1 - L_0$. Recall that $\infty$ is an eigenvalue of $L(\lambda) \Leftrightarrow 0$ is an eigenvalue of $rev(L(\lambda)) \Leftrightarrow 0$ is an eigenvalue of $L_1$. Since $M_{(\tau_1,\sigma_1)}(Y_1,X_1) $ is invertible, we have $\mathcal{N}_r( L_1) = \mathcal{N}_r(M^P_{\tau}M_{(\sigma_2, \tau_2)} (X_2,Y_2))$. Note that the map 
				$$
				\mathcal{N}_r(M^P_{\tau}M_{(\sigma_2, \tau_2)} (X_2,Y_2)) \rightarrow  \mathcal{N}_r(M^P_\tau),~ z \mapsto (M_{(\sigma_2, \tau_2)} (X_2,Y_2)) z,
				$$ is an isomorphism. Define $T(\lambda):= \lambda M^P_\tau - M^P_\sigma$. Then $\mathcal{N}_r(M^P_\tau)=\mathcal{N}_r \Bigs(rev(T(0))\Bigs)$.

				(a)  Now, by Theorem \ref{recoGFinfinityevn_PCh2}, the map
				$$  \mathcal{N}_r \Bigs (rev(T(0)) \Bigs) \longrightarrow  \mathcal{N}_r \Bigs (rev(P(0))\Bigs),~u \mapsto (e^T_{c_{-m} (\tau)+1} \otimes I_n) u ,$$  is an isomorphism. Hence the map
				\begin{equation*} \label{eqnrightnullatinfn}
					\mathcal{N}_r \Bigs(rev(L(0))\Bigs) \longrightarrow \mathcal{N}_r \Bigs(rev(P(0)) \Bigs), ~x \mapsto \Bigs( (e^T_{c_{-m}(\tau)+1} \otimes I_n)M_{(\sigma_2,\tau_2)} (X_2,Y_2)\Bigs) x ,
				\end{equation*}
				is an isomorphism. Now by Lemma~\ref{inf_brbc_efpr}, we have $(e^T_{c_{-m}(\tau)+1} \otimes I_n)M_{(\sigma_2,\tau_2)} (X_2,Y_2)= e^T_{c_{-m}(\tau,\tau_2)+1} \otimes I_n$.  Hence the result for the right eigenspace of $P(\lambda)$ at $\infty$ follows.
				
				The proof is similar for part (b).

				(c)	Next, we prove the results for left eigenspace of $P(\lambda)$ at $\infty$.  Since $M_{(\sigma_2, \tau_2)} (X_2,Y_2)$ is invertible, we have $\mathcal{N}_l \Bigs(rev(L(0))\Bigs) = \mathcal{N}_l(L_1) = \mathcal{N}_l \Bigs ( M_{(\tau_1,\sigma_1)} (Y_1,X_1) M^P_{\tau} \Bigs )$. Hence it follows that the map $\mathcal{N}_l \Bigs (M_{(\tau_1,\sigma_1)} (Y_1,X_1) M^P_{ \tau} \Bigs ) \longrightarrow \mathcal{N}_l(M^P_\tau),~ z \mapsto (M_{\tau_1, \sigma_1} (Y_1,X_1))^T z,$ is an isomorphism.  We have $\mathcal{N}_l(M^P_{\tau})= \mathcal{N}_l \Bigs(rev(T(0))\Bigs)$. Now, by Theorem \ref{recoGFinfinityevn_PCh2}, the map
				$$  \mathcal{N}_l \Bigs(rev(T(0)) \Bigs) \longrightarrow \mathcal{N}_l \Bigs(rev(P(0)) \Bigs),~v \mapsto (e^T_{i_{-m} (\tau)+1} \otimes I_n) v ,$$ is an isomorphism. Hence the map
				\begin{equation*}\label{eqnleftnullatinfn}
					\mathcal{N}_l \Bigs(rev(L(0)) \Bigs) \longrightarrow \mathcal{N}_l \Bigs(rev(P(0)) \Bigs),~y \mapsto \Bigs( (e^T_{i_{-m} (\tau)+1} \otimes I_n)(M_{(\tau_1,\sigma_1)} (Y_1,X_1))^T \Bigs) y ,
				\end{equation*}
				is an isomorphism. Now by Lemma~\ref{inf_brbc_efpr}, we have $ M_{(\tau_1,\sig_1)} (Y_1,X_1) (e_{i_{-m}(\tau)+1} \otimes I_n) = e_{i_{-m}(\tau_1,\tau)+1} \otimes I_n$.  Hence the result for the left eigenspaces of $P(\lambda)$ at $\infty$ follows.
				
				The proof is similar for part (d).	
			\end{proof}

			%
			
The following example illustrates our recovery rule for eigenvectors of $P(\lam)$ corresponding to an eigenvalue
			at $\infty$ from those of the EGFPs of $P(\lam)$.

			\begin{example}  Let $P(\lambda) := \sum_{i=0}^5 \lambda^i A_i$. Suppose that $P(\lam)$ is regular and $\infty$ is an eigenvalue of $P(\lam)$. Let $\sigma:=(0,2)$, $\sigma_1:=\emptyset= \sigma_2,\tau:=(-4,-5,-3, -1)$, $ \tau_2= (-4)$ and  $\tau_1 =\emptyset$. Let $X$ be any arbitrary matrix assignment for $\tau_2$. Then the  EGFP
				$L(\lam) =   \big ( \lambda M^P_{(-4,-5,-3,-1)} - M^P_{(0,2)} \big ) M_{-4} (X) =: \lam L_1 - L_0$ of $P(\lam)$ is given by
				$$L(\lam) =  \lam \left[
				\begin{array}{ccccc}
					0 &       0 &     I_n &   0 &      0\\
					0 &  A_5 &   A_4 &   0 &      0 \\
					I_n  &   X &   A_3  &  0 &      0\\
					0&      0 &      0&    0&     I_n\\
					0 &       0 &       0 &  I_n  &  A_1
				\end{array} \right] - \left[
				\begin{array}{ccccc}
					0 &       I_n &    0 &   0 &      0\\
					I_n & X &  0 &   0 &      0 \\
					0  & 0 &  -A_2  &  I_n &      0\\
					0&      0 &      I_n&    0&    0\\
					0 &       0 &       0 & 0  &  -A_0
				\end{array} \right].
				$$

				Let $x $ and $y$, respectively, be  right and  left  eigenvectors of $L(\lam)$ corresponding to the eigenvalue $\infty$. Define $ x_i := (e_i^T \otimes I_n) x$ and $ y_i := (e_i^T \otimes I_n) y$,  $i=1:5.$ We have $c_{-m}(\tau) = c_{-5}(-4,-5,-3,-1) =0  < 5$ and $c_{-5} (\tau, \tau_2) = c_{-5} (-4,-5,-3,-1,-4) =1 .$ Hence by Theorem~\ref{inf_reco_egfpr}, $ (e^T_{c_{-m}(\tau,\tau_2) +1 } \otimes I_n ) x = (e^T_{2} \otimes I_n) x = x_2 $ is a right eigenvector of $P(\lam)$ corresponding to the eigenvalue $\infty$. Similarly, $i_{-5} (\tau_1, \tau) = i_{-5} (-4,-5,-3,-1) =1 .$ Hence by Theorem~\ref{inf_reco_egfpr}, $ (e^T_{i_{-m}(\tau_1,\tau) +1 } \otimes I_n ) y = (e^T_{2} \otimes I_n) y = y_2 $ is a left eigenvector of $P(\lam)$ corresponding to the eigenvalue $\infty$.

				To verify the recovery rule, consider  $L_1 x =0.$ This gives $A_5 x_2 = 0$ and $x_i = 0 $ for $i = 3,4,5.$ Now if $x_2 =0$ then  $x_1 =0$. Thus $x_2 =0 \Rightarrow x=0$. Hence $x_2 \neq 0$ and is a right eigenvector of $P(\lam)$ corresponding to the eigenvalue  $\infty$.
				
				Similarly, $y^T L_1 =0$ implies $y_2^T A_5 =0$ and   $y^T_i = 0 $ for $i = 3,4,5.$ Now if $y^T_2 =0$ then  $y^T_1 =0$. Thus  $y^T_2 =0 \Rightarrow y=0$.  Hence $y_2 \neq 0$ and is a left eigenvector of $P(\lam)$ corresponding to the eigenvalue  $\infty$. $\blacksquare$	
			\end{example}

\subsection{Recovery of eigenvectors, minimal bases and minimal indices of $G(\lam)$} We now describe the recovery of eigenvectors, minimal bases and minimal indices of $G(\lam)$ from those of the EGFPs of $G(\lam)$ construed in Section~\ref{EGFPofG}. We proceed as follows. Throughout of this section, we consider $G(\lambda)= P(\lambda)+C(\lambda E-A)^{-1}B$ and $\mathcal{S}(\lam)$ as given in (\ref{minrel2_RCh41}) and (\ref{slamsystemmatrix_RCh41}), respectfully.

\begin{theorem} \cite{rafiran1, Verghese_VK_79} \label{StoG}  	
	(I) Suppose that  $G(\lam)$ is singular. Let $ Z(\lam) :=\left[\begin{array}{c} Z_n(\lam) \\ Z_r(\lam) \end{array}\right]  $ be a matrix polynomial, where $Z_n(\lam)$ has $n$ rows and $ Z_r(\lam)$ has $r $ rows. If $Z(\lam)$  is a right  (resp., left) minimal  basis of $ \mathcal{S}(\lam)$  then $ Z_n(\lam)$ is a right (resp., left) minimal basis of $G(\lam).$  Further, the right (resp., left) minimal indices of $G(\lambda)$ and $\mathcal{S}(\lambda)$ are the same.

	(II) Suppose that $G(\lam)$ is regular and $ \mu \in \C$ is an eigenvalue of $G(\lam)$. Let $ Z := \left[\begin{array}{c} Z_{n} \\  Z_r \end{array}\right]$ be an $(n+r)\times p $ matrix such that $\rank(Z) = p,$ where $ Z_{n}$ has $ n$ rows and $Z_r$ has $r$ rows. If $Z$ is a  basis of $\mathcal{N}_r(\mathcal{S} (\mu))$ (resp., $\mathcal{N}_l(\mathcal{S} (\mu))$) then $Z_n$  is a  basis of $\mathcal{N}_r(G(\mu))$ (resp., $\mathcal{N}_l(G(\mu))$).
\end{theorem}

Thus, in view of Theorem~\ref{StoG}, we only need to describe the recovery of eigenvectors, minimal bases and minimal indices of $\mathcal{S}(\lam)$ from those of the EGFPs of $G(\lam).$ To that end, we need the following result.

\begin{theorem} \cite{rafiran2} \label{eigpgf_RCh4} ~Let $\alpha$ be a permutation of $\{0:m-1\}$ and $\mathbb{T}(\lam) :=\left[
	\begin{array}{@{}c|c@{}}
		T(\lam) &  e_{m-i_0 (\alpha)} \otimes C \\[.1em] \hline \\[-1em]
		e^T_{m-c_0(\alpha)} \otimes B  & A-\lam E
	\end{array}
	\right]$ 
	be the Fiedler pencil of $G(\lam)$ associated with $\alpha$, where $T(\lam) := \lam M^P_{-m} -M^P_{\alpha} $ is the Fiedler pencil of $P(\lam)$. Then we have the following:
	
	{\bf (I)~Minimal bases}.  Suppose that $\mathcal{S}(\lam)$ is singular.   Then the maps 
	\begin{equation*}
		\begin{array}{l}
			\mathcal{F}: \mathcal{N}_r(\mathbb{T}) \rightarrow \mathcal{N}_r(\mathcal{S}),~ \left[ \begin{array}{@{}c@{}}
				u(\lambda)\\ 
				v(\lambda) \\
			\end{array}
			\right] \mapsto \left[ \begin{array}{@{}c@{}}
				(e^T_{m-c_0(\alpha)} \otimes I_n )u(\lambda) \\[.2em] 
				v(\lambda) \end{array} \right], \\ \\[-1em]
			\mathcal{H}: \mathcal{N}_l(\mathbb{T}) \rightarrow \mathcal{N}_l(\mathcal{S}),~ \left[ \begin{array}{@{}c@{}}
				u(\lambda)\\ 
				v(\lambda) \\
			\end{array}
			\right] \mapsto \left[ \begin{array}{@{}c@{}}
				(e^T_{m-i_0(\alpha)} \otimes I_n )u(\lambda) \\[.2em] 
				v(\lambda) \end{array} \right],
		\end{array}
	\end{equation*}
	are linear isomorphisms, where $ u(\lambda)\in\mathbb{C}(\lambda)^{mn}$ and $ v(\lambda)\in\mathbb{C}(\lambda)^{r}.$ Further, $	\mathcal{F}$ (resp., $	\mathcal{H}$) maps a minimal basis of $\mathcal{N}_r(\mathbb{T})$ (resp., $\mathcal{N}_l(\mathbb{T})$) to a minimal basis of $\mathcal{N}_r(\mathcal{S})$ (resp., $\mathcal{N}_l(\mathcal{S})$).
	
	Further, if $\varepsilon_1\leq   \cdots \leq\varepsilon_p$ are the right (resp., left) minimal indices of $\mathbb{T}(\lambda)$ then $\varepsilon_1 -i(\alpha)  \leq \cdots \leq \varepsilon_p - i(\alpha) $ (resp., $\varepsilon_1 -c(\alpha) \leq \cdots \leq \varepsilon_p - c(\alpha) $) are the right (resp., left) minimal indices of  $\mathcal{S}(\lambda)$, where $ c (\alpha)$ and $ i(\alpha)$ be the total number of consecutions and inversions of the permutation $\alpha$, respectively.

	{\bf (II)~Eigenvectors.} Suppose that $\mathcal{S}(\lam)$ is regular and $\mu \in \mathbb{C}$ is an eigenvalue of $\mathcal{S}(\lam)$. Let $ Z:= \left[\begin{array}{c} Z_{mn} \\  Z_r \end{array}\right]$ be an $(mn+r)\times p $ matrix such that $\rank(Z) = p,$ where $ Z_{mn}$ has $ mn$ rows and $Z_r$ has $r$ rows. If $Z$ is a basis of $\mathcal{N}_r(\mathbb{T}(\mu))$ (resp., $\mathcal{N}_l(\mathbb{T}(\mu))$) then $\left[\begin{array}{@{}c@{}}
		(e^T_{m-c_0(\alpha)} \otimes I_n ) Z_{mn}  \\[.3em]   Z_r \end{array} \right] $ (resp., $\left[\begin{array}{@{}c@{}} 	(e^T_{m-i_0(\alpha)} \otimes I_n ) Z_{mn}  \\[.3em]   Z_r \end{array} \right] $)  is a basis of $\mathcal{N}_r(\mathcal{S}(\mu))$ (resp., $\mathcal{N}_l(\mathcal{S}(\mu))$).
\end{theorem}




We now describe the recovery of eigenvectors, minimal bases and minimal indices of $\mathcal{S}(\lam)$ from those of the EGFPs of $\mathcal{S}(\lam)$. 

\begin{theorem}\label{gfprrecoslame_RCh4} 
	Let  $ \mathbb{L}(\lam) := 
	\left[
	\begin{array}{@{}c|c@{}}
		L(\lam) &  e_{m-i_0 ( \sigma_1, \sigma)} \otimes C \\[.1em] \hline \\[-1em]
		e^T_{m-c_0(\sigma, \sigma_2)} \otimes B  & A-\lam E
	\end{array}
	\right]$ 
	be an EGFP $G(\lam)$ as given in (\ref{DefEGFP}). Let $ Z(\lam) := \left[\begin{array}{@{}c@{}} Z_{mn}(\lam) \\  Z_r(\lam) \end{array} \right]$ be an $(mn+r)\times p $ matrix polynomial, where $ Z_{mn}(\lam)$ has $ mn$ rows and $Z_r(\lam)$ has $r$ rows. Then we have the following. 
	
	\noin
	~(a) If $Z(\lam)$ is a right (resp., left) minimal basis of $\mathbb{L}(\lambda)$ then $\left[\begin{array}{@{}c@{}}
		(e^T_{m-c_0(\sigma, \sigma_2)} \otimes I_n ) Z_{mn}(\lambda) \\[.3em]   Z_r(\lam) \end{array} \right] $ (resp., $\left[\begin{array}{@{}c@{}} 	(e^T_{m-i_0(\sigma_1, \sigma)} \otimes I_n ) Z_{mn}(\lambda) \\[.3em]   Z_r(\lam) \end{array} \right] $)  is a right (resp., left) minimal basis of $\mathcal{S}(\lambda).$

	\noin
	~(b) Let $ \tau$ be given by $ \tau :=( -\beta, -m, -\gamma).$ Set $\alpha := (rev(\beta), \sig,rev(\gamma))$. Let $c(\alpha) $ and $ i(\alpha)$ be the total number of consecutions and inversions of the permutation $\alpha$. If $\varepsilon_1\leq   \cdots \leq\varepsilon_p$ are the right (resp., left) minimal indices of $\mathbb{L}(\lambda)$ then $\varepsilon_1 -i(\alpha)  \leq \cdots \leq \varepsilon_p - i(\alpha) $ (resp., $\varepsilon_1 -c(\alpha) \leq \cdots \leq \varepsilon_p - c(\alpha) $) are the right (resp., left) minimal indices of  $\mathcal{S}(\lambda).$
	
\end{theorem}

\begin{proof}	It is given that $L(\lam)$ is the EGFP of $P(\lam)$ associated with $\sig, \tau, \sig_i,~i=1,2$, and $\tau_i, ~i=1,2$. Let $\tau = (-\beta, -m, -\gamma )$ for some permutations $\beta$ and $\gamma$. Define $\alpha := (rev(\beta), \sig,rev(\gamma))$. Then $\alpha$ is a permutation of $\{0:m-1\}$ and
	$\mathbb{T}(\lam) :=\left[
	\begin{array}{@{}c|c@{}}
		T(\lam) &  e_{m-i_0 (\alpha)} \otimes C \\[.1em] \hline \\[-1em]
		e^T_{m-c_0(\alpha)} \otimes B  & A-\lam E
	\end{array}
	\right]$ is an FP of $G(\lam)$, where $T(\lam) := \lam M^P_{-m} -M^P_{\alpha} $ is a FP of $P(\lam)$. From the proof of Theorem~\ref{EGFPofGLin}, recall that $\diag(\mathcal{A},I_r) \mathbb{T}(\lam) \diag(\mathcal{B},I_r) =\mathbb{L}(\lam),$ where $\mathcal{A} := M_{(\tau_1,\sig_1)}(Y_1, X_1)  M^P_{-\beta}$ and  $\mathcal{B} :=
	M^P_{-\gamma} M_{(\sig_2,\tau_2)}(X_2, Y_2)$. Since $\mathbb{V}:=  \diag(\mathcal{B},I_r)$ is a nonsingular matrix, it is easily seen that the map $ \mathbb{V} : \mathcal{N}_r(\mathbb{L}) \rightarrow \mathcal{N}_r(\mathbb{T}),~ z(\lambda) \mapsto \mathbb{V} z(\lambda) ,$ is an isomorphism and maps a minimal basis of $\mathcal{N}_r(\mathbb{L}) $ to a minimal basis of $\mathcal{N}_r(\mathbb{T})$. On the other hand, by Theorem \ref{eigpgf_RCh4},
	$ \mathcal{F} : \mathcal{N}_r(\mathbb{T}) \rightarrow  \mathcal{N}_r( \mathcal{S}), \left[
	\begin{array}{@{}c@{}}
		x(\lam)\\ 
		y(\lambda) \\
	\end{array}
	\right]  \mapsto \left[
	\begin{array}{@{}c@{}}
		(e^T_{m-c_0(\alpha)}\otimes I_n)x(\lam) \\ 
		y(\lambda) \\
	\end{array}
	\right]  ,$  is an isomorphism and maps a minimal basis of $\mathcal{N}_r(\mathbb{T})$ to a minimal basis of $\mathcal{N}_r(\mathcal{S})$, where $x(\lambda) \in \mathbb{C}(\lambda)^{mn}$ and $y(\lambda) \in \mathbb{C}(\lambda)^{r}$. Consequently, $\mathcal{F}\,\mathbb{V}: \mathcal{N}_r(\mathbb{L}) \rightarrow \mathcal{N}_r( \mathcal{S}),$ $ z(\lam) \mapsto \mathcal{F}\,\mathbb{V} z(\lam),$  is an isomorphism and maps a minimal basis of $\mathcal{N}_r(\mathbb{L})$ to a minimal basis of $\mathcal{N}_r(\mathcal{S})$. Note that in the proof of Theorem~\ref{EGFPofGLin} it is shown that $(e^T_{m-c_0(\alpha)} \otimes I_n)\,\mathcal{B}= e^T_{m-c_0(\sigma, \sigma_2)} \otimes I_n$. Consequently, we have $\mathcal{F}\, \mathbb{V}  = \left[
	\begin{array}{@{}c@{\,\,\,}|@{\,}c@{}}
		e^T_{m-c_0(\sigma, \sigma_2)} \otimes I_n &  \\ \hline
		& I_r \\
	\end{array}
	\right], $ and hence the desired result for the recovery of right minimal bases follows.

	An analogous proof holds for describing the recovery of left minimal bases.

	Finally, let $\varepsilon_1\leq \cdots \leq\varepsilon_p$ be the right (resp., left) minimal indices of $\mathbb{L}(\lambda)$. Since the pencil $\mathbb{T}(\lambda)$ is strictly equivalent to $ \mathbb{L}(\lambda)$, $\varepsilon_1\leq \cdots \leq\varepsilon_p$ are also the right (resp., left) minimal indices of $\mathbb{T}(\lambda)$. Hence by Theorem \ref{eigpgf_RCh4},  $\varepsilon_1 -i(\alpha)  \leq \cdots \leq \varepsilon_p - i(\alpha) $ (resp., $\varepsilon_1 -c(\alpha) \leq \cdots \leq \varepsilon_p - c(\alpha) $) are the right (resp., left) minimal indices of  $\mathcal{S}(\lambda).$ This completes the proof. 
\end{proof}

The next result describes the recovery of  eigenvectors of $\mathcal{S}(\lam)$ from those of the EGFPs of $ \mathcal{S} (\lam)$ when $\mathcal{S}(\lam)$ is regular.

\begin{theorem} \label{gfprrecoslamregular_RCh4} Let  $ \mathbb{L}(\lam) := 
	\left[
	\begin{array}{@{}c|c@{}}
		L(\lam) &  e_{m-i_0 ( \sigma_1, \sigma)} \otimes C \\[.1em] \hline \\[-1em]
		e^T_{m-c_0(\sigma, \sigma_2)} \otimes B  & A-\lam E
	\end{array}
	\right]$ 
	be an EGFP $\mathcal{S}(\lam)$ given in (\ref{DefEGFP}). Suppose that $\mathcal{S}(\lam)$ is regular and $ \mu \in \C$ is an eigenvalue of $\mathcal{S}(\lam)$. Let $ Z := \left[\begin{array}{c} Z_{mn} \\  Z_r \end{array}\right]$ be an $(mn+r)\times p $ matrix such that $\rank(Z) = p,$ where $ Z_{mn}$ has $ mn$ rows and $Z_r$ has $r$ rows. If $Z$ is a  basis of $\mathcal{N}_r(\mathbb{L} (\mu))$ (resp., $\mathcal{N}_l(\mathbb{L} (\mu))$) then $\left[\begin{array}{@{}c@{}}
		(e^T_{m-c_0(\sigma, \sigma_2)} \otimes I_n ) Z_{mn} \\[.3em]   Z_r \end{array} \right] $ (resp., $\left[\begin{array}{@{}c@{}} 	(e^T_{m-i_0(\sigma_1, \sigma)} \otimes I_n ) Z_{mn}  \\[.3em]   Z_r \end{array} \right] $)  is a  basis of $\mathcal{N}_r(\mathcal{S}(\mu))$ (resp., $\mathcal{N}_l(\mathcal{S}(\mu))$).
	
\end{theorem}

\begin{proof} A verbatim proof of Theorem~\ref{gfprrecoslame_RCh4} together with part (II) of Theorem~\ref{eigpgf_RCh4} yields the desired results.
\end{proof}

\renewcommand{\thesection}{\Alph{section}.\arabic{section}}
\setcounter{section}{0}

\begin{appendix} ~
	
\section{Proofs of Lemmas~\ref{egfpr_right_reco1}, \ref{egfpr_right_reco} and \ref{egfpr_left_reco}} \label{appendix_reco}~ The following result will be used frequently in this section. 
	\begin{proposition} \label{sp1rowblock} Let  $ \alpha$ be an index tuple containing indices from $ \{0:m-1\}$ such that $\alpha$ satisfies the  SIP. Let  $ Z$ be an arbitrary matrix assignment for $\alpha$. Let $ 0 \leq s \leq m-1$. Suppose that $ s+1 \in \alpha$. Then we have the following.
		\begin{itemize}
			\item[(a)] 	If the subtuple of $\alpha$ with indices $ \{ s , s+1\}$ starts with $ s+1$, then $ ( e^T_{ m-s} \otimes I_n ) M_{\alpha} (Z) = e^T_{ m-( s+1+c_{ s+1} (\alpha))} \otimes I_n. $  
			
			\item[(b)] If the subtuple of $\alpha$ with indices $ \{ s , s+1\}$ ends with $ s+1$, then   $  M_{\alpha} (Z) ( e_{ m-s} \otimes I_n ) = e_{ m-( s+1+i_{ s+1} (\alpha))} \otimes I_n. $
		\end{itemize}
	\end{proposition}
	
	\begin{proof} (a) Set $ p:= c_{s+1} (\alpha)$, i.e., $\alpha$ has $p$ consecutions at $s+1$. Since $ \alpha$ satisfies the SIP, by Proposition~\ref{lemeqnconrsfcsf16a18n730_PCh3},
		$$ \alpha \sim \big(  \alpha^L, s+1, s+2, \ldots , s+p+1 , \alpha^R\big),$$
		where $ s+1 \notin \alpha^L$ and $ s+p+1, s+p+2 \notin \alpha^R$. Further, since the subtuple of $ \alpha$ with indices $\{ s, s+1 \}$ starts with $ s+1$, we  have $ s \notin \alpha^L$. We denote by $(*)$ any arbitrary matrix assignment. Then we have
		\begin{align*}
			&  (e^T_{m-s} \otimes I_n) \,  M_{\alpha^L}(*) \, M_{(s+1, s+2, \ldots, s+p+1)} (*) \, M_{\alpha^R}  (*) \\
			& =  (e^T_{m-s} \otimes I_n) \,  M_{(s+1, s+2, \ldots, s+p+1)} (*) \,  M_{\alpha^R}  (*)~ \text{by}~ (\ref{eqnreco2_PCh4}) ~\text{since} ~ s , s+1 \notin \alpha^L\\
			& =  (e^T_{m-(s+p+1)} \otimes I_n) \, M_{\alpha^R}  (*)~ \text{by repeatedly applying}~ (\ref{eqnreco2_PCh4})\\
			& =  e^T_{m-(s+p+1)} \otimes I_n ~\text{by}~ (\ref{eqnreco2_PCh4}) ~\text{since}~ s+p+1, s+p+2 \notin \alpha^R.
		\end{align*}
		Thus, in particular, we have $ (e^T_{m-s} \otimes I_n) M_{\alpha} (Z) = e^T_{m-(s+p+1)} \otimes I_n. $	Since $ p = c_{s+1} (\alpha) $, the desired result follows.

		(b) Set $ q:= i_{s+1} (\alpha)$, i.e., $\alpha$ has $q$ inversions at $s+1$. Since $ \alpha$ satisfies the SIP, by Proposition~\ref{lemeqnconrsfcsf16a18n730_PCh3},
		$$ \alpha \sim \big(  \alpha^L, s+q+1,  \ldots  ,s+2, s+1, \alpha^R \big),$$
		where $ s+1 \notin \alpha^R$ and $ s+q+1, s+q+2 \notin \alpha^L$. Further, since the subtuple of $ \alpha$ with indices $\{ s, s+1 \}$ ends with $ s+1$, we  have $ s \notin \alpha^R$. Now  we have
		\begin{align*}
			&   M_{\alpha^L}(*) \, M_{(s+1+q, \ldots, s+2, s+1)} (*) \, M_{\alpha^R}  (*)  \, (e_{m-s} \otimes I_n) \\
			& =  M_{\alpha^L}(*) \, M_{(s+1+q, \ldots, s+2, s+1)} (*) \, (e_{m-s} \otimes I_n) ~ \text{by}~ (\ref{eqnreco1_PCh4}) \text{ since }  s , s+1 \notin \alpha^R\\
			& =   M_{\alpha^L}  (*) \, (e_{m-(s+q+1)} \otimes I_n) ~ \text{by repeatedly applying}~ (\ref{eqnreco1_PCh4})\\
			& =  e_{m-(s+q+1)} \otimes I_n ~\text{by}~ (\ref{eqnreco1_PCh4}) \text{ as } s+q+1, s+q+2 \notin \alpha^L.
		\end{align*}
		Thus, in particular, we have $  M_{\alpha} (Z) \, (e_{m-s} \otimes I_n) = e_{m-(s+q+1)} \otimes I_n. $	Since $ q= i_{s+1} (\alpha) $, the desired result follows. 
	\end{proof}

Analogous to Proposition~\ref{sp1rowblock}, we have the following results for negative index tuples.
	\begin{proposition} \label{ms_start} Let  $ \alpha$ be an index tuple containing indices from $ \{-m:-1\}$ such that $\alpha$ satisfies the  SIP. Let  $ Z$ be an arbitrary matrix assignment for $\alpha$. Let $ 1 \leq s \leq m-1$. Suppose that $ -s \in \alpha$. If the subtuple of $\alpha$ with indices $ \{ -s ,-( s+1) \}$ starts with $ -s$, then $ ( e^T_{ m-s} \otimes I_n ) M_{\alpha} (Z) = e^T_{ m-( s-c_{ -s} (\alpha)-1)} \otimes I_n. $
		
		Similarly, if  the subtuple of $\alpha$ with indices $ \{ -s ,-( s+1) \}$ ends with $ -s$, then  $  M_{\alpha} (Z) \, ( e_{ m-s} \otimes I_n ) = e_{ m-( s-i_{ -s} (\alpha)-1)} \otimes I_n. $
	\end{proposition}

\subsection{The proof of Lemma~\ref{egfpr_right_reco1}}~

	\begin{proof} (a) First we prove that  $(e_{m-c_0(\sigma)}^{T}\otimes I_n) \, M_{\sigma_{2}} (X_2) \,  M_{\tau_{2}}  (Y_2) = e_{m-c_0(\sigma,\sigma_2)}^{T}\otimes I_n $. Since $\sig$ has $c_0(\sig)$ $(\leq m-1)$ consecutions at $0$, we have $\sig \sim (\sig^L, 0,1,\ldots , c_0(\sig))$, where either $m\in \sigma^L$ or  $m \in -\tau$. If $-m \in \tau$ then the desired result follows from Lemma~\ref{blk_rows_n_col}.  But there is nothing special about $m \in \sig$. If $m \in \sig$ then by defining $\widetilde{\sig} := \sig \setminus \{m\}$ and $\widetilde{\tau} :=(\tau,-m)$, we have  $ c_0( \sig) = c_0(\widetilde{\sig})$ and  $ c_0( \sig, \sig_2) = c_0(\widetilde{\sig}, \sig_2)$. Now by Lemma~\ref{gfprptogArbitraryCoeff_PCh3}, $ (e_{m-c_0(\widetilde{\sigma})}^{T}\otimes I_n) \,  M_{\sigma_{2}} (X_2) M_{\tau_{2}} (Y_2) = e^T_{m-c_0(\widetilde{\sigma}, \sig_2)} \otimes I_n$. Consequently, $(e_{m-c_0(\sigma)}^{T}\otimes I_n) \,  M_{\sigma_{2}} (X_2)M_{\tau_{2}}  (Y_2) = (e_{m-c_0(\widetilde{\sigma})}^{T}\otimes I_n) \,  M_{\sigma_{2}} (X_2) M_{\tau_{2}}  (Y_2) = e^T_{m-c_0(\widetilde{\sigma}, \sig_2)} \otimes I_n =  e^T_{m-c_0(\sigma, \sig_2)} \otimes I_n$.
		
		Further, by similar arguments as above, we have $M_{\tau_{1}} (Y_1)\, M_{\sigma_{1}} (X_1) \, (e_{m-i_0(\sigma)}\otimes I_n) = e_{m-i_0(\sigma_1, \sigma)}\otimes I_n$. This completes the proof of (a).
		
		%
		%
		%
		%

		(b) Note that if $ k \notin \{ \pm (m-1), \pm m\}$, then $(e_1^T \otimes I_n) M_k (Z) = e_1^T \otimes I_n$ for any matrix $Z \in \mathbb{C}^{n \times n}$. Hence  $(e_1^T \otimes I_n) M_{\sig_2} (X_2) = e_1^T \otimes I_n$ and   $ M_{\sig_1} (X_1) (e_1 \otimes I_n) = e_1 \otimes I_n$ as $\sig_j$, $j=1,2,$ contains indices from $\{ 0:m-2\}$. Further, since $\tau_1 = \emptyset = \tau_2$, we have $ M_{\tau_1} (Y_1)= I_{mn} = M_{\tau_2} (Y_2)$. Consequently, we have  $ (e_{1}^{T}\otimes I_n) M_{\sigma_{2}} (X_2) M_{\tau_{2}}  (Y_2)=  e_{1}^{T}\otimes I_n$ and $M_{\tau_{1}} (Y_1) M_{\sigma_{1}} (X_1) (e_{1}\otimes I_n) =e_{1}\otimes I_n$. This completes the proof of (b).
	\end{proof}

\subsection{The proof of Lemma~\ref{egfpr_right_reco}}~		
	\begin{proof} (a) Given that $i_0(\omega) +1 \in \sig$ and $s:= i_0(\omega) + c_{i_0(\omega)+1}(\sig) +1 < m$. Set $q:=c_{i_0(\omega)+1}(\sig)$. Then $s = i_0(\omega) +1+q$. Note that $0,1, \ldots, i_0(\omega) \notin \sig$. Since $\sig$ is a sub-permutation and has $q$ consecutions at $i_0(\omega) +1$, we have $\sig \sim \Bigs (\widehat{\sig} , i_0(\omega) +1, i_0(\omega) +2, \ldots, i_0(\omega) +1+q \Bigs )$, that is,
		\begin{equation} \label{sig_formati0w}
			\sig \sim \Bigs (\widehat{\sig} , i_0(\omega) +1, i_0(\omega) +2, \ldots, s \Bigs ).
		\end{equation}
		Suppose that $s+1 \notin \sig_2$. Then $s \notin \sig_2$, since otherwise $(\sig, \sig_2)$ would not satisfy the SIP which would contradict the condition that $(\sig_1, \sig, \sig_2)$ satisfies the SIP. Thus $s , s+1 \notin \sig_2$. Hence by (\ref{eqnreco2_PCh4}), we have $(e^T_{m-s} \otimes I_n) M_{\sigma_{2}} (X_2) M_{\tau_{2}}  (Y_2) =  (e^T_{m-s} \otimes I_n)  M_{\tau_{2}}  (Y_2) = e^T_{m-s} \otimes I_n$, where the last equality follows from (\ref{eqnrecogfprp0329sept172046n_PCh2}) since $s \in \sig$ implies that $-s,-(s+1) \notin \tau_2.$	Now, since $ s+1 \notin \sig_2$, it is clear from (\ref{sig_formati0w}) that $c_{i_0(\omega)+1}(\sig, \sig_2) = c_{i_0(\omega)+1}(\sig)$. This shows that $p=s$ which yields the desired result.
		
		Next, suppose that $s+1 \in \sig_2$. Set $r := c_{s+1}(\sig_2)$, i.e., $\sig_2$ has $r$ consecutions at $s+1$. Then by Proposition~\ref{lemeqnconrsfcsf16a18n730_PCh3}, we have
		\begin{equation} \label{sig2form}
			\sig_2 \sim \left( \sigma_2^{L}, s+1, s+2, \ldots , s+1+r , \sigma_2^{R} \right)
		\end{equation}
		for some index tuples $ \sigma_2^{L} $  and $\sigma_2^{R}$, where $s+1 \notin \sigma_2^{L}$ and $ s+1+r, s+2+r \notin \sigma_2^{R}$.  Since $(\sig, \sig_2)$ satisfies the SIP and $s+1 \notin \sig_2^L$, it is clear from (\ref{sig_formati0w}) and (\ref{sig2form}) that $s \notin \sig_2^{L}.$ So the subtuple of $\sig_2$ with indices $\{ s, s+1 \}$ starts with $s+1$. Hence by Proposition~\ref{sp1rowblock},  $(e^T_{m-s} \otimes I_n) M_{\sigma_{2}} (X_2) =  e^T_{m-(s+1+r)} \otimes I_n.  $ Now since $ s+1+r \in \sig_2$, we have $s+1+r, s+2+r \in \sig$. This implies  $s+1+r,s+2+r \notin \tau_2$. Thus by (\ref{eqnrecogfprp0329sept172046n_PCh2}),  $ (e^T_{m-(s+1+r)} \otimes I_n ) M_{\tau_2} (Y_2) = e^T_{m-(s+1+r)} \otimes I_n.$ Now, by (\ref{sig_formati0w}) and (\ref{sig2form}),  $ c_{i_0(\omega)+1} (\sig,\sig_2) =  s+r - i_0(\omega) . $ Consequently, $ s+r +1 =p$ which proves (a).

		\noin ~(b)  Since $\sig_2$ contains indices from $\{0:m-2\}$, by (\ref{eqnreco2_PCh4}), we have $(e_1^T \otimes I_n)  M_{\sigma_{2}} (X_2) M_{\tau_{2}}  (Y_2) $ $= (e_1^T \otimes I_n) M_{\tau_{2}}  (Y_2) .$ As $i_0(\omega) +1 \in \sig$ and $s=m$, we have $m \in \sig$. Thus $-m \notin \tau$ and hence $-m \notin \tau_2$.
		
		Suppose that $-(m-1) \notin \tau_2$. As $-m,-(m-1) \notin \tau_2$, by (\ref{eqnrecogfprp0329sept172046n_PCh2}), we have $(e_1^T \otimes I_n) M_{\tau_{2}}  (Y_2) = e_1^T \otimes I_n.$ Hence the desired result follows from the fact that $-(m-1) \notin \tau_2$ implies $c_{-(m-1)} (\tau_2) = -1$.  Next, suppose that $-(m-1) \in \tau_2$. Set $q:=c_{-(m-1)} (\tau_2).$ As $-m \notin \tau_2$, the subtuple of $\tau_2$ with indices $\{-m, -(m-1) \}$ starts with $-(m-1).$ Hence by Proposition~\ref{ms_start}, $(e^T_{1} \otimes I_n) M_{\tau_2} (Y_2)= (e^T_{m-(m-1) } \otimes I_n) M_{\tau_2} (Y_2) = e^T_{m-(m-q-2) } \otimes I_n = e^T_{q+2 } \otimes I_n $  which proves (b).
		

		(c) Let   $i_0(\omega) <m$ and $i_0(\omega) +1 \notin \sig$, i.e., $-i_0(\omega), -(i_0(\omega) +1 )\in  \tau.$ This implies that $i_0(\omega), i_0(\omega) +1 \notin  \sig_2$. Thus by (\ref{eqnreco2_PCh4}), we have
		$$(e^T_{m-i_0(\omega)} \otimes I_n)  M_{\sigma_{2}} (X_2) M_{\tau_{2}}  (Y_2) = (e^T_{m-i_0(\omega)} \otimes I_n) M_{\tau_{2}}  (Y_2) . $$
		Since $\omega$ is a sub-permutation and  has $i_0(\omega)$ inversions at $0$, we have $\omega \sim (  i_0(\omega),\ldots,1,0, \widehat{\omega} ),$ where $i_0(\omega) +1 \in \widehat{\omega}$. Consequently, we have (recall that $\tau = -\omega$)
		\begin{equation} \label{tauform}
			\tau \sim \Big (  -i_0(\omega),\ldots,-1,-0, \widehat{\tau} \Big ) , \text{ where } -(i_0(\omega) +1)  \in \widehat{\tau}.
		\end{equation}
		Suppose that  $-i_0(\omega) \notin \tau_2$. Then $-(i_0(\omega)+1) \notin \tau_2$, since otherwise $(\tau, \tau_2)$ would not satisfy the SIP  which would contradict the condition that $(\tau_1, \tau, \tau_2)$ satisfies the SIP. As $-i_0(\omega) , -(i_0(\omega)+1) \notin \tau_2 $, by (\ref{eqnrecogfprp0329sept172046n_PCh2}), we have $(e^T_{m-i_0(\omega)} \otimes I_n) M_{\tau_{2}}  (Y_2) = e^T_{m-i_0(\omega)} \otimes I_n.$ Again, since $-i_0(\omega) \notin \tau_2 $, we have   $c_{-i_0(\omega)}(\tau_2) = -1$ which gives (\ref{eqn:recoev2CH3}).
		
		Next, suppose that  $-i_0(\omega) \in \tau_2$. Set $q := c_{-i_0(\omega)} (\tau_2).$ Then by Proposition~\ref{lemeqnconrsfcsf17a18d10},
		\begin{equation} \label{tau2form}
			\tau_2 \sim \Big (\tau_2^L , -i_0(\omega), -(i_0(\omega)-1) , \ldots, -(i_0(\omega)-q) , \tau_2^R \Big )
		\end{equation}
		for some index  tuples $\tau_2^L$ and $\tau_2^R$ such that $-i_0(\omega) \notin \tau_2^L$ and $-(i_0(\omega)-q), -(i_0(\omega)-q-1) \notin \tau_2^R$. As $(\tau,\tau_2)$ satisfies the SIP and $-i_0(\omega) \notin \tau_2^L$, it is clear from (\ref{tauform}) and (\ref{tau2form}) that $-(i_0(\omega)+1) \notin \tau_2^L$. Thus the  subtuple of $\tau_2$ with indices from $\{-i_0(\omega), -(i_0(\omega)+1) \}$ starts with $-i_0(\omega). $ Hence by Proposition~\ref{ms_start}, we have
		$$(e^T_{m-i_0(\omega)} \otimes I_n) M_{\tau_{2}}  (Y_2) = e^T_{m-(i_0(\omega)-q-1)} \otimes I_n, $$ which gives (\ref{eqn:recoev2CH3}) as $q = c_{-i_0(\omega)} (\tau_2).$ This completes the proof of (c).



		(d) Let $0 \in \omega$ and $i_0(\omega) =m$. Then  $\sig = \emptyset = \sig_2$. Hence $M_{\sig_2} (X_2) = I_{mn}$. Further, since $\tau_2$ contains indices from $\{-m:-2\}$, we have $-0,-1 \notin \tau_2$. Thus $(e^T_m \otimes I_n) M_{\tau_{2}}  (Y_2) =e^T_m \otimes I_n $ which proves (d).
	\end{proof}

\subsection{The proof of Lemma~\ref{egfpr_left_reco}}~

The proof of Lemma~\ref{egfpr_left_reco} is analogous to the proof of Lemma~\ref{egfpr_right_reco}.

\section{Proofs of Propositions~\ref{prop1:Bpenta}, \ref{Prop:2:Blkpenta}, \ref{not_blkpentaPve} and \ref{not_blkpentaNve}} \label{appendixBB}

\subsection{The proof of Proposition~\ref{prop1:Bpenta}}\label{appendix1}~
\begin{proof} It is easily seen that the elementary matrices $M_{j} (W)$, $j =0:m-1$, satisfies 
	$(e^T_{m-i} \otimes I_n) M_{j} (W) = 	(e^T_{m-i} \otimes W) + (e^T_{m-(i-1)} \otimes I_n)  \text{ for } j=i  \text{ and } i=1:m-1,$ and $(e^T_{m} \otimes I_n) M_{0} (W) = 	e^T_{m} \otimes W.$ Consequently, from (\ref{eqnreco2_PCh4}) and (\ref{eqnreco1_PCh4}) we have the following
	\begin{equation}\label{Pve_Brow}
		(e^T_{m-i} \otimes I_n) M_{j} (W)= 
		\left\{ 
		\begin{array}{ll}
			e^T_{m-(i+1)} \otimes I_n & \text{for }~ j=i+1 \text{ and } i=0:m-2,\\
			(e^T_{m-i} \otimes W) + (e^T_{m-(i-1)} \otimes I_n)  & \text{for } j=i  \text{ and } i=1:m-1,\\
			e^T_{m} \otimes W & \text{for } j=i = 0 ,\\
			e^T_{m-i} \otimes I_n & \left \{ \begin{array}{l}
				\mbox{otherwise, i.e., when}\\
				j \notin \{i,i+1\} \text{ for } i=0:m-1.  \end{array} \right.
		\end{array}
		\right. 
	\end{equation}

	It is given that $c_k(\alpha)  \leq 1$ and $i_k(\alpha)  \leq 1$ for $1 \leq k \leq m-1$. This implies $c_0(\alpha) \leq 2$ and  $i_0(\alpha) \leq 2$.  Recall from the definition of consecutions and inversions of any index tuple $\beta$ at any index $t$ that if $t \notin \beta$ then $c_t(\beta) = -1 = i_t(\beta)$. Consequently, we have 
	\begin{eqnarray} \label{ckik}
		\left \{ \begin{array}{l}
			-1 \leq c_k(\alpha) \leq 1 \text{ and } -1 \leq i_k(\alpha)   \leq 1 \text{ for } k = 1 : m-1,\\
			-1 \leq c_0(\alpha) \leq 2  \text{ and } -1 \leq i_0(\alpha)   \leq 2.
		\end{array}\right.
	\end{eqnarray}
	
	Consider $0 \leq k \leq m-1.$ We now prove (\ref{aim:to:prove}). There are two cases.
	
	Case-I: Suppose that the subtuple of $\alpha$ with indices $\{k,k+1\}$ starts with $k+1$. Then by Proposition~\ref{sp1rowblock}, we have $(e^T_{m-k} \otimes I_n)  M_\alpha (\mathcal{X}) = e^T_{m-(k+1+c_{k+1}(\alpha) )} \otimes I_n.$ Since $ -1 \leq c_{k+1} (\alpha)  \leq 1$, we have    (\ref{aim:to:prove}).

	Case-II: Suppose that the subtuple of $\alpha$ with indices $\{k,k+1\}$ starts with $k$. Let the row standard form of $\alpha$ be given by $ rsf(\alpha) : = (\beta , rev(a_k:k), \gamma) =$
	\begin{align} 
		\big (  rev(a_0:0), rev(a_1:1), \ldots, rev(a_k:k), rev(a_{k+1}:k+1),\ldots ,  rev(a_{m-1}:m-1) \big). \label{rsfalp}
	\end{align}
	Then we must have $rev(a_k:k) \neq \emptyset$, since otherwise it is clear from (\ref{rsfalp}) that the subtuple of $\alpha$ with indices $\{k,k+1\}$ would start with $k+1$ which would contradict our assumption of Case-II. Since $\alpha$ is a tuple containing nonnegative indices, by using (\ref{ckik})  it is clear from (\ref{rsfalp}) that  $a_j \geq j-1$ for all $ j =\{ 0:m-1\} \setminus \{ 2\}$ and $a_2 \geq 0$.  Further, since $rev (a_k:k) \neq \emptyset$, we have $ a_k \in \{k, k-1\}$ if $k \neq 2$, and  $ a_k \in \{k, k-1, k-2\}$ if $k = 2$.

	Since $\alpha \sim rsf(\alpha)$, without loss of generality, we assume that $\alpha = rsf(\alpha)$. That is, $\alpha : = (\beta , rev (a_k:k), \gamma) $. Let $\mathcal{Y}$ and $\mathcal{Z}$, respectively, be the corresponding matrix assignments for $\beta$ and $\gamma$ associated with $\mathcal{X}$. We denote by $X^p$, $p \in \{a_k :k\}$, the matrices associated with $\mathcal{X}$ and assigned to the index $p \in \{a_k:k \} $. That is, $\mathcal{X} =  \big (\mathcal{Y} , rev(X^{a_k} \cdots X^k), \mathcal{Z} \big )$.  Now we have
	\begin{align}
		& (e^T_{m-k} \otimes I_n)  M_\alpha (\mathcal{X}) =  (e^T_{m-k} \otimes I_n)  M_\beta (\mathcal{Y})  M_{rev(a_k:k)} (rev(X^{a_k} \cdots X^k))  M_{\gamma} (\mathcal{Z}) \nonumber\\
		& =  (e^T_{m-k} \otimes I_n)  M_{rev(a_k:k)} (rev(X^{a_k} \cdots  X^k))  M_{\gamma} (\mathcal{Z}) \text{ by } (\ref{Pve_Brow}) \text{ as } k,k+1 \notin \beta \nonumber\\
		& = \left\{ \begin{array}{ll} 
			\Big ((e^T_{m-(a_k-1)} \otimes I_n) + \sum\nolimits_{j= a_k}^{k} (e^T_{m-j} \otimes X^j) \Big ) M_{\gamma} (\mathcal{Z}) & \text{if } a_k > 0,   \\ \\[-1em] 
			\Big (\sum\nolimits_{j= a_k}^{k} (e^T_{m-j} \otimes X^j) \Big ) M_{\gamma} (\mathcal{Z})   & \text{if } a_k = 0 \end{array} \right. \label{eqn:rsf1}  \\
		&  \hspace{6cm} \text{ by applying }   (\ref{Pve_Brow})  \text{ repeatedly}. \nonumber 
	\end{align}
	Define $S := \{ a_k-1 :k \}$ if $a_k > 0$ and $S := \{a_k:k\}$ if $a_k =0$. It is clear from (\ref{eqn:rsf1}) that for evaluating $ (e^T_{m-k} \otimes I_n)  M_\alpha (X) $ we need to calculate  $(e^T_{m-\ell} \otimes I_n)  M_{\gamma} (\mathcal{Z}) $ for all $\ell \in S$.  Since $\gamma : = \big ( rev(a_{k+1}:k+1), \ldots,  rev(a_{m-1}:m-1) \big)$, it is clear that, for $\ell \in S$, we have either $\ell \notin \gamma $ or the subtuple of $\gamma$ with indices $\{ \ell, \ell +1\}$ starts with $\ell +1$. By Proposition~\ref{sp1rowblock}, we have $(e^T_{m-\ell} \otimes I_n)  M_{\gamma}  (\mathcal{Z}) = e^T_{m- (\ell+1 + c_{\ell+1}  (\gamma))} \otimes I_n.$  Hence from (\ref{eqn:rsf1}), we have $(e^T_{m-k} \otimes I_n)  M_\alpha (\mathcal{X})=$
	\begin{align}  \label{eqn:achieve}
		\left\{ \begin{array}{ll}  (e^T_{m- (a_k + c_{a_k} (\gamma)) } \otimes I_n) + \sum\nolimits_{j= a_k}^{k} (e^T_{m-( j+1 + c_{j+1}(\gamma))} \otimes X^j)  & \text{if } a_k >0, \\
			\sum\nolimits_{j= a_k}^{k} (e^T_{m-( j+1 + c_{j+1}(\gamma))} \otimes X^j)  & \text{if } a_k =0.
		\end{array} \right. 
	\end{align}
	Since $\gamma $ is a subtuple of $\alpha$, we have $ c_{t} (\gamma) \leq c_{t} (\alpha) $ for any index $t$. Hence $ -1 \leq  c_{\ell+1} (\gamma) \leq 1  $ for all $\ell \in S$. Further, since $ a_k \in \{k, k-1\}$ if $k \neq 2$, and  $ a_k \in \{k, k-1, k-2\}$ if $k = 2$, it is clear from (\ref{eqn:achieve}) that (\ref{aim:to:prove}) holds.  Hence $M_{\alpha} (\mathcal{X})$  is a block penta-diagonal matrix. This completes the proof.
\end{proof}

\subsection{The proof of  Proposition~\ref{Prop:2:Blkpenta}}\label{appendix2}~

\begin{proof} It is easily seen that the elementary matrices $M_{- j} (W)$, $j =0:m$, satisfies
	$$ (e^T_{m-(m-1)} \otimes I_n) M_{-m} (W) = e^T_{1} \otimes W  \text{ and}$$
	$$ (e^T_{m-i} \otimes I_n) M_{-j} (W) = (e^T_{m-i} \otimes W ) +  (e^T_{m-(i+1)} \otimes I_n)  \text{ for } j= i+1 , \text{ and } i=0:m-2.$$  
	Consequently, from (\ref{eqnrecogfprp0329sept172046n_PCh2}) and (\ref{eqnrecogfprgsettingp0329sept172045n_PCh2}) we have the following.	
	\begin{equation}\label{Nve_Brow}
		(e^T_{m-i} \otimes I_n) M_{-j} (W)= 
		\left\{ 
		\begin{array}{ll}
			(e^T_{m-i} \otimes W ) +  (e^T_{m-(i+1)} \otimes I_n) & \text{for } j= i+1 , \text{ and } i=0:m-2 \\
			e^T_{1} \otimes W & \text{for } j= i+1 ,~  i=m -1\\
			e^T_{m-(i-1)} \otimes I_n & \text{for } j = i, \text{ and } i=1:m-1 \\
			e^T_{m-i} \otimes I_n & \left \{  \begin{array}{l}
				\mbox{otherwise, i.e., when}\\
				j \notin \{i,i+1\} \text{ for } i=0:m-1. \end{array} \right. 
		\end{array}
		\right. 
	\end{equation} 
	The rest of the proof is similar to Proposition~\ref{prop1:Bpenta}.
\end{proof}

\subsection{The proof of Proposition~\ref{not_blkpentaPve}}\label{appendix3}~
\begin{proof} Suppose that $c_j(\alpha) \geq 2$ for some $j \geq 1$. Let $1 \leq  k \leq m-1$ be the smallest integer belongs to  $\alpha$ such that $ c_k(\alpha) \geq  c_\ell (\alpha)  $ for any index $\ell \in \alpha$ (i.e., $c_k(\alpha) \geq 2$). Then  by Proposition~\ref{lemeqnconrsfcsf16a18n730_PCh3}, we have $\alpha \sim ( \alpha^L , k,k+1, \ldots, k+c_k(\alpha), \alpha^R)$, where $ k \notin \alpha^L$. Now we have the following cases.
	
	(a) Assume that $k >1 $. Then $1 \notin \alpha^L$, since otherwise $k-1 \geq 1$ and $c_{k-1}(\alpha) > c_k(\alpha)$ which is a contradiction. This implies that the subtuple of $\alpha$ with indices $\{k-1,k\}$ starts with $k$. Hence by Proposition~\ref{sp1rowblock}, we have $ (e_{m-(k-1)}^T \otimes I_n) M_{\alpha} (\mathcal{X}) =  e_{m-(k+c_k(\alpha))}^T \otimes I_n$. That is, $\big (m-(k + c_k(\alpha)) \big )$-th block entry of $(m-(k-1))$-th block row is $I_n$. Since $c_k(\alpha) \geq 2$, we have $k+c_k(\alpha) > k+1 = (k-1)+2 $.  Hence $M_\alpha (\mathcal{X})$ is not a block penta-diagonal matrix. 
	
	(b) Assume that $k =1$. Then $( \alpha^L , (1: 1+c_1(\alpha)), \alpha^R)$, where $c_1(\alpha) \geq 2$. If $0 \notin \alpha^L$ then the subtuple of $ \alpha$ with indices $ \{0,1\}$ starts with $1$. Hence by Proposition~\ref{sp1rowblock}, we have $ (e_{m}^T \otimes I_n) M_{\alpha} (\mathcal{X}) =  e_{m-(1+c_1(\alpha))}^T \otimes I_n$. Since $c_1(\alpha) \geq 2$, we have $1+c_1(\alpha) \geq 3 $. Hence $M_\alpha (\mathcal{X})$ is not a block penta-diagonal matrix.

	Let $ X^0$ be the matrix assigned to the index $0 \in \alpha^L$ associated with $\mathcal{X}$, and let $\mathcal{X} = (\mathcal{X}^L, \mathcal{X}^M, \mathcal{X}^R)$, where $ \mathcal{X}^L, \mathcal{X}^M,$ and $ \mathcal{X}^R$ are the matrix assignments associated with $\alpha^L, (1:1+c_{1}(\alpha) )$ and $\alpha^R$, respectively. Then $ (e_{m}^T \otimes I_n) M_{\alpha^L} (\mathcal{X}^L)  =  e_{m}^T \otimes X^0$ by (\ref{Pve_Brow}) as $0 \in \alpha^L$ and $1 \notin \alpha^L$. Now the subtuple of $ \beta : =(1:1+c_{1}(\alpha) ,\alpha^R)$ with indices  $ \{0,1\}$ starts with $1$. Hence by Proposition~\ref{sp1rowblock}, we have $ (e_{m}^T \otimes ) M_{\beta} (*) =  e_{m-(1+c_1(\alpha))}^T \otimes I_n$, where $(*)$ denotes any arbitrary matrix assignment. Hence $ (e_{m}^T \otimes I_n) M_{\alpha} (\mathcal{X}) =  e_{m-(1+c_1(\alpha))}^T \otimes X^0$. Since $c_1(\alpha) \geq 2$, we have $1+c_1(\alpha) \geq 3 $. This shows that $M_\alpha (\mathcal{X})$ is not a block penta-diagonal matrix.	
	
	Similarly, if $i_j(\alpha) \geq 2$, for some $j \geq 1$, then it can be shown that $M_\alpha (\mathcal{X})$ is not a block penta-diagonal matrix.
\end{proof}

\subsection{The proof of the Proposition~\ref{not_blkpentaNve}}~ \label{appendix4}~

The proof of Proposition~\ref{not_blkpentaNve} is similar to the proof of  Proposition~\ref{not_blkpentaPve}.

\section{Proof of Proposition~\ref{OF_EGFPR}} \label{OpeFreeAppen}

First, we present some technical results which will be used to prove that EGFPs are operation free pencils.  Recall from (\ref{Pve_Brow}) and (\ref{Nve_Brow}) that the Fiedler matrices $M_j^P$, $j \in \{-m:m-1\}$, satisfies the following: 
\begin{equation}\label{M_potri}
	(e^T_{m-i} \otimes I_n) M^P_{j}=
	\left\{
	\begin{array}{ll}
		e^T_{m-(i+1)} \otimes I_n & \mbox{if } j=i+1, ~ i=0:m-2\\
		(e^T_{m-i} \otimes (-A_i) ) + (e^T_{m-(i-1)} \otimes I_n)  & \mbox{if}~j=i,~ i=1:m-1 \\
		e^T_{m} \otimes (-A_0) & \mbox{if}~j=i = 0\\
		e^T_{m-i} \otimes I_n & \begin{array}{l}
			\mbox{otherwise, i.e., when}\\
			~j \notin \{i,i+1\},~ i=0:m-1,\end{array}
	\end{array}
	\right.
\end{equation}
\begin{equation}\label{M_netri}
	(e^T_{m-i} \otimes I_n) M^P_{-j}=
	\left\{
	\begin{array}{ll}
		(e^T_{m-i} \otimes A_{i+1} ) +  (e^T_{m-(i+1)} \otimes I_n) & \mbox{if }~ j= i+1 ,~ i=0:m-2 \\
		e^T_{1} \otimes A_m & \mbox{if}~j= i+1 ,~ i=m -1\\
		e^T_{m-(i-1)} \otimes I_n & \mbox{if}~j = i~ \mbox{and}~ i=1:m-1 \\
		e^T_{m-i} \otimes I_n & \begin{array}{l}
			\mbox{otherwise, i.e., when}\\
			~j \notin \{i,i+1\},~ i=0:m-1.\end{array}
	\end{array}
	\right.
\end{equation}

 The following result follows from Proposition~\ref{sp1rowblock}.  We denote by $(*)$ any arbitrary matrix assignment.

\begin{remark} \label{lem_ip1_starts}  Let  $\alpha$ be an index tuple containing indices from $ \{0:m-1\}$ such that $\alpha$ satisfies the  SIP.  Let $ 0 \leq s \leq m-1$. Suppose that $s, s+1 \notin \alpha$ or the  subtuple of $\alpha$ with indices $ \{ s , s+1\}$ starts with $ s+1$. Then it follows from Proposition~\ref{sp1rowblock} that $ ( e^T_{ m-s} \otimes I_n ) M_{\alpha} (*) = e^T_{ m-k} \otimes I_n $, where $k = s+1+c_{ s+1} (\alpha)$.
\end{remark}

The following result follows from Proposition~\ref{ms_start}.

\begin{remark} \label{lem_mi_starts} Let  $ \beta$ be an index tuple containing indices from $ \{-m:-1\}$ such that $\beta$ satisfies the  SIP. Let $ 1 \leq s \leq m-1$.  Suppose that $-s, -(s+1) \notin \beta$ or the subtuple of $\beta$ with indices $ \{ -s ,-( s+1) \}$ starts with $ -s$. Then it follows from Proposition~\ref{ms_start} that $ ( e^T_{ m-s} \otimes I_n ) M_{\alpha}  = e^T_{ m-t} \otimes I_n $, where $t =  s-1-c_{ -s} (\alpha)$.
\end{remark}


\begin{proposition} \label{lem_dpos} Let  $\alpha$ be an index tuple containing indices from $ \{0:m-1\}$ such that $\alpha$ satisfies the  SIP. Let $ 0 \leq s \leq m-1$. Suppose that the subtuple of $\alpha$ with indices $\{s+1,s+2\}$ starts with $s+2$. Let $ \alpha = (\beta, s+2,  \gamma) $, where $s+1,s+2 \notin \beta$. Further, suppose that the subtuple of $\gamma$ with indices $\{s,s+1\}$ starts with $s+1$. Then $( e^T_{ m-(s+1)} \otimes I_n ) M_{\alpha} (*)  \neq ( e^T_{ m-s} \otimes I_n ) M_{\gamma} (*)$. Moreover, $( e^T_{ m-(s+1)} \otimes I_n ) M_{\alpha} (*) = e^T_{ m-k} \otimes I_n$ and $( e^T_{ m-s} \otimes I_n ) M_{\gamma}  (*)= e^T_{ m-\ell} \otimes I_n$ with $k > \ell $, where $k = s+2+c_{ s+2} (\alpha)$ and $\ell = s+1+c_{ s+1} (\gamma)$.
\end{proposition} 

\begin{proof} Since the subtuple  of $\alpha$ with indices $\{s+1,s+2\}$ starts with $s+2$, by Remark~\ref{lem_ip1_starts}, we have $( e^T_{ m-(s+1)} \otimes I_n ) M_{\alpha}(*)  = e^T_{ m-k} \otimes I_n$, where $k = s+2+c_{ s+2} (\alpha)$. Similarly, since the subtuple of $\gamma$ with indices $\{s,s+1\}$ starts with $s+1$, by Remark~\ref{lem_ip1_starts}, we have $( e^T_{ m-s} \otimes I_n ) M_{\gamma} (*)  = e^T_{ m-\ell} \otimes I_n$, where $\ell = s+1+c_{ s+1} (\gamma)$. Now let $c_{s+1}(\gamma) = p$, i.e., $(s+1,s+2,\ldots,s+p+1)$ is a subtuple of $\gamma$. Since $\alpha  = (\beta, s+2,  \gamma)$ satisfies the SIP, it follows that $(s+2,\ldots,s+p+1,s+p+2)$ must be a subtuple of $\alpha$. Hence $c_{s+2} (\alpha) \geq p$ and $ k = s+2+c_{ s+2} (\alpha) \geq s+p+2 > s+p+1 = \ell$ which gives the desired result.
\end{proof}

The following result is analogous to Proposition~\ref{lem_dpos}.

\begin{proposition}  \label{lem_dneg} Let  $\alpha$ be an index tuple containing indices from $ \{-m:-1\}$ such that $\alpha$ satisfies the SIP. Suppose that the subtuple of $\alpha$ with indices $\{-s,-(s+1)\}$ starts with $-s$. Let $ \alpha = (\beta, -s,  \gamma) $, where $-s,-(s+1) \notin \beta$. Further, suppose that the subtuple of $\gamma$ with indices $\{-(s+1),-(s+2)\}$ starts with $-(s+1)$. Then $ (e^T_{m-s} \otimes I_n) M_{\alpha} (*) \neq  (e^T_{m-(s+1)} \otimes I_n) M_{\gamma} (*)$. Moreover, $ (e^T_{m-s} \otimes I_n) M_{\alpha} (*) = e^T_{m-k} \otimes I_n  $ and $(e^T_{m-(s+1)} \otimes I_n) M_{\gamma}(*) = e^T_{m-\ell} \otimes I_n  $ with $k< \ell $, where $k = s-1- c_{-s}(\alpha) $ and $\ell  = s- c_{-(s+1)}(\gamma)$.
\end{proposition} 

For $a,b,q \in \mathbb{Z}$, we denote $\{a :_q b \} : = \{ a, a+q, a+2q, \ldots, b\}$. The following facts will be used to prove that EGFPs are operation free pencils.

\begin{remark} \label{rem_pn} Let $1\leq h_1 \leq h_2 < h_2+1 \leq h_3 \leq m-1$. Let $\alpha$ be an index tuple containing indices from $\{-h_1:_{-1} -h_2\} \cup \{ h_2+1 : h_3\}$. Let $\beta$ (resp., $\gamma$) be the subtuple of $\alpha$ with indices $\{-h_1:_{-1} -h_2\}$ (resp., $\{ h_2+1 : h_3\}$). Suppose that $\beta$ and $\gamma$ satisfy the SIP. Then from (\ref{M_potri}) and (\ref{M_netri}) we have the following. 
	\begin{itemize}
		\item  Let $h_1 -1 \leq  j  \leq h_2-1$. If the subtuple of $\alpha$ with indices $\{-j,-(j+1)\}$ starts with $-j$ then the indices $ h_2+1 : h_3$ in $\alpha$ are redundant for evaluating $(e^T_{m-j} \otimes I_n) M^P_{\alpha}$, i.e., $(e^T_{m-j} \otimes I_n) M^P_\alpha  = (e^T_{m-j} \otimes I_n) M^P_\beta $.
		
		\item If the the subtuple of $\alpha$ with indices $\{-h_2, h_2+1\}$ starts with $-h_2$ then similarly as above the indices $ h_2+1 : h_3$ in $\alpha$ are redundant for evaluating $(e^T_{m-h_2} \otimes I_n) M^P_{\alpha}$. On the other hand, if the subtuple of $\alpha$ with indices $\{-h_2, h_2+1\}$ starts with $h_2+1$ then the indices $- h_1 :_{-1} -h_2$ are redundant for evaluating $(e^T_{m- h_2} \otimes I_n) M^P_{\alpha} $, i.e.,  $ (e^T_{m- h_2} \otimes I_n) M^P_{\alpha} = (e^T_{m- h_2} \otimes I_n) M^P_{\gamma} $.
		
		\item Let $h_2 +1 \leq  j  \leq h_3$. If the subtuple of $\alpha$ with indices $\{j,j+1\}$ starts with $j+1$ then the indices $- h_1 :_{-1} -h_2$ are redundant for evaluating $(e^T_{m- j} \otimes I_n) M^P_{\alpha} $, i.e,   $ (e^T_{m- j} \otimes I_n) M^P_{\alpha} = (e^T_{m-j} \otimes I_n) M^P_{\gamma} $.
		
	\end{itemize}
\end{remark}

The following result will be useful for proving that EGFPs are operation free.

\begin{lemma} \label{OF_lem} Let $1\leq h_1 \leq h_2 < h_2+1 \leq h_3 \leq m-1$. Let $\alpha$ be an index tuple containing indices from $\{-h_1:_{-1} -h_2\} \cup \{ h_2+1 : h_3\}$. Define $\beta :=$ the subtuple of $\alpha$ with indices $\{-h_1:_{-1} -h_2\}$ and $\gamma$ be the subtuple of $\alpha$ with indices $\{ h_2+1 : h_3\}$. Suppose that $\beta$ and $\gamma$ satisfies the SIP. Then $ M_{\alpha} (*)$ is operation free.
\end{lemma} 

\begin{proof} We will prove that $M^P_\alpha$ is operation free. The proof depends only on the indices of $\alpha$ and does not depend on the matrix assignments for $\alpha$.  Hence $M_\alpha (*)$ is operation free. 
	
	For proving $M^P_\alpha$ is operation free it is equivalent to show that $(e^T_{m-j} \otimes I_n) M^P_\alpha$ is operation free for all $j =0:m-1$. It follows from (\ref{M_potri}) and (\ref{M_netri}) that $ (e^T_{m-j} \otimes I_n) M^P_\alpha = e^T_{m-j} \otimes I_n$ for
	all $j \in \{ 0: h_1-2 \} \cup \{h_3+1 : m-1\}$. Hence it remains to show that  $(e^T_{m-j} \otimes I_n) M^P_\alpha$ is operation free for all $j = h_1 -1 : h_3$. We proceed as follows.
	
	(a) Let $h_1 -1 \leq  j  \leq h_2-1$. We prove that $(e^T_{m-j} \otimes I_n) M^P_\alpha$ is operation free. 
	
	Case-I: Suppose that $-j,-(j+1) \notin \alpha$ or the subtuple of $\alpha$ with indices $\{-j,-(j+1)\}$ starts with $-j$. Then by Remark~\ref{lem_mi_starts} and Remark~\ref{rem_pn}, we have $ (e^T_{m-j} \otimes I_n) M^P_\alpha  = e^T_{m-k} \otimes I_n$, where $k = j-1- c_{-j}(\alpha) $. Hence $(e^T_{m-j} \otimes I_n) M^P_\alpha$ is operation free.
	
	Case-II: Suppose that the subtuple of $\alpha$ with indices $\{-j,-(j+1)\}$ starts with $-(j+1)$. The following steps show that $ (e^T_{m-j} \otimes I_n) M^P_\alpha$ is operation free.
	
	Step-1: Let $\alpha = ( \delta^j, -(j+1), \alpha^j)$, where $ -j,-(j+1) \notin \delta^j$. Then we have  
	\begin{eqnarray}
		(e^T_{m-j} \otimes I_n) M^P_\alpha & =  &(e^T_{m-j} \otimes I_n) M^P_{(-(j+1), \alpha^j)} \text{ by } (\ref{M_potri}) \text{ and } (\ref{M_netri}) \text{ since }  -j,-(j+1) \notin \delta^j \nonumber \\
		&= & \Big ( (e^T_{m-j} \otimes A_{j+1}) + (e^T_{m-(j+1)} \otimes I_n) \Big ) M^P_{\alpha^j}  \text{ by } (\ref{M_netri})   \nonumber \\
		&= &(e^T_{m-j} \otimes A_{j+1})  M^P_{\alpha^j}  + (e^T_{m-(j+1)} \otimes I_n)  M^P_{\alpha^j}  \label{eqn_OF_lem1}.
	\end{eqnarray}
	Now, since $\alpha =( \delta^j, -(j+1), \alpha^j)$ satisfies the SIP, we have either $-j,-(j+1)  \notin \alpha^j$ or the subtuple of $\alpha^j$ with indices $\{ -j,-(j+1)  \}$ starts with $ -j$. Then 
	by Remark~\ref{lem_mi_starts} and Remark~\ref{rem_pn}, we have $ (e^T_{m-j} \otimes I_n) M^P_{\alpha^j} =  e^T_{m-k_j} \otimes I_n$, where $k_j := j-1- c_{-j}(\alpha^j)$. Hence from (\ref{eqn_OF_lem1}) we have 
	\begin{eqnarray}
		(e^T_{m-j} \otimes I_n) M^P_\alpha  = (e^T_{m-k_j} \otimes A_{j+1})   + (e^T_{m-(j+1)} \otimes I_n)  M^P_{\alpha^j}  \label{eqn2_OF_lem}.
	\end{eqnarray}
	The evaluation of $(e^T_{m-j} \otimes I_n) M^P_\alpha$ is completed if any one of the following cases hold. Otherwise, we move to Step-2.
	\begin{itemize} 
		\item Suppose that $-(j+1), -(j+2) \notin \alpha^j$ or the subtuple of $\alpha^j$ with indices $\{ -(j+1), -(j+2)  \}$ starts with $ -(j+1)$. Then by Remark~\ref{lem_mi_starts} and Remark~\ref{rem_pn} we have $ (e^T_{m-(j+1)} \otimes I_n) M^P_{\alpha^j} =  e^T_{m-k_{j+1}} \otimes I_n$, where $k_{j+1} := j- c_{-(j+1)}(\alpha^j)$. It follows by Proposition~\ref{lem_dneg} that $ k_{j} < k_{j+1}$. Hence it follows from (\ref{eqn2_OF_lem})  that $ (e^T_{m-j} \otimes I_n) M^P_\alpha$ is operation free.
		
		\item Suppose that $j+1 = h_2$.  If the subtuple of $\alpha^j$ with indices $\{ -h_2, h_2+1 \}$ starts with $ - h_2$ then by similar argument as above it follows that $ (e^T_{m-j} \otimes I_n) M^P_\alpha$ is operation free. On the other hand, if the subtuple of subtuple of $\alpha^j$ with indices $\{ -h_2, h_2+1 \}$ starts with $ h_2+1 $ then by Remark~\ref{lem_ip1_starts} and Remark~\ref{rem_pn}, we have  $ (e^T_{m- h_2} \otimes I_n) M^P_{\alpha^j} =  e^T_{m-\ell} \otimes I_n$, where $\ell :=  h_2+1+c_{h_2+1}(\alpha^j)$. Hence it follows from (\ref{eqn2_OF_lem})  that $ (e^T_{m-j} \otimes I_n) M^P_\alpha$ is operation free.
	\end{itemize}
	
	Step-2: Suppose that the subtuple of $\alpha^j$ with indices $\{ -(j+1), -(j+2)  \}$ starts with $ -(j+2)$. Then repeat Step-1 with $\alpha$ replaced by $\alpha^j$  and  $j$ replaced by $j+1$.
	
	Since $ \{ h_1-1 : h_2-1 \}$ contains $ s:=h_2 -h_1+1$ number of indices, we must stop before $s$ numbers of steps. This completes the proof of $ (e^T_{m-j} \otimes I_n) M^P_\alpha$ is operation free.

	(b) Let $ j = h_2$. If the subtuple of $\alpha$ with indices $\{-h_2, h_2+1\}$ starts with $h_2+1$ then  by Remark~\ref{lem_ip1_starts} and Remark~\ref{rem_pn}, we have  $ (e^T_{m- h_2} \otimes I_n) M^P_{\alpha} =  e^T_{m-\ell} \otimes I_n$, where $\ell :=  h_2+1+c_{h_2+1}(\alpha)$. Hence  $ (e^T_{m-{h_2}} \otimes I_n) M^P_\alpha$ is operation free. On the other hand if the the subtuple of $\alpha$ with indices $\{-h_2, h_2+1\}$ starts with $-h_2$ then by Remark~\ref{lem_mi_starts} and Remark~\ref{rem_pn}, we have $ (e^T_{m-{h_2}} \otimes I_n) M^P_{\alpha} =  e^T_{m- \ell} \otimes I_n$, where $\ell := h_2-1- c_{-h_2}(\alpha)$. Hence  $ (e^T_{m-{h_2}} \otimes I_n) M^P_\alpha$ is operation free.

	(c) Let $h_2+1 \leq  j  \leq h_3$. We prove that $(e^T_{m-j} \otimes I_n) M^P_\alpha$ is operation free.
	
	Case-I: Suppose that $j,j+1 \notin \alpha$ or the subtuple of $\alpha$ with indices $\{j,j+1\}$ starts with $j+1$. Then by Remark~\ref{lem_ip1_starts} and Remark~\ref{rem_pn}, we have $ (e^T_{m-j} \otimes I_n) M_\alpha = (e^T_{m-j} \otimes I_n) M_\gamma = e^T_{m-k} \otimes I_n$ which is operation free, where $k = j+1+c_{j+1}(\gamma) $.
	
	Case-II: Suppose that the subtuple of $\alpha$ with indices $\{j,j+1\}$ starts with $j$. Then the following steps show that $ (e^T_{m-j} \otimes I_n) M_\alpha$ is operation free.
	
	Step-1: Let $\alpha = ( \xi^j, j, \alpha^j)$, where $ j,j+1 \notin \xi^j$. Then we have
	\begin{eqnarray}
		(e^T_{m-j} \otimes I_n) M^P_\alpha & = & (e^T_{m-j} \otimes I_n) M^P_{(j, \alpha^j)}  \text{ by } (\ref{M_potri}) \text{ and } (\ref{M_netri}) \text{ since }  j,j+1 \notin \xi^j \nonumber \\
		& =& \Big ( (e^T_{m-j} \otimes (-A_{j})) + (e^T_{m-(j-1)} \otimes I_n) \Big ) M^P_{\alpha^j}  \text{ by } (\ref{M_potri}) \nonumber \\
		& = &  (e^T_{m-j} \otimes (-A_{j})) M^P_{\alpha^j} + (e^T_{m-(j-1)} \otimes I_n) M^P_{\alpha^j}  \label{eqn3_OF_lem}
	\end{eqnarray}
	Since $\alpha = ( \xi^j, j, \alpha^j)$ satisfies the SIP, we have either $j, j+1 \notin \alpha^j$ or the subtuple of $\alpha^j$ with indices $\{ j,j+1  \}$ starts with $ j+1$. Hence by Remark~\ref{lem_ip1_starts} and Remark~\ref{rem_pn}, we have $ (e^T_{m-j} \otimes I_n) M^P_{\alpha^j} =  e^T_{m-k_j} \otimes I_n$, where $k_j := j+1+ c_{j+1}(\alpha^j)$. Hence from (\ref{eqn3_OF_lem}) we have 
	\begin{equation}\label{eqn3_OF_lem2}
		(e^T_{m-j} \otimes I_n) M^P_\alpha =  (e^T_{m-k_j} \otimes (-A_{j})) + (e^T_{m-(j-1)} \otimes I_n) M^P_{\alpha^j}. 
	\end{equation}
	The evaluation of $(e^T_{m-j} \otimes I_n) M^P_\alpha$ is completed if any one of the following cases hold. Otherwise, we move to Step-2.
	\begin{itemize} 
		\item Suppose that $j, j-1 \notin \alpha^j$ or the subtuple of $\alpha^j$ with indices $\{ j, j-1  \}$ starts with $ j$. Then by Remark~\ref{lem_ip1_starts} and Remark~\ref{rem_pn}, we have $ (e^T_{m-(j-1)} \otimes I_n) M^P_{\alpha^j} =  e^T_{m-k_{j-1}} \otimes I_n$, where $k_{j-1} := j+ c_{j}(\alpha^j)$. By Proposition~\ref{lem_dpos} we have  $ k_{j} > k_{j-1}$. Hence  it follows from (\ref{eqn3_OF_lem2}) that  $ (e^T_{m-j} \otimes I_n) M^P_\alpha$ is operation free.
		
		\item Suppose that $j-1 = h_2$. If the subtuple of $\alpha^j$ with indices $\{ -h_2, h_2+1 \}$ starts with $ h_2+1$ then by similar arguments as above it follows that it follows that  $ (e^T_{m-j} \otimes I_n) M^P_\alpha$ is operation free. On the other hand, if the subtuple of subtuple of $\alpha^j$ with indices $\{ -h_2, h_2+1 \}$ starts with $- h_2 $ then by Remark~\ref{lem_mi_starts} and Remark~\ref{rem_pn}, we have $ (e^T_{m- h_2} \otimes I_n) M^P_{\alpha^j} =  e^T_{m-\ell} \otimes I_n$, where $\ell :=  h_2-1+c_{-h_2}(\alpha^j)$. Hence  it follows from (\ref{eqn3_OF_lem2}) that $ (e^T_{m-j} \otimes I_n) M^P_\alpha$ is operation free.
	\end{itemize}
	
	Step-2: Suppose that the subtuple of $\alpha^j$ with indices $\{ j, j-1  \}$ starts with $ j-1$. Then repeat Step-1 with $\alpha$ replaced by $\alpha^j$  and  $j$ replaced by $j-1$.
	
	Since $ \{ h_1+1 : h_3 \}$ contains $ s:=h_3 -h_2$ number of indices, we must stop before $ s$ numbers of steps. This prove that $ (e^T_{m-j} \otimes I_n) M^P_\alpha$ is operation free.  	 
\end{proof}

\subsection{Proof of Proposition~\ref{OF_EGFPR}}~
\begin{proof} Recall that $ (\sig, \omega)$ is a permutation of $ \{0:m\}$ and $\tau = - \omega$. Since  $ M_{-0}(Z)$ is a block diagonal matrix and  $ -0,-1 \notin \tau_1 \cup \tau_2$ (i.e., $-0$ and $-1$ do not repeat),  without loss of generality we assume that $0 \in \sig$. Similarly, since  $ M_{m}(Z)$ is a block diagonal matrix and  $ m-1,m \notin \sig_j$, $j=1,2$, (i.e., $m-1$ and $m$ do not repeat),  without loss of generality we assume that $m \in - \tau$.

	Now, since $ (\sig, -\tau)$ is a permutation of $ \{0:m\}$ with $0\in \sig$ and $m \in -\tau$, there exist $ 0 \leq h_1 < h_2 < \cdots < h_{k-1}< h_k \leq m-1$ (with odd $k$) such that $\sig$ is a permutation of $ \{ 0:h_1 \} \cup \{ h_2+1 : h_3\} \cup  \cdots \cup \{h_{k-1} +1:h_k \}$ and $-\tau$ is a permutation of $ \{h_1+1:h_2 \} \cup \{h_3+1 : h_4\} \cup \cdots \cup \{h_k +1:m \}$. This implies that  $\sig_1$ and $\sig_2$ contain indices from $ \{ 0:h_1-1 \} \cup \{ h_2+1 : h_3-1\} \cup  \cdots \cup \{h_{k-1} +1:h_k-1 \}$ since $(\sig_1,\sig,\sig_2)$ satisfies the SIP, and  $-\tau_1$ and $-\tau_2$ contain indices from $ \{h_1+2:h_2 \} \cup \{h_3+2 : h_4\} \cup \cdots \cup \{h_k +2:m \}$ since $(\tau_1,\tau,\tau_2)$ satisfies the SIP.

	Set $\alpha : = (\tau_1,\sig_1,\sig,\sig_2,\tau_2 )$. We now show that $ M_{\alpha} (*)$ is operation free (i.e., $L_0$, given in Proposition~\ref{OF_EGFPR}, is operation free). From the above paragraph, it is clear that $\alpha$ contains indices from $ \{0:h_1\} \cup \{ -(h_1+2):_{-1} -h_2, h_2+1 : h_3 \} \cup \{ -(h_3+2) :_{-1} -h_4, h_4+1 :h_5 \} \cup \cdots \cup  \{-(h_{k-2}+2) :_{-1} -h_{k-1}, h_{k-1} +1:h_k\} \cup \{ h_k +2:m \}$ $=: H_1 \cup H_3 \cup H_5 \cup \cdots \cup H_k \cup H_{k+1}$. Note that, for indices $i \in H_s$ and $j  \in H_t$, we have $ ||i| - |j|| \geq 2$ if $s \neq t$.  Hence $M_{\alpha} (*)$ is operation free if $M_{\alpha^{H_j}} (*)$ is operation free for all $j = 1,3,\ldots, k+1$, where $ \alpha^{H^j}$ is the subtuple of $\alpha$ with indices from $H_j$. Since $\tau_1 \cup \tau_2$ (resp., $(\sig_1,\sig,\sig_2)$) does not contains any index of $H_1$ (resp., $H_{k+1}$), we have $ M_{\alpha^{H_1}} (*)$ (resp., $ M_{\alpha^{H_{k+1}}} (*)$) is operation free. Further, by Lemma~\ref{OF_lem}, we have $M_{\alpha^{H_j}} (*)$ is operation free for all $j = 3,5,\ldots, k$. Hence $M_{\alpha} (*)$ is operation free, i.e., $L_0$ is operation free. 
	
	Similar proof for $M_{\beta} (*)$ is operation free, where $\beta : = (\tau_1,\sig_1,\tau,\sig_2,\tau_2 )$, i.e., $L_1$, given in Proposition~\ref{OF_EGFPR}, is operation free.  Hence the EGFP $ L(\lam)$ is operation free. 
\end{proof}

\end{appendix}

\end{document}